\documentclass[smallcondensed]{svjour3} 
\usepackage{amssymb,amsfonts,amsmath,mathrsfs,amsbsy}
\usepackage{graphicx,subfigure}
\usepackage{bm}
\usepackage[margin=1in,papersize={8.5in,11in}]{geometry}

\usepackage{xcolor}
\usepackage[colorlinks=true, linkcolor=blue, urlcolor=blue, citecolor = blue]{hyperref}

\renewcommand{\bf}{\bm}
\renewcommand{\top}{{\rm T}}

\title{
Adaptive integration of nonlinear evolution equations on tensor manifolds
\thanks{
This research was supported by the U.S. Army Research Office 
grant W911NF1810309, and by the U.S. Air Force 
Office of Scientific Research grant FA9550-20-1-0174.
}}

\author{Abram Rodgers
\and Alec Dektor
\and Daniele Venturi}

\institute{A. Rodgers \at
Department of Applied Mathematics,
University of California Santa Cruz, Santa Cruz (CA) 95064\\
\email{akrodger@ucsc.edu}
\and
A. Dektor \at 
Department of Applied Mathematics,
University of California Santa Cruz, Santa Cruz (CA) 95064\\
\email{adektor@ucsc.edu}
\and
D. Venturi \at
Department of Applied Mathematics,
University of California Santa Cruz, Santa Cruz (CA) 95064\\
\email{venturi@ucsc.edu}
}
\date{}
\begin{document}

\maketitle

\begin{abstract}
We develop new adaptive algorithms for temporal 
integration of nonlinear evolution 
equations on tensor manifolds.
These algorithms, which we call step-truncation 
methods, are based on performing one time 
step with a conventional time-stepping 
scheme, followed by a truncation operation 
onto a tensor manifold. 
By selecting the rank of the tensor manifold 
adaptively to satisfy stability and accuracy 
requirements, 
we prove convergence of a wide range 
of step-truncation methods, including 
explicit one-step and multi-step methods.
These methods are very easy to implement 
as they rely only on arithmetic operations 
between tensors, which can be performed by 
efficient and scalable parallel algorithms.
Adaptive step-truncation methods 
can be used to compute numerical solutions 
of high-dimensional PDEs, which, have become 
central to many new areas of application 
such optimal mass transport, random dynamical 
systems, and mean field optimal control.
Numerical applications are presented and 
discussed for a Fokker-Planck equation
with spatially dependent drift on a flat 
torus of dimension two and four.
\keywords{
High-dimensional PDEs \and
Low-rank tensor manifolds \and
Dynamical tensor approximation\and 
Tensor train\and 
Hierarchical Tucker format.}
\end{abstract}

% AMS CLASSIFICATION
%\begin{AMS}
%15A69, %Multilinear algebra, tensor calculus
%65M22, % Numerical solution of discretized equations for initial value and initial-boundary value problems involving PDEs
%35Q84, %Fokker-Planck equations
%65M15. %Error bounds for initial value and initial- boundary value problems involving PDEs
%\end{AMS}

\section{Introduction}
\label{sec:intro}

Consider the initial value problem 
\begin{align}
\frac{\partial f({\bf x},t) }{\partial t} = 
{\cal N}\left(f({\bf x},t),{\bf x}\right), \qquad 
f({\bf x},0) = f_0({\bf x}),
\label{nonlinear-ibvp} 
\end{align}
where $f:  \Omega \times [0,T] \to\mathbb{R}$ 
is a $d$-dimensional (time-dependent) scalar field 
defined on the domain $\Omega\subseteq \mathbb{R}^d$ 
($d\geq 2$), and $\cal N$ is a nonlinear operator which 
may depend on the variables 
${\bf x}=(x_1,\ldots,x_d)$ and may 
incorporate boundary conditions.
By discretizing \eqref{nonlinear-ibvp} in $\Omega$, 
e.g., by finite differences, finite 
elements, or pseudo-spectral methods, 
we obtain the system of ordinary differential equations 
\begin{equation}
\label{mol-ode}
\frac{d{\bf f}(t)}{dt} = {\bf N}({\bf f}(t)), 
\qquad {\bf f}(0)={\bf f}_0.
\end{equation}
Here, ${\bf f}:[0,T]\rightarrow
{\mathbb R}^{n_1\times n_2 \times \cdots \times n_d}$ 
is a multi-dimensional array of real numbers (the solution tensor), 
and $\bf N$ is a tensor-valued nonlinear map  
(the discrete form of $\mathcal{N}$ corresponding 
to the chosen spatial discretization). 
The number of degrees of freedom associated 
with the solution ${\bf f}(t)$ to the Cauchy problem \eqref{mol-ode} 
is $N_{\text{dof}}=n_1 \cdot n_2 \cdots n_d$ at each 
time $t\geq 0$, which can 
be extremely large even for 
small $d$. For instance, the solution 
to the Boltzmann-BGK equation
on a $6$-dimensional flat torus \cite{BoltzmannBGK2020}  
with $128$ points in each variable $x_i$ ($i=1,\ldots,6)$ 
yields $N_{\text{dof}}=128^6=4398046511104$ degrees 
of freedom at each time $t$.  

In order to reduce the number of degrees of freedom 
in the solution tensor ${\bf f}(t)$, we 
seek a representation of the solution  
on a low-rank tensor format
{
\cite{Kolda,Cho2016,dolgov2012fast,chertkov2021solution}} 
for all $t \geq 0$. 
To complement the low-rank structure of ${\bf f}(t)$, 
we also represent the operator $\bf N$ in a 
compatible low-rank format, allowing for an efficient 
computation of ${\bf N}({\bf f}(t))$ at each time $t$ 
in \eqref{mol-ode}. 
One method for the temporal integration 
of \eqref{mol-ode} on a smooth tensor manifold with 
constant rank \cite{uschmajew2013geometry,Holtz_2012} 
is dynamic tensor approximation  
\cite{lubich2013dynamical,koch2010dynamtucker,dektor2020dynamically,Alec2020}. 
This method keeps the solution ${\bf f}(t)$ on the
tensor manifold for all $t \geq 0$ by integrating 
the projection of \eqref{mol-ode} 
onto the tangent space of the manifold 
foward in time.
While such an approach has proven effective, 
it also has inherent computational drawbacks. 
Most notably, the system of evolution equations 
arising from the projection of \eqref{mol-ode}
onto the tangent space of the tensor manifold 
contains inverse auto-correlation matrices 
which may become ill-conditioned as time 
integration proceeds. This problem was addressed 
in \cite{Lubich2014,Lubich_2015} by using operator 
splitting time integration methods 
(see also \cite{unc_int}). 
\begin{figure}[t]
\begin{center}
\includegraphics[scale=0.35]{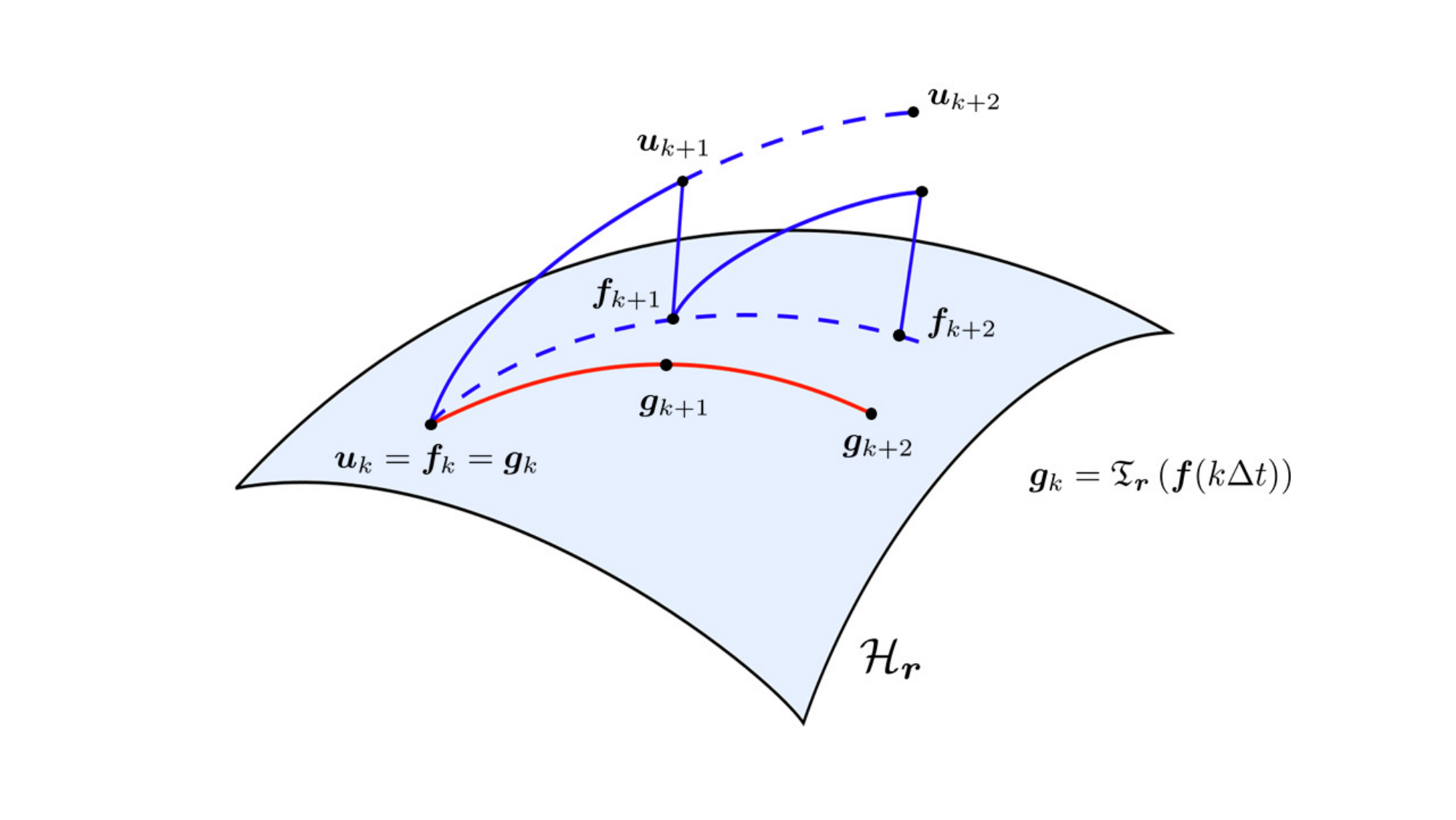}
\end{center}
\caption{
A graphical representation of a step-truncation method
to integrate equation \eqref{mol-ode} on a 
Hierarchical Tucker tensor manifold ${\cal H}_{\bf r}$
with multilinear rank $\bf r$.
The ${\bf u}_k$
denote time snapshots of the numerical approximation 
to \eqref{mol-ode} obtained using a conventional time-stepping scheme, 
the ${\bf g}_k$ are time snapshots of 
the exact solution ${\bf f}(k\Delta t)$ 
projected onto ${\cal H}_{\bf r}$, and 
the ${\bf f}_k$ are time snapshots of 
the step-truncation solution on ${\cal H}_{\bf r}$.
}
\end{figure}
A different class of algorithms to integrate 
the Cauchy problem \eqref{mol-ode} on a 
low-rank tensor manifold was recently proposed 
in \cite{rodgers2020stability,kieri2019projection,venturi2018,VenturiSpectral}.
These algorithms are based on integrating the solution 
${\bf f}(t)$ off the tensor manifold 
for a short time using any conventional time-stepping 
scheme, and then mapping it back onto the manifold 
using a tensor truncation operation. We will refer to 
these methods as {\em step-truncation} methods. 
To describe these methods further, let us 
discretize the ODE \eqref{mol-ode} in time with a conventional 
one-step scheme on an evenly-spaced temporal grid 
\begin{equation}
{{\bf u}}_{k+1} = {\bf u}_k+\Delta t 
{\bf \Phi}({\bf N},{\bf u}_k,\Delta t),
\label{AB2}
\end{equation}
where ${\bf u}_{k}$ denotes an approximation of 
${\bf f}(k\Delta t)$ for $k=0,1,\ldots$, 
and $\bf \Phi$ is an increment function.
To obtain a step-truncation scheme, we simply apply 
a nonlinear projection (truncation operator), denoted 
by $\mathfrak{T}_{\bf r}(\cdot)$, onto a tensor manifold 
with multilinear rank $\bf r$ 
\cite{grasedyck2010hierarchical,grasedyck2018distributed,kressner2014algorithm,parr_tensor,BoltzmannBGK2020} 
to the scheme \eqref{AB2}. This yields 
\begin{equation}
{\bf f}_{k+1} = \mathfrak{T}_{\bf r}\left({\bf f}_k+\Delta t 
{\bf \Phi}\left({\bf N},{\bf f}_k,\Delta t\right)\right),
\label{AB2-ST}
\end{equation} 
where ${\bf f}_{k}$ here denotes an approximation of 
$\mathfrak{T}_{\bf r}({\bf f}(k\Delta t))$ for $k=0,1,\ldots$. 
The need for tensor rank-reduction when iterating 
\eqref{AB2} can be easily understood by noting that 
tensor operations such as the application of 
an operator to a tensor and the addition between 
two tensors naturally increase tensor rank
\cite{kressner2014algorithm}. 
Hence, iterating \eqref{AB2} with no rank 
reduction can yield a fast increase 
in tensor rank, which, in turn, can tax 
computational resources significantly.
Step-truncation algorithms of the form \eqref{AB2-ST} 
were subject to a thorough error analysis in 
\cite{kieri2019projection}, where  convergence 
results were obtained in the context of 
fixed-rank tensor integrators, i.e., integrators in which 
the tensor rank $\bf r$ in \eqref{AB2-ST} is kept 
constant at each time step.

In this paper, we develop {\em adaptive} step-truncation 
algorithms in which the tensor rank $\bf r$ is selected 
at each time step based on desired accuracy 
and stability constraints.
These methods are very simple to implement 
as they rely only on arithmetic operations
between tensors, which can be performed
by efficient and scalable parallel algorithms 
\cite{daas2020parallel,AuBaKo16,grasedyck2018distributed}.

The paper is organized as follows.
In section \ref{sec:varying-rank-integrators}, we 
review low-rank integration techniques 
for time-dependent tensors, including dynamic 
approximation and step-truncation methods.
In section \ref{sec:lte-svd-st}, we develop
a new criterion for tensor rank adaptivity 
based on local error estimates.
In section \ref{sec:global-error},
we prove convergence 
of a wide range rank-adaptive step-truncation
algorithms, including one-step methods 
of order 1 and 2, and multi-step methods of 
arbitrary order.
In section \ref{sec:geom-interpreation}
we establish a connection between rank-adaptive
step-truncation methods and rank-adaptive dynamical
tensor approximation.
{
%\color{blue}
In section \ref{sec:numerical-apps}
we present and discuss numerical applications of the proposed
algorithms.
In particular, we study a prototype problem with 
rapidly varying rank and a Fokker-Planck equation
with spatially dependent drift on 
a flat torus of dimension two and four.
}

\section{Low-rank integration of time-dependent tensors}
\label{sec:varying-rank-integrators}

Denote by ${\cal H}_{\bf r}\subseteq 
\mathbb{R}^{n_1\times n_2\times\cdots\times n_d}$ 
the manifold of hierarchical Tucker tensors with 
multilinear rank ${\bf r}= \{r_{\alpha}\}$ corresponding to 
a prescribed dimension tree $\alpha \in {\cal T}_d$
\cite{uschmajew2013geometry}.
\begin{remark}
\label{remark:all-tensors-have-rank}
Every tensor ${\bf f}\in\mathbb{R}^{n_1\times n_2\times\cdots
\times n_d}$ has an exact hierarchical Tucker (HT) 
decomposition \cite{grasedyck2010hierarchical}. 
Thus, if $\bf f$ is not the zero tensor, then 
$\bf f$ belongs to a manifold ${\cal H}_{\bf r}$ 
for some $\bf r$. 
\end{remark}
We begin by introducing three maps which 
are fundamental to the analysis of low-rank tensor 
integration. 
%Each projection can be realized as a special 
%case of a more general and well-studied nonlinear map. 
First, we define the nonlinear map 
\begin{equation}
\label{ht-best}
\begin{aligned}
{\mathfrak T}_{\bf r}^{\text{best}} : \, &\mathbb{R}^{n_1\times n_2\times\cdots
\times n_d} \to \overline{\cal H}_{\bf r},\\
&\qquad {\bf f}\to {\mathfrak T}_{\bf r}^{\text{best}}({\bf f}) = 
\underset{{\bf h}\in \overline{\cal H}_{\bf r}}{\text{argmin}}\ 
\left\|{\bf f} - {\bf h}\right\|_2.
\end{aligned}
\end{equation}
Here, $\overline{\cal H}_{\bf r}$ denotes the closure 
of the tensor manifold ${\cal H}_{\bf r}$ and contains all tensors of 
multilinear rank smaller than or equal to $\bm r$ \cite{uschmajew2013geometry}. 
The map \eqref{ht-best} provides the optimal rank-$\bf r$ approximation of a tensor ${\bf f} \in \mathbb{R}^{n_1\times n_2\times\cdots\times n_d}$. 
%This map is known as 
%``closest point function'' or ``metric projection'' in more general settings.
%
The second map, known as high-order singular value decomposition 
(HOSVD) \cite{grasedyck2018distributed}, is 
defined as a composition of linear maps
obtained from a sequence of singular value 
decompositions of appropriate matricizations of the tensor
$\bf f$ . Such map can be written explicitely as 
\begin{equation}
\begin{aligned}
{\mathfrak T}_{\bf r}^{\text{SVD}} : \, &\mathbb{R}^{n_1\times n_2\times\cdots
\times n_d} \to \overline{\cal H}_{\bf r} , \\
&\bf f\to{\mathfrak T}_{\bf r}^{\text{SVD}}({\bf f}) =
\prod_{\alpha \in{\cal T}_d^p}{\bf P}_{\alpha}
\cdots
\prod_{\alpha \in{\cal T}_d^1}{\bf P}_{\alpha}{\bf f},
\label{ht-SVD}
\end{aligned}
\end{equation}
where ${\cal T}_d^1\dots{\cal T}_d^p$ are
the layers of the dimension tree ${\cal T}_d$. 
The map \eqref{ht-SVD} provides a quasi-optimal rank-$\bf r$ approximation 
of the tensor ${\bf f} \in\mathbb{R}^{n_1\times n_2\times\cdots
\times n_d}$, and is related to the optimal rank-$\bf r$ truncation 
by the inequalities \cite{grasedyck2010hierarchical}
\begin{align}
\label{svd-trunc-close}
\left\|{\bf f}- {\mathfrak T}_{\bf r}^{\text{best}}({\bf f})\right\|_2\leq
\left\|{\bf f}- {\mathfrak T}_{\bf r}^{\text{SVD}}({\bf f})\right\|_2
\leq
\sqrt{2d-3}
\left\|{\bf f}- {\mathfrak T}_{\bf r}^{\text{best}}({\bf f})\right\|_2.
\end{align}
When combined with linear multistep integration schemes, 
the sub-optimal approximation \eqref{ht-SVD} 
has proven to yield stable step-truncation methods 
\cite{rodgers2020stability}.
The third map we define is an orthogonal projection onto 
the tangent space $T_{\bf f}{\cal H}_{\bf r}$ of ${\cal H}_{\bf r}$ 
at the point $\bf f$. 
This projection is defined by the minimization problem 
\begin{equation}
\label{ht-best-tan}
\begin{aligned}
{\cal P}_{\bf f}  : \, &\mathbb{R}^{n_1\times n_2\times
\cdots\times n_d} \to T_{\bf f}{\cal H}_{\bf r}, \\
&{\bf v}\to {\cal P}_{\bf f} {\bf v} = \underset{{\bf h}\in T_{\bf f}{\cal H}_{\bf r}}{\text{argmin}}\ 
\left\|{\bf v} - {\bf h}\right\|_2,
\end{aligned}
\end{equation}
which is a linear function of 
${\bf v}$ ($\bf v$ is the solution
to a linearly constrained least squares problem).

With these three maps defined, hereafter we 
describe two methods for integrating \eqref{mol-ode} on 
the manifold ${\cal H}_{\bm r}$. 
Before doing so, let us discretize the 
temporal domain of interest $[0, T]$ into $N+1$
evenly-spaced\footnote{In order 
to streamline our presentation, we will develop our 
theory using evenly-spaced temporal grids.  
{
%\color{red}
A similar} theory can be developed for grids with 
variable time step size.}
time instants,
\begin{equation}
\label{time_discretization}
t_i = i \Delta t, \qquad 
\Delta t = \frac{T}{N}, \qquad i=0,1,\ldots,N,
\end{equation} 
and let
\begin{equation}
\label{discrete-time-ode}
{\bf u}_{k+1} = {\bf u}_k+\Delta t 
{\bf \Phi}({\bf N},{\bf u}_k,\Delta t) 
\end{equation}
be a convergent one-step scheme\footnote{As is well known,
the scheme \eqref{discrete-time-ode} 
includes all explicit Runge-Kutta methods \cite{HairerErnst1993SODE}. For example, 
the Heun method (explicit RK2) 
takes the form
\eqref{discrete-time-ode} with 
\begin{equation}
\nonumber
{\bf \Phi}({\bf N},{\bf f}_k,\Delta t) =
\frac{1}{2}\left[{\bf N}({\bf f}_k)+{\bf N}\left({\bf f}_k+\Delta t {\bf N}({\bf f}_k)\right)\right].
\end{equation}
Explicit linear multistep methods can be expressed 
in a similar form by replacing the argument ${\bf f}_k$
with an array $\{{\bf f}_k,{\bf f}_{k-1},\ldots\}$ 
(see section \ref{subsec:adaptive-adams}).}
approximating the solution to the initial value problem
\eqref{mol-ode}. In \eqref{discrete-time-ode}, ${\bf u}_k$ denotes the numerical solution to \eqref{mol-ode} at
time instant $t_k$.

\subsection{Best tangent space projection (B-TSP) method}

The first method we present maps the initial condition 
${\bf f}_0$ onto the manifold ${\cal H}_{\bm r}$ 
using either \eqref{ht-best} or \eqref{ht-SVD} and then 
utilizes the orthogonal projection 
\eqref{ht-best-tan} to project ${\bf N}({\bf f}(t))$ 
onto the tangent space $T_{\bf f}{\cal H}_{\bf r}$ at 
each time. We write this method as (see 
\cite{lubich2013dynamical})
\begin{equation}
\frac{{ d}{\bf  w}}{{d} t} 
={\cal P}_{{\bf w}}{\bf N}({\bf  w}),  
\qquad {\bf w}(0) = {\mathfrak T}_{\bf r}({\bf f}_0),
\label{tan-proj-method}
\end{equation}
where ${\mathfrak T}_{\bf r}$ is either the mapping in 
\eqref{ht-best} or \eqref{ht-SVD}.
Discretizing \eqref{tan-proj-method} 
with a one-step method \eqref{discrete-time-ode} 
yields the fully discrete scheme 
\begin{equation}
 \label{discrete-proj-dynam}
 {\bf w}_{k+1} =
 {\bf w}_k + 
 \Delta t{\bf \Phi} \left(
 {\cal P}_{{\bf w}}{\bf N},
     {\bf w}_k,\Delta t\right) , 
     \qquad {\bf w}(0) = {\mathfrak T}_{\bf r}({\bf f}_0).
\end{equation}
While the scheme \eqref{discrete-proj-dynam} 
has proven effective, {
%\color{blue}
explicitly} computing the orthogonal projection 
${\cal P}_{{\bf w}}{\bf N}({\bf  w})$ comes with computational drawbacks. 
Most notably, inverse auto-correlation matrices of tensor modes 
appear in the projection \cite{lubich2013dynamical} (see also
\cite{Lubich2014,koch2007dynamical,koch2010dynamtucker}). 
If the tensor solution is comprised 
of small singular values, then the auto-correlation matrices are 
ill-conditioned. It has been shown in \cite{Lubich2014} 
that this phenomenon is due to the curvature of the 
tensor manifold ${\cal H}_{\bf r}$ being inversely 
proportional to the smallest singular value 
present in the tensor solution.
Thus, special care is required when choosing 
an integration scheme for \eqref{tan-proj-method}.
Operator splitting methods \cite{Lubich2014,Lubich_2015} and unconventional 
integration schemes \cite{unc_int} have been introduced to integrate 
\eqref{tan-proj-method} when the tensor solution is 
comprised of small singular values.
{%\color{blue}
It has also been shown in \cite{kieri2019projection}
that by using an extrinsic representation, the artificial 
stiffness due to the tensor manifold curvarture 
can be avoided.
}
Since this method comes from a 
minimization principle over the tensor 
manifold tangent space, we refer to it as 
the best tangent space projection (B-TSP) method.

\subsection{Step-truncation methods (B-ST, SVD-ST)}
The second method we present allows the  
solution to leave the tensor manifold 
${{\cal H}}_{\bf r}$, and then maps 
it back onto the manifold at each time step.
Applying either \eqref{ht-best} or \eqref{ht-SVD}
to the right hand side of \eqref{discrete-time-ode}  
results in a {\em step-truncation} method
\begin{align}
 \label{one-step-trunc-best}
 {\bf f}_{k+1} =&
 {\mathfrak T}_{{\bf r}}^{\text{best}}\left({\bf f}_k + 
 \Delta t{\bf \Phi}({\bf N},{\bf f}_k,\Delta t)\right) ,\\
 \label{one-step-trunc-svd}
 {\bf f}_{k+1} =&
 {\mathfrak T}_{{\bf r}}^{\text{SVD}}\left({\bf f}_k + 
 \Delta t{\bf \Phi}({\bf N},{\bf f}_k,\Delta t)\right),
\end{align}
which is a low-rank tensor approximation 
to \eqref{mol-ode}. We will refer to 
\eqref{one-step-trunc-best} as
the fixed-rank best 
step-truncation (B-ST) method and 
to \eqref{one-step-trunc-svd} 
as the fixed-rank SVD step-truncation 
(SVD-ST) method. This definition 
emphasizes that the multivatiate tensor 
rank ${\bf r}$ does not change with time.
The schemes \eqref{one-step-trunc-best} and 
\eqref{one-step-trunc-svd} were studied extensively 
in \cite{kieri2019projection}. One of the main 
findings is that a low-rank approximability 
condition is required in order to obtain error 
estimates for the low-rank tensor approximation 
to \eqref{mol-ode}. 
The low-rank approximability condition can be written as 
\begin{equation}
	\label{low-rank-approx}
	\left | \left |
		({\bf I} - {\cal P}_{ {\bf \bar f}}) N({ {\bf \bar f}})
	\right | \right |_2
	\leq E ,\qquad E>0,
\end{equation}
for all ${ {\bf \bar f}}\in {\cal H}_{\bf r}$ 
in a suitable neighbourhood of the exact solution. 
Under this assumption, it can be shown that 
a one-step integration scheme with arbitrary order 
increment function $\bf \Phi$ applied to 
\eqref{one-step-trunc-best} or \eqref{one-step-trunc-svd} 
results in an approximation to \eqref{mol-ode}
with error dominated by $E$.
As an alternative to the fixed-rank schemes 
\eqref{one-step-trunc-best}-\eqref{one-step-trunc-svd} 
combined with the low-rank approximability assumption \eqref{low-rank-approx}, we propose the following {\em rank-adaptive} 
step-truncation schemes 
\begin{align}
 \label{one-step-trunc-best_adapt_rank}
 {\bf f}_{k+1} =&
 {\mathfrak T}_{{\bf r}_{k}}^{\text{best}}\left({\bf f}_k + 
 \Delta t{\bf \Phi}({\bf N},{\bf f}_k,\Delta t)\right) ,\\
 \label{one-step-trunc-svd_adapt_rank}
 {\bf f}_{k+1} =&
 {\mathfrak T}_{{\bf r}_{k}}^{\text{SVD}}\left({\bf f}_k + 
 \Delta t{\bf \Phi}({\bf N},{\bf f}_k,\Delta t)\right).
\end{align}
The selection of a new rank ${\bf r}_k$ at each time step 
allows us to obtain convergence results for step-truncation schemes 
without assuming \eqref{low-rank-approx}. 
We will refer to the schemes \eqref{one-step-trunc-best_adapt_rank} and \eqref{one-step-trunc-svd_adapt_rank} as rank-adaptive B-ST and rank-adaptive 
SVD-ST, respectively.

\section{Consistency of step-truncation methods}
\label{sec:lte-svd-st}
In this section, we prove a number of consistency 
results for step-truncation methods. In particular, 
we show that the fixed-rank step-truncation method 
\eqref{one-step-trunc-best} is consistent with 
the B-TSP method \eqref{discrete-proj-dynam}, 
and the rank-adaptive step-truncation methods 
\eqref{one-step-trunc-best_adapt_rank}-\eqref{one-step-trunc-svd_adapt_rank} are consistent 
with the  fully discrete system \eqref{discrete-time-ode} 
(provided the truncation ranks are chosen 
to satisfy a suitable criterion).
Our analysis begins with stating a few known results for 
the truncation operator ${\mathfrak T}_{\bf r}^{\text{best}}$.
Consider the formal power series expansion of 
${\mathfrak T}_{\bf r}^{\text{best}}$ around $\bf f\in {\cal H}_{\bf r}$
\begin{equation}
{\mathfrak T}_{\bf r}^{\text{best}}({\bf f}
+\varepsilon {\bf v}) = {\mathfrak T}_{\bf r}^{\text{best}}({\bf f})+ \varepsilon \frac{\partial \mathfrak T_{\bf r}^{\text{best}}({\bf f})}{\partial \bf f}{\bf v}
+\cdots,
\label{perturbationseries}
\end{equation}
where $\partial \mathfrak T_{\bf r}^{\text{best}}({\bf f})/
\partial \bf f$ denotes the Jacobian of 
${\mathfrak T}_{\bf r}^{\text{best}}$ at 
${\bf f}$, 
${\bf v}\in\mathbb{R}^{n_1\times n_2\times\cdots
\times n_d}$, and $\varepsilon \in \mathbb{R}$ is small. 
Since ${\bf f}\in {\cal H}_{\bf r}$,
we have that ${\mathfrak T}_{\bf r}^{\text{best}}({\bf f}) 
= {\bf f}$, which allows us to write \eqref{perturbationseries}
as 
\begin{equation}
{\mathfrak T}_{\bf r}^{\text{best}}({\bf f}
+\varepsilon {\bf v}) = {\bf f} + 
\varepsilon \frac{\partial \mathfrak T_{\bf r}^{\text{best}}({\bf f})}{\partial \bf f}{\bf v}
+\cdots .
\label{perturbationseries1}
\end{equation}
In the following Lemma we show that the Jacobian 
$\partial \mathfrak T_{\bf r}^{\text{best}}({\bf f})/\partial \bf f$ 
coincides with the orthogonal projection 
\eqref{ht-best-tan} onto the tangent space 
$T_{\bf f}{\cal H}_{\bf r}$. 
\begin{lemma}[Smoothness of the best truncation operator]
\label{trunc-smooth-thm}
The map ${\mathfrak T}_{\bf r}^{\text{best}}(\cdot)$ is
continuously differentiable on ${\cal H}_{\bf r}$. Moreover, 
\begin{equation}
\nonumber
\frac{\partial {\mathfrak T}_{\bf r}^{\text{best}}({\bf f})}
{\partial{\bf f}}= {\cal P}_{{\bf f}}, \qquad \forall {\bf f} \in {\cal H}_{\bf r},
\end{equation}
where ${\cal P}_{{\bf f}}$ is the orthogonal projection \eqref{ht-best-tan} 
onto the tangent space of ${\cal H}_{\bf r}$ at ${\bf f}$.
\end{lemma}
This result has been proven in 
\cite{lewis2008alternating} and 
\cite{absil2012projection} for 
finite-dimensional manifolds. 
A slightly different proof which holds 
for finite-dimensional manifolds without
boundary is given in 
\cite{marz2012calculus}. In Appendix 
\ref{sec:appendix-bst-analysis} we provide 
an alternative proof which is based primarily on linear 
algebra rather than differential geometry.
With Remark \ref{remark:all-tensors-have-rank} in mind, we can apply Lemma 
\ref{trunc-smooth-thm} to every tensor except the zero 
tensor.
We now use Lemma \ref{trunc-smooth-thm} 
to prove consistency between the fixed-rank B-ST method \eqref{one-step-trunc-best} 
and the B-TSP method  
\eqref{discrete-proj-dynam}.

\begin{proposition}[Consistency of B-ST and B-TSP]
\label{prop:consistency_with_B-TSP}
Let ${\bf \Phi}({\bf N}, {\bf f}, \Delta t)$
denote an order-$p$ increment function defining a 
one-step temporal integration scheme as in \eqref{discrete-time-ode}
and let ${\bf f}_{k}\in{\cal H}_{\bf r}$. We have that 
\begin{equation}
\label{best-trun-taylor}
 {\mathfrak T}_{\bf r}^{\text{best}}({\bf f}_k + 
 \Delta t{\bf \Phi}({\bf N}, {\bf f}_k, \Delta t))
 = {\bf f}_k +
\Delta t
{\cal P}_{{\bf f}_k}
{\bf \Phi}({\bf N}, {\bf f}_k, \Delta t) 
+\cdots ,
\end{equation}
i.e., B-ST is at least order $1$ consistent with B-TSP 
in $\Delta t$.
\end{proposition}
This proposition follows immediately from 
using the perturbation series \eqref{perturbationseries1} 
together with Lemma \ref{trunc-smooth-thm}.
Next, we provide a condition 
for rank selection in the rank-adaptive methods \eqref{one-step-trunc-best_adapt_rank}-\eqref{one-step-trunc-svd_adapt_rank} 
which guarantees a consistent approximation to equation \eqref{mol-ode}.

\begin{proposition}[Rank selection for B-ST consistency]
\label{prop:bst-adaptive-rank-consistent}
Let ${\bf \Phi}({\bf N}, {\bf f}_k, \Delta t)$
be an order-$p$ increment function. The step-truncation method
\begin{equation}
\nonumber
{\bf a}_k ={\bf f}_k + \Delta t {\bf \Phi}({\bf N}, {\bf f}_k, \Delta t),\qquad {\bf f}_{k+1} =
{\mathfrak{T}}_{{\bf r}_k}^{\text{best}}({\bf a}_k),
\end{equation}
approximates \eqref{mol-ode} with order-$p$ 
local truncation error if and only if 
there exists an $M>0$ (independent of $k$) 
such that the rank ${\bf r}_{k}$ at time index $k$
satisfies the inequality
\begin{equation}
\label{truncation-accuracy-with-order}
\left | \left |
{\bf a}_k - {\mathfrak{T}}_{{\bf r}_k}^{\text{best}}({\bf a}_k) 
\right | \right |_2
\leq M\Delta t^{p+1}.
\end{equation} 
\end{proposition}
\begin{proof}
Denote by ${\bf f}(t_{k+1})$ the exact solution to 
${{d}{\bf f}}/{{ d} t} ={\bf N}({\bf f})$
with initial condition ${\bf f}_k$ at time $t_k$. 
For the forward implication, suppose there exists a constant 
$C_1$ such that 
$\left | \left |
{\bf f}(t_{k+1}) - {\bf f}_{k+1}
\right | \right |_2\leq C_1 \Delta t^{p+1}$. Then, 
\begin{align*}
\left | \left |
{\bf a}_k - {\mathfrak{T}}_{{\bf r}_k}^{\text{best}}({\bf a}_k) 
\right | \right |_2
&\leq \left | \left |
{\bf a}_k - {\bf f}(t_{k+1})
\right | \right |_2
+
\left | \left |
 {\bf f}(t_{k+1}) - {\mathfrak{T}}_{{\bf r}_k}^{\text{best}}({\bf a}_k) 
\right | \right |_2\\
&\leq C_2 \Delta t^{p+1} + \left | \left |
{\bf f}(t_{k+1}) - {\bf f}_{k+1}
\right | \right |_2 \\
&\leq C_2 \Delta t^{p+1} +  C_1 \Delta t^{p+1},
\end{align*}	
where $C_2$ is a constant.
To prove the converse, we estimate
the local truncation error as
\begin{align*}
	\left | \left |
		{\bf f}(t_{k+1}) - {\bf f}_{k+1}	
	\right | \right |_2
	&\leq
	\left | \left |
		{\bf f}(t_{k+1}) - {\bf a}_{k}	
	\right | \right |_2
	+
	\left |	\left |
		{\bf a}_k - {\bf f}_{k+1}	
	\right | \right |_2	\\
	&\leq
	C_2 \Delta t^{p+1}	+
	\left | \left |
		{\bf a}_k - {\mathfrak{T}}_{{\bf r}_k}^{\text{best}}({\bf a}_k) 
	\right | \right |_2\\
	&\leq C_2\Delta t^{p+1} + M\Delta t^{p+1},
\end{align*}
where $C_2$ is a constant. 
%\begin{flushright}
\hfill\(\qed\)
%\end{flushright}
\end{proof}
Recalling Remark \ref{remark:all-tensors-have-rank}, for any 
given tensor ${\bf a}_k \in \mathbb{R}^{n_1 \times \cdots \times n_d}$
there exists a rank ${\bf r}_k$ which makes the 
left hand side of the inequality 
\eqref{truncation-accuracy-with-order} equal to zero. 
Thus, there always exists a rank ${\bf r}_k$ which 
satisfies \eqref{truncation-accuracy-with-order}.
Using consistency of the rank-adaptive B-ST scheme 
\eqref{one-step-trunc-best_adapt_rank} proven in Proposition \ref{prop:bst-adaptive-rank-consistent}, 
we can easily obtain consistency results 
for step-truncation methods based 
on quasi-optimal truncation operators such as 
${\mathfrak T}_{\bf r}^{\text{SVD}}$. 
%{\color{blue}
%In principle, the results below also apply to the
%quasi-optimal projections of \cite{kieri2019projection}.
%For concreteness, we will express
%the results in terms of 
%${\mathfrak T}_{\bf r}^{SVD}$.
To do so, we first show that the local truncation error of SVD-ST 
scheme \eqref{one-step-trunc-svd} 
is dominated by the local truncation 
error of the B-ST scheme \eqref{one-step-trunc-best}.
%
%For convenience of notation, we will the drop
%the discrete time subscript $k$ whenever we are
%studying properties due to a single step of a
%temporal integrator. We will explicitly write the
%subscript $k$ whenever relating multiple time
%steps is relevant.

\begin{lemma}[Error bound on the SVD 
step-truncation scheme]
\label{svd-local-error-statement}
Let ${\bf f}(t_{k+1})$ denote the exact solution to 
${{\text d}{\bf f}}/{{\text d} t} ={\bf N}({\bf f})$
with initial condition ${\bf f}_k$ at time $t_k$, 
and let ${\bf \Phi}({\bf N}, {\bf f}, \Delta t)$ be
an order-$p$ increment function.
The local truncation error of the SVD-ST integrator 
\eqref{one-step-trunc-svd} satisfies 
\begin{align}
&\left \|
	{\bf f}(t_{k+1})-
	{\mathfrak T}_{\bf r}^{\text{SVD}}\left ( {\bf f}(t_k)+
	\Delta t {\bf \Phi}({\bf N}, {\bf f}(t_k), \Delta t) \right )
\right \|_2 \leq
\nonumber \\ 
&\hspace{17mm}
K\left (1+\sqrt{2d - 3}\right )\Delta t^{p+1} + \sqrt{2d - 3}
\left\|{\bf f}(t_{k+1})- 
{\mathfrak T}_{\bf r}^{\text{best}}\left({\bf f}(t_k)+
\Delta t {\bf \Phi}({\bf N}, {\bf f}(t_k), \Delta t) \right)\right\|_2.
\nonumber
\end{align}
\end{lemma}
\begin{proof}
First, we apply triangle inequality 
\begin{align}
 \nonumber
 &\left\|{\bf f}(t_{k+1})-
 {\mathfrak T}_{\bf r}^{\text{SVD}}({\bf f}(t_k)+
\Delta t {\bf \Phi}({\bf N}, {\bf f}(t_k), \Delta t)) \right\|_2
 \leq\\ \nonumber
 \qquad 
 &\hspace{15mm}
 \left\|{\bf f}(t_{k+1})-
 \left( {\bf f}(t_k)+
\Delta t {\bf \Phi}({\bf N}, {\bf f}(t_k), \Delta t)\right)\right\|_2+\\ 
 &\hspace{34mm} 
 \left\|{\bf f}(t_k)+
\Delta t {\bf \Phi}({\bf N}, {\bf f}(t_k), \Delta t)-
 {\mathfrak T}_{\bf r}^{\text{SVD}}\left({\bf f}(t_k)+
\Delta t {\bf \Phi}({\bf N}, {\bf f}(t_k), \Delta t)\right)\right\|_2.
\label{g1}
\end{align}
Since the increment function ${\bf \Phi}({\bf N}, {\bf f}(\tau), \Delta t)$ is of order $p$, we can replace the 
first term at the right hand side of \eqref{g1} 
by $K\Delta t^{p+1}$, i.e., 
\begin{align}
\nonumber
&\left\|
 {\bf f}(t_{k+1})-
 {\mathfrak T}_{\bf r}^{\text{SVD}}({\bf f}(t_k)+
 \Delta t {\bf \Phi}({\bf N}, {\bf f}(t_k), \Delta t))
\right\|_2
 \leq\\ \nonumber
&\hspace{20mm}K\Delta t^{p+1} +
 \left\|{\bf f}(t_k)+
\Delta t {\bf \Phi}({\bf N}, {\bf f}(t_k), \Delta t)-
 {\mathfrak T}_{\bf r}^{\text{SVD}}\left({\bf f}(t_k)+
\Delta t {\bf \Phi}({\bf N}, {\bf f}(t_k), \Delta t)\right)\right\|_2,\nonumber
\end{align}
where $K$ is a constant.
Next, we use the inequality \eqref{svd-trunc-close} 
to obtain 
\begin{align}
&\left\|{\bf f}(t_{k+1})-
 {\mathfrak T}_{\bf r}^{\text{SVD}}({\bf f}(t_k)+
\Delta t {\bf \Phi}({\bf N}, {\bf f}(t_k), \Delta t)) \right\|_2
 \leq \nonumber \\
 &\hspace{20mm}K\Delta t^{p+1}
 +
 \left(\sqrt{2d-3}\right )\left\|{\bf f}(t_k)+
\Delta t {\bf \Phi}({\bf N}, {\bf f}(t_k), \Delta t)-
 {\mathfrak T}_{\bf r}^{\text{best}}\left({\bf f}(t_k)+
\Delta t {\bf \Phi}({\bf N}, {\bf f}(t_k), \Delta t)\right)\right\|_2.
\nonumber
\end{align}
Another application of triangle inequality yields 
\begin{align}
\nonumber
 &\left\|{\bf f}(t_{k+1})-
 {\mathfrak T}_{\bf r}^{\text{SVD}}({\bf f}(t_k)+
\Delta t {\bf \Phi}({\bf N}, {\bf f}(t_k), \Delta t)) \right\|_2
 \leq\\ \nonumber
 &\hspace{20mm}
 K\Delta t^{p+1} +
 \left(\sqrt{2d-3}\right )\bigg(
\left\|{\bf f}(t_k)+
\Delta t {\bf \Phi}({\bf N}, {\bf f}(t_k), \Delta t)- 
{\bf f}(t_{k+1})\right\|_2+\nonumber\\
 &\hspace{54mm}
 \left\|{\bf f}(t_{k+1})-
 {\mathfrak T}_{\bf r}^{\text{best}}\left({\bf f}(t_k)+
\Delta t {\bf \Phi}({\bf N}, {\bf f}(t_k), \Delta t)\right)\right\|_2 
\bigg).
\nonumber
\end{align}
Finally, collecting like terms yields the desired result.
%\begin{flushright}
\hfill\(\qed\)
%\end{flushright}
\end{proof}
By combining Proposition 
\ref{prop:bst-adaptive-rank-consistent} and
Lemma \ref{svd-local-error-statement}, it is 
straightforward to prove the following 
consistency result for the rank-adaptive SVD-ST 
integrator \eqref{one-step-trunc-best_adapt_rank}.

\begin{corollary}[Rank selection for SVD-ST consistency]
\label{svdst-adaptive-rank-consistent}
Let ${\bf \Phi}({\bf N}, {\bf f}_k, \Delta t)$
be an order-$p$ increment function. The step-truncation 
method  
\begin{equation}
\nonumber
{\bf a}_k ={\bf f}_k + \Delta t {\bf \Phi}({\bf N}, {\bf f}_k, \Delta t),\qquad {\bf f}_{k+1} =
{\mathfrak{T}}_{{\bf r}_k}^{\text{SVD}}({\bf a}_k)
\end{equation}
approximates 
\eqref{mol-ode} with order-$p$ local truncation error
if and only if 
there exists an $M>0$ 
such that the rank ${\bf r}_{k}$ at time index $k$
satisfies the inequality
\begin{equation}
\label{truncation-accuracy-with-order_2}
\left | \left |
{\bf a}_k - {\mathfrak{T}}_{{\bf r}_k}^{\text{best}}({\bf a}_k) 
\right | \right |_2
\leq M\Delta t^{p+1}.
\end{equation} 
\end{corollary}
Note that by inequality \eqref{svd-trunc-close}, the statement 
in \eqref{truncation-accuracy-with-order_2} is equivalent to 
\begin{equation}
\label{truncation-accuracy-with-order_3}
\left |\left |
{\bf a}_k - {\mathfrak{T}}_{{\bf r}_k}^{\text{SVD}}({\bf a}_k)
\right |\right |_2\leq M' \Delta t^{p+1},
\end{equation}
for another constant $M' > 0$, which depends on $d$.
Consistency results analogous to Corollary 
\ref{svdst-adaptive-rank-consistent} for step-truncation 
integrators based on any quasi-optimal truncation 
can be obtained in a similar way.
{%\color{blue}
\subsection{Error constants}
\label{subsec:const-selection}
In this section we 
provide a lower bound for the constant $M$ 
appearing in Corollary \ref{svdst-adaptive-rank-consistent}.
To simplify the presentation we develop the bounds for the matrix 
case ($d=2$) and note that similar results for $d>2$ 
can be obtained by using the hierarchical approximability 
theorem discussed in \cite{grasedyck2010hierarchical}.

With reference to Corollary \ref{svdst-adaptive-rank-consistent}, 
let $\{\sigma_i\}$ be the set of
singular values of ${\bf a}_k$ 
and let ${\bf a}_k -
{\mathfrak T}_{r_k}({\bf a}_k)={\bf E}_k
\in {\mathbb{R}}^{n_1\times n_2}$ be the error
matrix due to tensor truncation. 
%We will relate the decay
%rate of the singular values of ${\bf a}_k$
%to the rank $r_k$ needed to satisfy  Corollary
%\ref{svdst-adaptive-rank-consistent}.
Then the local consistency 
condition \eqref{truncation-accuracy-with-order_2} 
can be written as
\begin{equation}
\begin{aligned}
\label{25}
\left \|
{\bf a}_k -
{\mathfrak T}_{r_k}({\bf a}_k)
\right \|_2^2
&=
\left \|
{\bf E}_k
\right \|_2^2 \\
&=
\sum_{i={r_k}+1}^{\text{min}(n_1,n_2)}
\sigma_i^2
\leq M^2\Delta t^{2p+2}.
\end{aligned}
\end{equation}
Equation \eqref{25} can be used to obtain the following lower bound for the coefficient $M$
\begin{equation}
\label{eqn:local-error-coef-bound}
M\geq\frac{1}{\Delta t^{p+1}}\sqrt{{\displaystyle
\sum_{i={r_k}+1}^{\text{min}(n_1,n_2)}
\sigma_i^2
}}.
\end{equation}
The lower bound can be explicitly computed if we 
have available the decay rate of the singular values $\{ \sigma_i \}$. 
For instance, if  the singular values decay 
exponentially fast (as in the case of
singular values considered in
\cite{opmeer2015decay}), i.e.,
$\sigma^2_i \leq Cq^i$ for some
$C>0$ and $q\in(0,1)$, 
Then by the geometric series formula we have that
\begin{equation*}
\left\|{\bf a}_k\right\|_2^2\leq C\sum_{i=1}^\infty q^i =C\frac{q}{(1-q)},
\end{equation*} 
which yields $Cq \geq (1-q)\|{\bf a}_k\|_2^2$.
In this case we may bound the local error as 
\begin{equation}
\nonumber
\left\|{\bf E}_k\right\|_2^2
\leq
C\sum_{i=r_k+1}^\infty
q^i
=
C\sum_{i=1}^\infty
q^i
-
C\sum_{i=1}^{r_k}
q^i
=
C\frac{q - (q - q^{r_k+1})}{1-q}
=
C\frac{q^{r_k+1}}{1-q}.
\end{equation}
Inserting this bound into
\eqref{eqn:local-error-coef-bound}
and recalling that $Cq \geq (1-q)\left\|{\bf a}_k\right\|_2^2$
and $\left\|{\bf a}_k\right\|_2 \geq \left\|{\bf E}_k\right\|_2$
we obtain
\begin{equation}
\label{eqn:rank-constant-bounding}
M\geq 
\frac{1}{\Delta t^{p+1}}\sqrt{ \frac{Cq^{r_k+1}}{(1-q)} }
\geq
\frac{1}{\Delta t^{p+1}}
\sqrt{\frac{(1-q)\left\|{\bf a}_k\right\|_2^2q^{r_k}}
{(1-q)}}
=\frac{\left\|{\bf a}_k\right\|_2\sqrt{q^{r_k}}}{\Delta t^{p+1}}. 
\end{equation}
Equation \eqref{eqn:rank-constant-bounding} establishes a relationship
between the local error coefficient $M$, the solution rank $r_k$ at time step $k$, the time step $\Delta t$, and the 2-norm of the 
solution ${\bf a}_k$ at time step $k$. 

A similar relation can be derived for singular values 
$\{\sigma_i\}$ decaying algebraically, i.e., 
$\sigma_i^2 \leq Ci^{-1-2s}$, where $s\in\mathbb{N}$. 
It was shown in
\cite{griebel2018decay} that this decay rate
occurs when discretizing an $s$-times 
differentiable bivariate function.
Moreover, it was also shown that 
\begin{equation*}
\left\|{\bf E}\|_2 \leq K\|{\bf a}_k\right\|_2(r_k +1)^{-s},
\end{equation*}
where $K$ is a constant related to the
measure of the domain of the aforementioned 
$s$-times differentiable bivariate function.
Therefore, if we choose the rank $r_k$ to satisfy the inequality
\begin{equation}
\nonumber
\frac{M}{K}\geq \frac{\left\|{\bf a}_k\right\|_2}{(r_k +1)^{s}}
\end{equation}
then we have that condition \eqref{truncation-accuracy-with-order_2} is also satisfied. An expression for $K$ may be found in Theorem 3.3 of
\cite{griebel2018decay}.
}
\section{Convergence of rank-adaptive step-truncation schemes}
\label{sec:global-error}
We have shown that the proposed methods are consistent,
now we prove convergence. To do so,
let us assume that the increment function 
${\bf \Phi}({\bf N}, {\bf f}, \Delta t)$ satisfies the 
following stability condition: There exist 
constants $C,D,E \geq 0$ and a positive integer 
$m \leq p$ so that as $\Delta t\rightarrow 0$, the inequality 
\begin{equation}
\label{stability_of_increment_function}
\left \|
{\bf \Phi}({\bf N}, \hat{\bf f}, \Delta t) - 
{\bf \Phi}({\bf N}, {\tilde{\bf f}}, \Delta t)
\right \|_2 \leq 
(C+D\Delta t) 
\left \|
\hat{\bf f} - {\tilde{\bf f}}
\right \|_2
+E\Delta t^{m}
\end{equation}
holds for all
$\hat{\bf f},\tilde{\bf f} \in \mathbb{R}^{n_1 \times \cdots \times n_d}$.
This assumption is crucial in our development of global error 
analysis for rank-adaptive step-truncation methods.
\begin{theorem}[Global error for rank-adaptive schemes]
\label{thm:st-global-error}
Let ${\bf f}(t)$ be the exact solution
to \eqref{mol-ode}, assume $\bf N$ is 
Lipschitz continuous with constant $L$, and  
let ${\bf \Phi}({\bf N}, {\bf f}, \Delta t)$ be
an order-$p$ increment function satisfying the stability 
criterion \eqref{stability_of_increment_function}.
If 
\begin{equation}
{\bf f}_{k+1} =
{\mathfrak{T}}_{{\bf r}_k}({\bf f}_k +
\Delta t{\bf \Phi}({\bf N}, {\bf f}_k, \Delta t))
\nonumber
\end{equation}
is an order-$p$ consistent step-truncation method, 
where ${\mathfrak{T}}_{{\bf r}_k} =
{\mathfrak{T}}_{{\bf r}_k}^{\text{best}}$ or
${\mathfrak{T}}_{{\bf r}_k} =
{\mathfrak{T}}_{{\bf r}_k}^{\text{SVD}}$
(see Proposition \ref{prop:bst-adaptive-rank-consistent}
or Corollary \ref{svdst-adaptive-rank-consistent}),
then the global error satisfies 
\begin{equation}
\left \|
{\bf f}(T) - {\bf f}_N 
\right \|_2
\leq
Q  \Delta t^{z},\nonumber
\end{equation}
where $z = \min(p,m)$.
The constant $Q$ depends on
the local error and stability
coefficients of the increment function
$\bf {\Phi}$, and the truncation constant $M$ in \eqref{truncation-accuracy-with-order} (or $M'$ in \eqref{truncation-accuracy-with-order_3}).
\end{theorem}
\begin{proof}
We induct on the number of time steps 
(i.e., $N$ in \eqref{time_discretization}), assuming
that $\Delta t$ is small enough for the local
error estimations to hold true.
The base case is given by one step error ($N=1$) which is 
local truncation error. Thus, from our consistency assumption 
we immediately obtain 
\begin{equation}
\left | \left |
{\bf f}(t_1) - {\mathfrak{T}}_{{\bf r}_0}({\bf f}_0 +
\Delta t{\bf \Phi}({\bf N}, {\bf f}_0, \Delta t))
\right | \right |_2
\leq C_0 \Delta t^{p+1},\nonumber
\end{equation}
which proves the base case. 
Now, assume that the error after $N-1$ steps satisfies 
\begin{equation}
\left | \left |
{\bf f}(t_{N-1}) - {\bf f}_{N-1}
\right | \right |_2
\leq 
Z_{N-1} \Delta t^{z},\nonumber
\end{equation}
where $z = \min(p,m)$.
Letting
${\bf a}_k = {\bf f}_k +
\Delta t{\bf \Phi}({\bf N}, {\bf f}_{k}, \Delta t)$
denote one step prior to truncation, we 
expand the final step error in terms of 
penultimate step
\begin{align}
\label{global-error-inductive-start}
\left | \left |
{\bf f}(T) - {\bf f}_{N}
\right | \right |_2
&=
\left | \left |
{\bf f}(t_{N}) - {\mathfrak{T}}_{{\bf r}_{N-1}}({\bf f}_{N-1} +
\Delta t {\bf \Phi}({\bf N}, {\bf f}_{N-1}, \Delta t))
\right | \right |_2\\ \nonumber
&\leq
\left | \left |
{\bf f}(t_{N}) - {\bf a}_{N-1}
\right | \right |_2+
\left | \left |
{\bf a}_{N-1} - {\mathfrak{T}}_{{\bf r}_{N-1}}({\bf f}_{N-1} +
\Delta t {\bf \Phi}({\bf N}, {\bf f}_{N-1}, \Delta t))
\right | \right |_2\\ \nonumber
&=
\left | \left |
{\bf f}(t_{N}) - {\bf a}_{N-1}
\right | \right |_2+
\left | \left |
{\bf a}_{N-1} - {\mathfrak{T}}_{{\bf r}_{N-1}}({\bf a}_{N-1})
\right | \right |_2\\ \nonumber
&\leq\left | \left |
{\bf f}(t_{N}) - {\bf a}_{N-1}
\right | \right |_2+ M \Delta t^{p+1}\\ \nonumber
&\leq
\left | \left |
{\bf f}(t_{N}) - \left (
{\bf f}(t_{N-1}) +
\Delta t {\bf \Phi}({\bf N},{\bf f}(t_{N-1}),\Delta t)
\right )
\right | \right |_2\\ \nonumber
&\ \ +
\left | \left |
{\bf f}(t_{N-1}) +
\Delta t {\bf \Phi}({\bf N},{\bf f}(t_{N-1}),\Delta t)
-{\bf a}_{N-1}
\right | \right |_2+ M \Delta t^{p+1}\\ \nonumber
&\leq 
\left | \left |
{\bf f}(t_{N-1}) +
\Delta t {\bf \Phi}({\bf N},{\bf f}(t_{N-1}),\Delta t )
-{\bf a}_{N-1}
\right | \right |_2+ K_{N-1}\Delta t^{p+1}+ M \Delta t^{p+1},
\end{align}
where $K_{N-1}$ is a local error constant
for the untruncated scheme \eqref{discrete-time-ode}.
%Recall the notation
%${\bf g}_{N-1} = {\bf f}_{N-1} +
%\Delta t_{N-1}{\bf \Phi}({\bf N}, {\bf f}_{N-1}, \Delta t_{N-1})$.
Expanding ${\bf a}_{N-1}$ and using the triangle inequality 
we find 
\begin{align}
\label{global-err-recursion}
\nonumber
\left | \left |
{\bf f}(T) - {\bf f}_{N}
\right | \right |_2
&\leq 
\left | \left |
{\bf f}(t_{N-1}) -
{\bf f}_{N-1}
\right | \right |_2 \\ \nonumber
&\qquad+\Delta t \left | \left |
{\bf \Phi}({\bf N},{\bf f}(t_{N-1}),\Delta t)
-{\bf \Phi}({\bf N},{\bf f}_{N-1},\Delta t)
\right | \right |_2\\ 
&\qquad \qquad + K_{N-1}\Delta t^{p+1}+ M \Delta t^{p+1}.
\end{align}
Using our assumption that the increment function is stable, 
\eqref{stability_of_increment_function} yields
\begin{align}
 |  |
{\bf f}(T) &- {\bf f}_{N}
 |  |_2
\leq 
\left | \left |
{\bf f}(t_{N-1}) -
{\bf f}_{N-1}
\right | \right |_2 +
(C+D\Delta t )\Delta t 
\left | \left |
{\bf f}(t_{N-1}) -
{\bf f}_{N-1}
\right | \right |_2 \nonumber \\ \nonumber
&\qquad\qquad\qquad +E\Delta t^{m+1}
+ K_{N-1}\Delta t^{p+1}+ M \Delta t^{p+1}\\
\nonumber
&=
(1+C\Delta t + D \Delta t^2)
\left | \left |
{\bf f}(t_{N-1}) -
{\bf f}_{N-1}
\right | \right |_2
+E\Delta t^{m+1}
+ K_{N-1}\Delta t^{p+1}
+ M \Delta t^{p+1}\\\nonumber
&\leq
(1+C\Delta t +D\Delta t^2)
Z_{N-1} \Delta t^{z}
+E\Delta t^{m+1} 
+ K_{N-1}\Delta t^{p+1} 
+ M \Delta t^{p+1}.
\end{align}
Finally, recalling that $z=\min(p,m)$, we obtain 
\begin{align}
|  |
{\bf f}(T) - {\bf f}_{N}
 |  |_2
&\leq
(1+C\Delta t +D\Delta t^2)
Z_{N-1} \Delta t^{z}
+E\Delta t^{z+1} 
+ K_{N-1}\Delta t^{z+1} 
+ M \Delta t^{z+1} \nonumber \\ 
\nonumber
&= \left[ (1+C\Delta t +D\Delta t^2)
Z_{N-1} + E \Delta t + K_{N-1} \Delta t + M \Delta t \right] \Delta t^z,
\end{align}
concluding the proof.
\hfill\(\qed\)
\end{proof}
Since the constants $M,C,D$, and $E$ are fixed in time, the local error 
coefficients $K_{j}$, which depend only on the untruncated scheme 
\eqref{discrete-time-ode}, determine if the error blows
up as the temporal grid is refined.
Hereafter we provide several examples of globally 
convergent rank-adaptive step-truncation methods. 
In each example, $\mathfrak{T}_{{\bf r}}$ denotes any optimal or 
quasi-optimal truncation operator, e.g., the best rank-$\bf r$ 
truncation operator \eqref{ht-best} or the SVD 
truncation operator \eqref{ht-SVD}.

\subsection{Rank-adaptive Euler scheme}
\label{subsec:adaptive-euler}
Our first example is a first-order method 
for solving \eqref{mol-ode} based on Euler forward. 
From Theorem \ref{thm:st-global-error}, we know that 
the scheme 
\begin{equation}
\label{adaptive-euler-method0}
{\bf f}_{k+1} =
{\mathfrak{T}}_{{\bf r}_k}({\bf f}_k +
\Delta t{\bf N}( {\bf f}_k))
\end{equation}
is order one in $\Delta t$, provided 
the vector field $\bf N$ is Lipschitz 
and the truncation rank ${\bf r}_{k}$ 
satisfies 
\begin{equation}
\nonumber
\| {\bf f}_k + \Delta t {\bf N}({\bf f}_k) - \mathfrak{T}_{{\bf r}_k} ({\bf f}_k + \Delta t {\bf N}({\bf f}_k ))\|_2 \leq M \Delta t^2, 
\end{equation}
for all $k = 1,2,\ldots$.
Applying the nonlinear vector field $\bf N$ to the solution tensor 
$\bf f$ can result in a tensor with large rank. 
Therefore, it may be desirable to apply a tensor 
truncation operator to ${\bf N}({\bf f})$ at 
each time step. To implement this, 
we build the additional truncation 
operator into the increment function 
\begin{equation}
\label{modified_inc_func_first_order}
{\bf \Phi}({\bf N}, {\bf f}_k, {\bf s}_k, \Delta t) =
{\mathfrak{T}}_{{\bf s}_k}({\bf N}( {\bf f}_k)),
\end{equation}
to obtain the new scheme 
\begin{equation}
\label{adaptive-euler-method}
{\bf f}_{k+1} =
{\mathfrak{T}}_{{\bf r}_k}({\bf f}_k +
\Delta t {\mathfrak{T}}_{{\bf s}_k}({\bf N}( {\bf f}_k))).
\end{equation}
We now determine conditions for ${\bf s}_k$ and ${\bf r}_k$ 
which make the scheme \eqref{adaptive-euler-method} 
first-order. 
To address consistency, suppose ${\bf f}(t_{k+1})$ is 
the analytic solution to \eqref{mol-ode} with initial 
condition ${\bf f}_k$ at time $t_k$. 
Then, bound the local truncation error as
\begin{align}
\nonumber
\left | \left |
{\bf f}(t_{k+1}) - {\bf f}_{k+1}
\right | \right |_2
&\leq
\left | \left |
{\bf f}(t_{k+1})
-({\bf f}_k + \Delta t {\bf N}({\bf f}_k))
\right | \right |_2
+
\left | \left |
{\bf f}_k + \Delta t {\bf N}({\bf f}_k)-
{\mathfrak{T}}_{{\bf r}_k}({\bf f}_k +
\Delta t{\mathfrak{T}}_{{\bf s}_k}({\bf N}( {\bf f}_k)))
\right | \right |_2\\
\nonumber
&\leq
K\Delta t^{2}
+
\left | \left |
{\bf f}_k + \Delta t {\bf N}({\bf f}_k)-
({\bf f}_k +
\Delta t{\mathfrak{T}}_{{\bf s}_k}({\bf N}( {\bf f}_k)))
\right | \right |_2\\ \nonumber
&\quad +
\left | \left |{\bf f}_k +
\Delta t{\mathfrak{T}}_{{\bf s}_k}({\bf N}( {\bf f}_k))
-
{\mathfrak{T}}_{{\bf r}_k}({\bf f}_k +
\Delta t{\mathfrak{T}}_{{\bf s}_k}({\bf N}( {\bf f}_k)))
\right | \right |_2\\
&=K\Delta t^{2}
+\Delta t 
\left | \left |
{\bf N}({\bf f}_k)-
{\mathfrak{T}}_{{\bf s}_k}({\bf N}( {\bf f}_k))
\right | \right |_2
+
\left | \left |{\bf f}_k +
\Delta t{\mathfrak{T}}_{{\bf s}_k}({\bf N}( {\bf f}_k))
-
{\mathfrak{T}}_{{\bf r}_k}({\bf f}_k +
\Delta t{\mathfrak{T}}_{{\bf s}_k}({\bf N}( {\bf f}_k)))
\right | \right |_2.
\nonumber
\end{align}
From this bound, we see that by selecting 
${\bf s}_k$ and ${\bf r}_k$ so that
\begin{equation}
\begin{aligned}
&\left | \left |
{\bf N}({\bf f}_k)-
{\mathfrak{T}}_{{\bf s}_k}({\bf N}( {\bf f}_k))
\right | \right |_2 \leq M_1 \Delta t , \\
&\left | \left |{\bf f}_k +
\Delta t{\mathfrak{T}}_{{\bf s}_k}({\bf N}( {\bf f}_k))-
{\mathfrak{T}}_{{\bf r}_k}({\bf f}_k +
\Delta t{\mathfrak{T}}_{{\bf s}_k}({\bf N}( {\bf f}_k)))
\right | \right |_2 \leq M_2 \Delta t^{2},
\nonumber
\end{aligned}
\end{equation}
for all $k = 1,2,\ldots$, 
yields an order one local truncation error for 
\eqref{adaptive-euler-method}.
To address stability, we show that the increment function 
\eqref{modified_inc_func_first_order} satisfies \eqref{stability_of_increment_function} with $m = 1$, 
assuming $\bf N$ is Lipschitz. Indeed, 
\begin{align}
	\nonumber
    \left | \left |
    {\mathfrak{T}}_{{\bf s}_k}({\bf N}( \hat{\bf f}))
    -   {\mathfrak{T}}_{{\bf s}_k}({\bf N}( \tilde{\bf f}))
    \right | \right |_2
    &\leq
    \left | \left |
    {\mathfrak{T}}_{{\bf s}_k}({\bf N}( \hat{\bf f}))
    -{\bf N}( \hat{\bf f})
    \right | \right |_2
    +
    \left | \left |
    {\bf N}( \hat{\bf f})
    -
    {\bf N}( \tilde{\bf f})
    \right | \right |_2
    +
    \left | \left |
    {\bf N}( \tilde{\bf f})
    -   {\mathfrak{T}}_{{\bf s}_k}({\bf N}( \tilde{\bf f}))
    \right | \right |_2\\
    &\leq 
     L \left | \left |
    \hat{\bf f}-\tilde{\bf f}
    \right | \right |_2+ 2M_1\Delta t,
\nonumber
\end{align}
where $L$ is the Lipschitz constant of $\bf N$.
Now applying Theorem 
\ref{thm:st-global-error} proves that 
the rank-adaptive Euler method
\eqref{adaptive-euler-method} has $O(\Delta t)$ global error.
%}
%{\color{blue}

\subsection{Rank-adaptive explicit midpoint scheme}
\label{subsec:adaptive-midpoint}

Consider the following rank-adaptive step-truncation method 
based on the explicit midpoint rule (see \cite[II.1]{HairerErnst1993SODE})
\begin{equation}
\label{adaptive-midpoint-method_1}
    {\bf f}_{k+1} =
    {\mathfrak{T}}_{{\bm \alpha}_k} \left (
    {\bf f}_{k} + \Delta t
     \left (
    {\bf N}\left (
    {\bf f}_{k} + \frac{\Delta t}{2}
    {\bf N}({\bf f}_{k})
    \right )\right ) \right ).
\end{equation}
We have proven in Theorem \ref{thm:st-global-error} that \eqref{adaptive-midpoint-method_1} is order 2 in $\Delta t$, 
provided the vector field $\bf N$ is Lipschitz 
and the truncation rank ${\bm \alpha}_{k}$ satisfies 
\begin{equation}
\nonumber
\| {\bf a}_k - \mathfrak{T}_{{\bm \alpha}_k} ({\bf a}_k)\|_2 \leq M \Delta t^3, 
\end{equation}
for all $k = 1,2,\ldots$. Here, 
\begin{equation}
\nonumber
{\bf a}_k =  {\bf f}_{k} + \Delta t
     \left (
    {\bf N}\left (
    {\bf f}_{k} + \frac{\Delta t}{2}
    {\bf N}({\bf f}_{k})
    \right )\right ) 
\end{equation}
denotes the solution tensor at time $t_{k+1}$ 
prior to truncation.
For the same reasons we discussed 
in section \ref{subsec:adaptive-euler}, 
it may be desirable to insert truncation 
operators inside the increment 
function. For our rank-adaptive explicit midpoint method 
we consider the increment function 
\begin{equation}
\label{increment_fun_adapt_mid}
{\bf \Phi}({\bf N}, {\bf f}_k, {\bm \beta}_k, {\bm \gamma}_k,\Delta t) = 
{\mathfrak{T}}_{{\bm \beta}_k}\left (
    {\bf N}\left (
    {\bf f}_{k} + \frac{\Delta t}{2}
    {\mathfrak{T}}_{{\bm \gamma}_k}(
    {\bf N}({\bf f}_{k}))
    \right )\right ) ,
\end{equation}
which results in the step-truncation scheme
\begin{equation}
\label{adaptive-midpoint-method_2}
    {\bf f}_{k+1} =
    {\mathfrak{T}}_{{\bm \alpha}_k} \left (
    {\bf f}_{k} + \Delta t
    {\mathfrak{T}}_{{\bm \beta}_k}\left (
    {\bf N}\left (
    {\bf f}_{k} + \frac{\Delta t}{2}
    {\mathfrak{T}}_{{\bm \gamma}_k}(
    {\bf N}({\bf f}_{k}))
    \right )\right ) \right ).
\end{equation}
Following a similar approach as in 
section \ref{subsec:adaptive-euler}, 
we aim to find conditions on ${\bm \alpha}_k$, 
${\bm\beta}_k$ and ${\bm\gamma}_k$ so that the 
local truncation error of the scheme 
\eqref{adaptive-midpoint-method_2} is order 2. 
For ease of notation, let us denote the truncation errors by
$\varepsilon_{{\bm \kappa}} = 
\left\|{\bf g} - {\mathfrak{T}}_{{\bm \kappa}} 
({\bf g})\right\|_2$, where 
${\bm \kappa }={\bm \alpha}_k$, ${\bm \beta}_k$, or 
${\bm \gamma}_k$.
The local truncation error of the scheme \eqref{adaptive-midpoint-method_2} can be estimated as 
\begin{align}
	\nonumber
    \left | \left |
    {\bf f}(t_{k+1}) - {\bf f}_{k+1}
    \right | \right |_2
    &\leq \varepsilon_{{\bm\alpha}_k}
    +\left | \left |
    {\bf f}(t_{k+1}) - \left (
    {\bf f}_k + \Delta t
    {\mathfrak{T}}_{{\bm \beta}_k}\left (
    {\bf N}\left (
    {\bf f}_k + \frac{\Delta t}{2}
    {\mathfrak{T}}_{{\bm \gamma}_k}(
    {\bf N}({\bf f}_k))
    \right )\right )
    \right )
    \right | \right |_2\\ \nonumber
     &\leq \varepsilon_{{\bm\alpha}_k}
     +\Delta t\varepsilon_{{\bm\beta}_k}
     +\left | \left |
    {\bf f}(t_{k+1}) - \left (
    {\bf f}_k + \Delta t
    {\bf N}\left (
    {\bf f}_k + \frac{\Delta t}{2}
    {\mathfrak{T}}_{{\bm \gamma}_k}(
    {\bf N}({\bf f}_k))
    \right )\right )
    \right | \right |_2\\
    &\leq K\Delta t^{3} +\varepsilon_{{\bm\alpha}_k}
     +\Delta t\varepsilon_{{\bm\beta}_k} +
     \frac{L\Delta t^2}{2}\varepsilon_{{\bm\gamma}_k},
\nonumber
\end{align}
where $L$ is the Lipschitz constant of $\bf N$.
From this bound, we see that if the 
truncation ranks ${\bm \alpha}_k, {\bm\beta}_k$ and
${\bm\gamma}_k$ are chosen such that 
\begin{equation}
\label{mid_point_trun_requirements}
\varepsilon_{{\bm\alpha}_k} \leq A\Delta t^3, \qquad
\varepsilon_{{\bm\beta}_k} \leq B\Delta t^2, \qquad
\varepsilon_{{\bm\gamma}_k}\leq G\Delta t, 
\end{equation}
for some constants $A$, $B$, and $G$, then the 
local truncation error of the scheme 
\eqref{adaptive-midpoint-method_2} is 
order $2$ in $\Delta t$.
Also, if \eqref{mid_point_trun_requirements} 
is satisfied then the stability requirement \eqref{stability_of_increment_function} is 
also satisfied. Indeed,
\begin{align}
	\nonumber
    &\left | \left |
    {\mathfrak{T}}_{{\bm \beta}_k}\left (
    {\bf N}\left (
    \hat{\bf f} + \frac{\Delta t}{2}
    {\mathfrak{T}}_{{\bm \gamma}_k}(
    {\bf N}(\hat{\bf f}))
    \right )\right )
    -{\mathfrak{T}}_{{\bm \beta}_k}\left (
    {\bf N}\left (
    \tilde{\bf f} + \frac{\Delta t}{2}
    {\mathfrak{T}}_{{\bm \gamma}_k}(
    {\bf N}(\tilde{\bf f}))
    \right )\right )
    \right | \right |_2
    \leq 
    \left (
     L+\frac{L}{2}\Delta t
    \right )\left | \left |
    \hat{\bf f} - \tilde{\bf f}
    \right | \right |_2 +2\varepsilon_{{\bm\beta}_k}+
     {L\Delta t}\varepsilon_{{\bm\gamma}_k}
\end{align}
holds for all tensors $\hat{\bf f}, \tilde{\bf f} \in \mathbb{R}^{n_1 \times \cdots \times n_d}$.
To arrive at the above relationship, we applied
triangle inequality several times to pull out
the $\varepsilon_{\bm \kappa}$ terms and then used 
Lipschitz continuity of $\bf N$ multiple times.
Thus, if the truncation ranks ${\bm \alpha}_k, {\bm\beta}_k$ and
${\bm\gamma}_k$ are chosen to satisfy \eqref{mid_point_trun_requirements} 
and the vector field $\bf N$ is Lipschitz, then 
Theorem \ref{thm:st-global-error} 
proves the method \eqref{adaptive-midpoint-method_2}
has $O(\Delta t^2)$ global error.

\subsection{Rank-adaptive Adams-Bashforth scheme}
\label{subsec:adaptive-adams}
With some minor effort we can extend the rank-adaptive 
global error estimates to the well-known multi-step methods 
of Adams and Bashforth (see \cite[III.1]{HairerErnst1993SODE}).
These methods are of the form
\begin{equation}
    \label{adams-bashforth}
    {\bf f}_{k+1} = {\bf f}_k +
    \Delta t\sum_{j=0}^{s-1}b_j {\bf N}({\bf f}_{k-j}),
\end{equation}
where $s$ is the number of steps.
A rank-adaptive step-truncation version of this method 
is
\begin{equation}
    \label{adaptive-ab}
    {\bf f}_{k+1} ={\mathfrak T}_{{\bm \alpha}_k}\left (
    {\bf f}_k +
    \Delta t
    {\mathfrak T}_{{\bm \beta}_k}\left (
    \sum_{j=0}^{s-1}b_j
    {\mathfrak T}_{{\bm \gamma}_{k}({j})}\left (
    {\bf N}({\bf f}_{k-j})
    \right )
    \right )
    \right ).
\end{equation}
In order to obtain a global error estimate 
for \eqref{adaptive-ab}, we follow the same steps 
as before. First we prove consistency, then 
we prove stability, and finally combine these 
results to obtain a global convergence result.
For consistency, let  
${\bf f}_0={\bf f}(t_0)$, ${\bf f}_1={\bf f}(t_1)$, $\dots$,
${\bf f}_{s-1}={\bf f}(t_{s-1})$
be the exact solution to \eqref{mol-ode} given at the 
first $s$ time steps. 
For ease of notation, we do not let the truncation rank 
depend on time step $k$ as we are only analyzing one iteration 
of the multi-step scheme \eqref{adaptive-ab}. 
Also, define the truncation errors 
$\varepsilon_{\bm \kappa} = 
\left\|{\bf g} - {\mathfrak{T}}_{\bm \kappa}({\bf g})\right\|_2$, where 
${\bm \kappa }={\bm \alpha}$, ${\bm \beta}$, 
${{\bm \gamma}(j)}$, $j=0,\dots ,s-1$.
The local error admits the bound
\begin{align}
    \nonumber
    \left |\left |
        {\bf f}(t_s) - {\bf f}_s
    \right |\right |_2
    &\leq
    \varepsilon_{\bm \alpha}+
    \left |\left |
     {\bf f}(t_s) -
     \left (
        {\bf f}_{s-1}
        +
        \Delta t
        {\mathfrak T}_{{\bm \beta}}\left (
        \sum_{j=0}^{s-1}b_j
        {\mathfrak T}_{{\bm \gamma}({j})}\left (
        {\bf N}({\bf f}_{s-1-j})
        \right )
        \right )
     \right )
    \right |\right |_2\\  \nonumber
    &\leq 
    \varepsilon_{\bm \alpha}+
    \Delta t\varepsilon_{\bm \beta}
    +\left |\left |
     {\bf f}(t_s) -
     \left (
        {\bf f}_{s-1}
        +
        \Delta t
        \sum_{j=0}^{s-1}b_j
        {\mathfrak T}_{{\bm \gamma}({j})}\left (
        {\bf N}({\bf f}_{s-1-j})
        \right )
     \right )
    \right |\right |_2\\
    \nonumber
    &\leq 
    \varepsilon_{\bm \alpha}+
    \Delta t\varepsilon_{\bm \beta}
    +\Delta t\sum_{j=0}^{s-1}|b_j|
        \varepsilon_{{\bm \gamma}(j)}
    +\left |\left |
     {\bf f}(t_s) -
     \left (
        {\bf f}_{s-1}
        +
        \Delta t
        \sum_{j=0}^{s-1}b_j
        {\bf N}({\bf f}_{s-1-j})
     \right )
    \right |\right |_2 .
\end{align}
The last term is the local error for an order-$s$
Adams-Bashforth method \eqref{adams-bashforth}.
Therefore, the local error of the step-truncation 
method \eqref{adaptive-ab} is also of order $s$ 
if the truncation ranks ${\bm \alpha}$, ${\bm \beta}$,
and ${{\bm \gamma}(j)}$ are chosen such that 
\begin{equation}
\label{truncation_errors_adams}
\varepsilon_{\bm \alpha} \leq A\Delta t^{s+1}, \qquad
\varepsilon_{\bm \beta} \leq B\Delta t^{s}, \qquad
\varepsilon_{{\bm \gamma}(j)} \leq G_j\Delta t^{s}.
\end{equation}
To address stability, we first need to generalize the stability condition 
\eqref{stability_of_increment_function} to the increment function 
\begin{equation}
\label{adams_increment_function}
{\bf \Phi}({\bf N},{\bf f}_1,{\bf f}_2, \ldots, {\bf f}_s, \Delta t) =  
 {\mathfrak T}_{{\bm \beta}_k}\left (
    \sum_{j=0}^{s-1}b_j
    {\mathfrak T}_{{\bm \gamma}_{k}({j})}\left (
    {\bf N}({\bf f}_{k-j})
    \right )
    \right )
\end{equation}
for the multi-step method \eqref{adams-bashforth}. 
A natural choice is 
\begin{equation}
\label{explicit-adams-stability}
\left | \left |{\bf \Phi}({\bf N},\hat{\bf f}_1,
\hat{\bf f}_2,\dots,\hat{\bf f}_{s}, \Delta t) - 
{\bf \Phi}({\bf N}, {\tilde{\bf f}}_1,
{\tilde{\bf f}}_2,\dots,{\tilde{\bf f}}_s, \Delta t)
\right | \right |_2 \leq 
\sum_{j=1}^s C_j
\left | \left |
\hat{\bf f}_j - {\tilde{\bf f}}_j
\right | \right |_2
+E\Delta t^{m}.
\end{equation}
Clearly, for $s=1$ the criterion 
\eqref{explicit-adams-stability} specializes to 
the stability criterion given in 
\eqref{stability_of_increment_function}. We have the bound 
\begin{align}
\nonumber
\left | \left |{\bf \Phi}({\bf N},\hat{\bf f}_1,
\hat{\bf f}_2,\dots,\hat{\bf f}_{s}, \Delta t) - 
{\bf \Phi}({\bf N}, {\tilde{\bf f}}_1,
{\tilde{\bf f}}_2,\dots,{\tilde{\bf f}}_s, \Delta t)
\right | \right |_2 &\leq  
    \left | \left | 
    {\mathfrak T}_{{\bm \beta}}\left (
    \sum_{j=0}^{s-1}b_j
    {\mathfrak T}_{{\bm \gamma}({j})}\left (
    {\bf N}(\hat{\bf f}_{s-j})
    \right )
    \right )
    -
    {\mathfrak T}_{{\bm \beta}}\left (
    \sum_{j=0}^{s-1}b_j
    {\mathfrak T}_{{\bm \gamma}({j})}\left (
    {\bf N}(\tilde{\bf f}_{s-j})
    \right )
    \right )
    \right | \right |_2
%    \\ \nonumber
%     \hspace{40mm}
%     \leq 2\varepsilon_{\bm \beta}
%    +\left | \left | 
%    \sum_{j=0}^{s-1}b_j
%    {\mathfrak T}_{{\bm \gamma}({j})}\left (
%    {\bf N}({\bf f}_{s-j})
%    \right )
%    -
%    \sum_{j=0}^{s-1}b_j
%    {\mathfrak T}_{{\bm \gamma}({j})}\left (
%    {\bf N}(\tilde{\bf f}_{s-j})
%    \right )
%    \right | \right |_2
    \\ \nonumber
%    &\hspace{40mm}
    &\leq 2\varepsilon_{\bm \beta}
    +2\sum_{j=0}^{s-1}|b_j|
        \varepsilon_{{\bm \gamma}(j)}
    +
    \sum_{j=0}^{s-1}|b_j|
    \left | \left | 
    {\bf N}(\hat{\bf f}_{s-j})
    -
    {\bf N}(\tilde{\bf f}_{s-j})
    \right | \right |_2 \\
%   &\hspace{40mm}
    &\leq 2\varepsilon_{\bm \beta}
    +2\sum_{j=0}^{s-1}|b_j|
        \varepsilon_{{\bm \gamma}(j)}
    +
    \sum_{j=0}^{s-1}L|b_j|
    \left | \left | \hat{\bf f}_{s-j}
    -\tilde{\bf f}_{s-j}
    \right | \right |_2,
    \nonumber
\end{align}
where we used triangle inequality to
set aside the ${\varepsilon}_{\bm \kappa}$ terms
and subsequently applied Lipschitz continuity 
multiple times.
From the above inequality, it is seen that if 
\eqref{truncation_errors_adams} is satisfied, 
then the stability condition \eqref{explicit-adams-stability}
is also satisfied with $m = s$. 
With the consistency and stability results for the 
multistep step-truncation method \eqref{adaptive-ab} just obtained, it is straightforward to 
obtain the following  global error estimate 
for \eqref{adaptive-ab}. 
\begin{corollary}[Global error of rank-adaptive Adams-Bashforth scheme]
\label{cor:ab-global-error}
Assume
${\bf f}_0={\bf f}(t_0)$, ${\bf f}_1={\bf f}(t_1)$, $\dots$,
${\bf f}_{s-1}={\bf f}(t_{s-1})$ are given initial steps
for a convergent order-$s$ method of the form
\eqref{adams-bashforth}, 
and assume $\bf N$ is Lipschitz with constant $L$.
If the rank-adaptive step-trunctation method \eqref{adaptive-ab} 
is order-$s$ consistent with \eqref{mol-ode}, 
and the corresponding increment function 
${\bf \Phi}$ defined in \eqref{adams_increment_function} 
satisfies the stability condition \eqref{explicit-adams-stability}, then the global error satisfies
\begin{equation}
\nonumber
    \left | \left |
    {\bf f}(T) - {\bf f}_N
    \right | \right |_2
    \leq Q\Delta t^s,
\end{equation}
where $Q$ depends only on the local error constants of the
Adams-Bashforth scheme \eqref{adams-bashforth}.
\end{corollary}
\begin{proof}
The proof is based on an inductive argument on the number of 
steps taken ($N$ in equation \eqref{time_discretization}), 
similar to the proof of Theorem \ref{thm:st-global-error}. 
First,
notice that by assuming the method \eqref{adaptive-ab} 
is order-$s$ consistent, we immediately obtain 
\eqref{adams-bashforth} for the base case $N = s$.
Now, suppose that 
\begin{equation}
    \label{adaptive-ab-induction}
    \left | \left |
    {\bf f}(t_{N-k}) - {\bf f}_{N-k}
    \right | \right |_2
    \leq Q_{N-k}\Delta t^s
\end{equation}
for all $k$, $1\leq k \leq N-s$.
It can be immediately verified that
\eqref{global-error-inductive-start}-\eqref{global-err-recursion} 
were derived without 
reference to a one-step method, so 
we can follow a very similar
string of inequalities to obtain 
\begin{align}
\left | \left |
{\bf f}(T) - {\bf f}_{N}
\right | \right |_2
&\leq 
\left | \left |
{\bf f}(t_{N-1}) -
{\bf f}_{N-1}
\right | \right |_2 \nonumber \\
&\ \  +\Delta t\left | \left |
{\bf \Phi}({\bf N},{\bf f}(t_{N-1}),
{\bf f}(t_{N-2}),\dots,
{\bf f}(t_{N-s}),\Delta t)
-{\bf \Phi}({\bf N},{\bf f}_{N-1},
{\bf f}_{N-2},\dots,
{\bf f}_{N-s},\Delta t)
\right | \right |_2 \nonumber \\
&\qquad + K_{N-1}\Delta t^{s+1}+ A \Delta t^{s+1}.
\nonumber
\end{align}
Applying the stability
condition \eqref{explicit-adams-stability} yields 
\begin{equation}
\left | \left |
{\bf f}(T) - {\bf f}_{N}
\right | \right |_2
\leq 
\left | \left |
{\bf f}(t_{N-1}) -
{\bf f}_{N-1}
\right | \right |_2 +
\Delta t
\sum_{j=1}^{s} C_j
\left | \left |
{\bf f}(t_{N-j}) - {\bf f}_{N-j}
\right | \right |_2
+(K_{N-1} + A +E)\Delta t^{s+1}.
\nonumber
\end{equation}
The above inequality together with the inductive hypothesis 
\eqref{adaptive-ab-induction} implies 
\begin{equation}
\label{ab_bound_final}
\begin{aligned}
\left | \left |
{\bf f}(T) - {\bf f}_{N}
\right | \right |_2
&\leq 
Q_{N-1}\Delta t^s +
\Delta t^{s+1}
\sum_{j=1}^{s} C_j
Q_{N-j}
+(K_{N-1} + A +E)\Delta t^{s+1},
\end{aligned}
\end{equation}
concluding the proof.
%Since in a single step $O(\Delta t^{s+1})$ is added
%the the global error, the $O(\Delta t^s)$ bound is
%maintained.
%\begin{flushright}
\hfill\(\qed\)
%\end{flushright}
\end{proof}
Similar to Theorem \ref{thm:st-global-error}, only 
the constants $K_j$ appearing in \eqref{ab_bound_final} 
depend on the time step (note that $Q_i$ also depends on 
$K_j$). Moreover, the constants $K_j$ depend only 
on the multi-step method \eqref{adams-bashforth}, 
and not on truncation.

\section{Consistency between rank-adaptive B-TSP and step-truncation schemes}
\label{sec:geom-interpreation}
In section \ref{sec:lte-svd-st}, Proposition \ref{prop:consistency_with_B-TSP}, 
we have shown that the fixed-rank step-truncation method 
\eqref{one-step-trunc-best} 
is consistent with the fixed-rank B-TSP method 
\eqref{discrete-proj-dynam}.
In this section we connect our rank-adaptive 
step-truncation schemes 
\eqref{one-step-trunc-best_adapt_rank}-\eqref{one-step-trunc-svd_adapt_rank} with the rank-adaptive B-TSP method we recently proposed in \cite{dektor2020rankadaptive}. 
In particular, we prove that the rank requirements for 
consistency in the rank-adaptive B-TSP method are equivalent 
to the rank requirements for a consistent step-truncation 
method as the temporal step size is sent to zero.
The rank-adaptive criterion for B-TSP 
checks if the normal component 
$\left\| ({\bf I} - {{\cal P}_{\bf f}}) {\bf N}({\bf f}) \right\|_2$ 
of ${\bf N}({\bf f})$ relative to the 
tangent space $T_{{\bf f}} {\cal H}_{ \bf r}$ is smaller than 
a threshold $\varepsilon_{\text{inc}}$, i.e., if 
\begin{equation}
\label{dynamic-threshold}
\left\|({\bf I} - {{\cal P}_{\bf f}}) {\bf N}({\bf f})\right\|_2\leq \varepsilon_{\text{inc}}.
\end{equation}
If \eqref{dynamic-threshold} is violated, 
then a rank increase is triggered 
and integration continues.
It was proven in \cite{dektor2020rankadaptive} that 
rank-adaptive B-TSP methods are consistent if the 
threshold in \eqref{dynamic-threshold} is chosen as 
$\varepsilon_{\text{inc}} = K\Delta t$ for any constant $K>0$.
We now show that this consistency condition 
for rank-adaptive B-TSP is equivalent to 
our rank selection 
requirements \eqref{truncation-accuracy-with-order} and 
\eqref{truncation-accuracy-with-order_3}
in the limit $\Delta t\rightarrow 0$.
\begin{proposition}[Geometric interpretation of rank addition]
\label{rmk:rank-limit-equiv}
Let ${\bf g}\in{\cal H}_{\bf r}$ and 
${\bf v}\in {\mathbb R}^{n_1\times n_2\times \dots \times n_d}$.
The following are equivalent as $\Delta t\rightarrow 0$:
\begin{align}
\label{eq:68}
\exists K\geq 0 \text{ so that } 
\left\| ({\bf I} - {{\cal P}_{\bf g}}) {\bf v} \right\|_2&\leq K\Delta t,\\
\label{eq:69}
\exists M\geq 0 \text{ so that } 
\left\|({\bf g} + \Delta t{\bf v}) - {\mathfrak{T}}_{{\bf r}}^{\text{best}}({\bf g} +
\Delta t {\bf v})\right\|_2 &\leq M\Delta t^{2}, \\
\label{eq:70}
\exists N\geq 0 \text{ so that } 
\left\|({\bf g} + \Delta t{\bf v}) - {\mathfrak{T}}_{{\bf r}}^{\text{SVD}}({\bf g} +
\Delta t {\bf v})\right\|_2 &\leq N\Delta t^{2}. 
\end{align}
\end{proposition}
\begin{proof}
The equivalence between \eqref{eq:69} and \eqref{eq:70} 
is an immediate consequence of \eqref{svd-trunc-close}. 
We now prove that \eqref{eq:68} is equivalent to 
\eqref{eq:69}. For the forward implication, 
assume $\left\|({\bf I} - {{\cal P}_{\bf g}}) {\bf v} 
\right\|_2\leq K\Delta t$.
We have 
\begin{align*}
\left\|({\bf g} + \Delta t{\bf v}) - {\mathfrak{T}}_{{\bf r}}^{\text{best}}({\bf g} +
\Delta t {\bf v})\right\|_2 &\leq
\left\|({\bf g} + \Delta t{\bf v}) - ({\bf g} +\Delta t {{\cal P}_{\bf g}} {\bf v})
\right\|_2 + \Delta t^2 C\\
&= \Delta t\left\|{\bf v} - {{\cal P}_{\bf g}} {\bf v}\right\|_2
+ \Delta t^2 C\\
&\leq\Delta t^2 K + \Delta t^2 C,
\end{align*}
where $C\geq 0$ denotes a constant obtained by a Taylor expansion of
${\mathfrak{T}}_{{\bf r}}^{\text{best}}$ (see \eqref{perturbationseries}).
Setting  $M \geq K+C$, proves the forward implication. 
Conversely, if we assume $
\left\|({\bf g} + \Delta t{\bf v}) - {\mathfrak{T}}_{{\bf r}}^{\text{best}}({\bf g} +
\Delta t {\bf v})\right\|_2 \leq M\Delta t^{2}$,
then
\begin{align*}
\left\|({\bf I} - {{\cal P}_{\bf g}}) {\bf v} \right\|_2&=
\Delta t^{-1}\left\|{\bf g} + \Delta t{\bf v} -
({\bf g}+ \Delta t{{\cal P}_{\bf g}}{\bf v})\right\|_2 \\
&\leq \Delta t^{-1}\left ( \left\|({\bf g} + \Delta t{\bf v}) - {\mathfrak{T}}_{{\bf r}}^{\text{best}}({\bf g} +
\Delta t {\bf v})\right\|_2  + C\Delta t^2\right ) \\
&\leq
%\Delta t^{-1}(M\Delta t^2 + C \Delta t^2)\\
%&=
\Delta t M +  \Delta t C.
\end{align*}
Setting  $K \geq M+C$, we prove 
\eqref{eq:69} implies \eqref{eq:68}.
\begin{flushright}
\qed
\end{flushright}
\end{proof}
The rank increase criterion \eqref{eq:68} for B-TSP 
offers geometric intuition which is not apparent from 
the step-truncation rank criterions \eqref{eq:69}-\eqref{eq:70}. 
That is, the solution rank should increase if 
the dynamics do not admit a sufficient approximation 
on the tensor manifold tangent space. Moreover, 
the accuracy required for approximating the dynamics 
depends directly on the time step size $\Delta t$ and the 
desired order of accuracy.
{%\color{blue}
We emphasize that
by applying
condition \eqref{eq:68}
to
\eqref{discrete-proj-dynam}
it is possible
to develop a rank-adaptive
version of the
step-truncation scheme
recently
proposed in
\cite{kieri2019projection}.
Specifically, the solution
rank $\bf r$ at each time step  can be 
chosen to satisfy a bound on the component 
of \eqref{discrete-proj-dynam} 
normal to the tensor manifold ${\cal H}_{\bf r}$.
}

{%\color{blue}
\section{Numerical applications}
\label{sec:numerical-apps}
In this section we present and discuss numerical applications 
of the proposed rank-adaptive step-truncation methods.
%
%%
%Specifically, in section \ref{sec:rank-shock} we 
%test the ability of the proposed
%rank-adaptive schemes to compute the solution
%to a problem in which the rank of the vector field 
%suddenly jumps in value (rank shock problem). This allows
%us to test the robustness of the proposed
%adaptive schemes in settings requiring
%abrupt changes in rank.
%%
%In subsection \ref{subsec:fokker-planck}
%we apply the proposed 
%rank-adaptive step-truncation 
%algorithms to a Fokker-Planck equation with 
%space-dependent drift and constant diffusion, 
%and demonstrate their accuracy in predicting 
%relaxation to statistical equilibrium.
%This problem allows us to observe the
%numerical effect of rank truncation
%of error. We also use this
%problem to numerically verify the
%proven order of the global error.
%
%
We have seen that these methods are defined by parameters 
summarized in Table \ref{table:method-params}. 
\begin{table}
\begin{centering}
\renewcommand{\arraystretch}{1.3}
\begin{tabular}{ |c|c|c|}
\hline
 Integration Method &
 Free Parameters &
 Dependent Parameters\\
\hline
\begin{tabular}{c}
Rank-adaptive Euler
\\
(Sec. \ref{subsec:adaptive-euler})
\end{tabular}
 &$\Delta t,M_1,M_2$   &
 \begin{tabular}{rl}
 $\varepsilon_{\bf r}$&$=M_1\Delta t^{2},$\\
$\varepsilon_{\bf s}$&$=M_2\Delta t$
 \end{tabular}\\
\hline
\begin{tabular}{c}
Rank-adaptive midpoint\\
(Sec. \ref{subsec:adaptive-midpoint}) 
\end{tabular}&
$\Delta t,A,B,G$
&
\begin{tabular}{rl}
$\varepsilon_{{\bm\alpha}}$&$=A\Delta t^3,$\\
$\varepsilon_{{\bm\beta}}$&$=B\Delta t^2,$\\
$\varepsilon_{{\bm\gamma}}$&$=G\Delta t$
\end{tabular}
\\
\hline
\begin{tabular}{c}
Two-step rank-adaptive
Adams-Bashforth\\
(Sec. \ref{subsec:adaptive-adams})
\end{tabular}
&
$\Delta t,A,B,G_0,G_1$
&
\begin{tabular}{rl}
$\varepsilon_{{\bm\alpha}} $&$= A\Delta t^3,$\\
$\varepsilon_{{\bm\beta}} $&$= B\Delta t^2,$\\
$\varepsilon_{{\bm\gamma}(0)} $&$= G_0\Delta t^2$,\\
$\varepsilon_{{\bm\gamma}(1)} $&$= G_1\Delta t^2$
\end{tabular}\\
\hline
\end{tabular}\\
\renewcommand{\arraystretch}{1.0}
\end{centering}
\caption{
\label{table:method-params}
Free and dependent parameters of
the rank-adaptive step-truncation integrators
presented in Section \ref{sec:global-error}.}
\end{table}
To choose such parameters in each numerical example 
we proceed as follows: We 
first choose the time step $\Delta t$ so that the scheme 
without tensor truncation is stable. 
Theorem \ref{thm:st-global-error} then
guarantees convergence of the rank-adaptive  step-truncation
scheme for any selection of the other parameters, e.g.,  
$M_1$ and $M_2$ in the rank-adaptive Euler scheme 
listed in Table \ref{table:method-params}. 
For guidance on how to select the remaining 
parameters one may apply the results 
of section \ref{subsec:const-selection}, which 
are based on the knowledge of the 
singular values of the solution.
An alternative heuristic criterion is to select the 
free parameters roughly inverse to the time step
so that the local error parameters, e.g., $\varepsilon_{\bf r}$ and 
$\varepsilon_{\bf s}$ in Table \ref{table:method-params}, 
do not exceed a specified threshold $\varepsilon^*$, i.e., 
$\varepsilon_{\bf r}\leq \varepsilon^*$ and   
$\varepsilon_{\bf s}\leq \varepsilon^*$. 

\subsection{Rank shock problem}
\label{sec:rank-shock}

\begin{figure}
\centerline{\hspace{0.0cm}
\footnotesize 
\hspace{1cm}Solution Rank \hspace{4.5cm}
Adaptive Euler Error \hspace{0.5cm}}
\centerline{\line(1,0){420}}
\begin{center}
\includegraphics[scale=0.57]{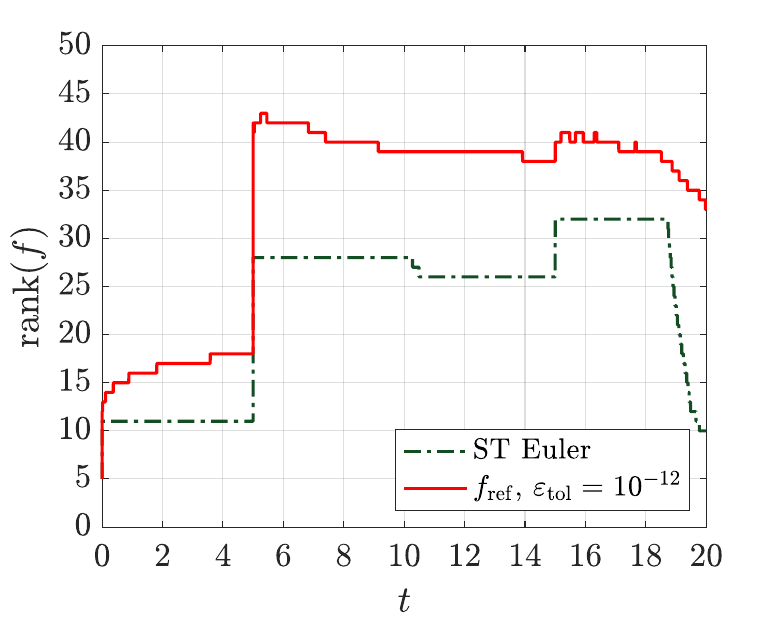}
\includegraphics[scale=0.57]{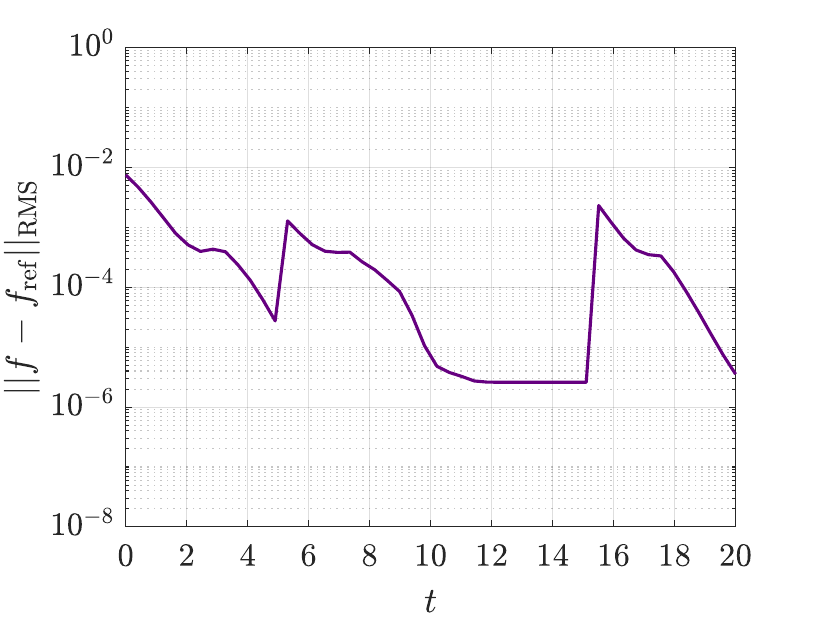}
\end{center}
\caption{
Rank shock problem.
Numerical performance
of rank-adaptive Euler method
applied to the ODE \eqref{eqn:rank-shock-problem}-\eqref{eqn:def-of-switching}. It is seen that the 
method accurately tracks the overall shape 
of the reference solution rank,
{%\color{red}
which was computed to a singular
value threshold of $10^{-12}$.} Moreover, the 
numerical error behaves as expected, 
decreasing as steady-state is approached.
}
\label{fig:rank-shock-result}
\end{figure}
In this section we test ability of the proposed 
rank-adaptive schemes to track accuracy
and rank for a problem where the
rank of the vector field suddenly jumps to a
higher value. To this end,
consider the following matrix-valued 
ordinary differential equation
\begin{equation}
\label{eqn:rank-shock-problem}
\frac{\text{d}{\bf f}}{\text{d}t}
=
{\bf A}{\bf f} + {\bf f}{\bf A}^{\top}
+
{\bf v}(t),
\quad
t\in[0,20],
\quad
{\bf f}(0) \in {\mathbb R}^{N\times N},
\end{equation}
where ${\bf A}$ is a symmetric negative
definite matrix and ${\bf v}(t)$ a forcing term 
that switches between a low rank and high rank
matrix 
\begin{equation}
\label{eqn:def-of-switching}
{\bf v}(t) = 
\begin{cases}
{\bf v}_{\text{high}},
&t\in (5,15)\\
{\bf v}_{\text{low}},
&t\not \in (5,15)
\end{cases}.
\end{equation}
In equation \eqref{eqn:rank-shock-problem} 
${\bf A}{\bf f} + {\bf f}{\bf A}^{\top}$ 
is a stabilizing term which is tangent to the 
fixed rank manifold at all time while 
${\bf v}(t)$ steers the solution 
off of the fixed rank manifold. 
For our numerical experiment 
we let $\bf A$ take the
form 
\begin{equation}
{\bf A}
=
\begin{bmatrix}
-b & a & 0& 0& \dots  &0\\
a  & -b & a & 0& \dots & 0\\
\vdots & & &\ddots  & & \vdots\\
0& \dots & 0&a  & -b & a & \\
0 &\dots& 0 &0& a &-b
\end{bmatrix}\qquad a,b\in \mathbb{R},
\end{equation}
which is a finite difference
stencil with shifted eigenvalues.
We set  $a=1$ and $b=3$
to ensure the matrix $\bf A$ is diagonally dominant
with negative eigenvalues. 
This guarantees that the initial value problem
\eqref{eqn:rank-shock-problem} will
be stable regardless of how large
the $N \times N$ matrix 
size is, for our demonstration we set $N=100$. 
For the forcing term ${\bf v}(t)$ 
we set 
\begin{equation}
{\bf v}_{\text{low}}=
\sum_{j=1}^{r_{\text{low}}}
{\bm \phi}_{j}{\bm \psi}_j^\top,
\qquad
{\bf v}_{\text{high}}=
\sum_{j=1}^{r_{\text{high}}}
\sigma^{j}{\bm \psi}_{j}{\bm \phi}_j^\top,
\end{equation}
with ranks $r_{\text{low}} = 6$ and
$r_{\text{high}}=25$. Here,  ${\bm \psi}_j[i] = 
\sin(2\pi ij/N)$,
${\bm \phi}_j[i] = 
\cos(2\pi ij/N)$ and 
$\sigma^j=(3/4)^j$. 
Since the vector field is discontinuous in
time, we apply the order 1 
rank-adaptive Euler method with
parameters summarized in Table \ref{rank-shock-table}.
\begin{table}
\begin{centering}
\renewcommand{\arraystretch}{1.3}
\begin{tabular}{ |c|c|c|}
\hline
 Integration Method &
 Free Parameters &
 Dependent Parameters\\
\hline
\begin{tabular}{c}
Rank-adaptive Euler
\\
(Sec. \ref{subsec:adaptive-euler})
\end{tabular}
 &
\begin{tabular}{c}
 $\Delta t=2\times 10^{-3},$\\
 $M_1=M_2=10^2$ 
\end{tabular}  &
 \begin{tabular}{rl}
 $\varepsilon_{\bf r}$&$=4\times 10^{-4},$\\
$\varepsilon_{\bf s}$&$=2\times 10$
 \end{tabular}\\
\hline
\end{tabular}\\
\renewcommand{\arraystretch}{1.0}
\end{centering}
\caption{Integration parameters for the rank shock problem \eqref{eqn:rank-shock-problem}.}
\label{rank-shock-table}
\end{table}
For quantification of the numerical error, 
we use the root mean square error of matrices (Frobenious norm)
\begin{equation}
\label{eqn:def-MSE}
\|
{\bf g} - {\bf f}
\|_{\text{RMS}}
=
\sqrt{
\sum_{i=1}^{N} \sum_{j=1}^{N}
\frac{({\bf g}[i,j]-{\bf f}[i,j])^2}
{N^2}
}.
\end{equation}
To obtain a reference solution ${f}_{\text{ref}}$ we 
simply integrate \eqref{eqn:rank-shock-problem} using RK4.
As seen in Figure \ref{fig:rank-shock-result}, the numerical 
solution successfully tracks the overall
shape of the reference solution's rank over time.
The numerical error also behaves as
expected, decreasing as a steady-state is approached.
}

\subsection{Fokker-Planck equation}
\label{subsec:fokker-planck}

In this section we apply the proposed 
rank-adaptive step-truncation 
algorithms to a Fokker-Planck equation with 
space-dependent drift and constant diffusion, 
and demonstrate their accuracy in predicting 
relaxation to statistical equilibrium.
As is well-known, the Fokker-Planck equation describes the evolution 
of the probability density function (PDF) of the state vector
solving the It\^o stochastic differential equation (SDE)
\begin{equation}
\label{Ito_SDE}
d \bm X_t = \bm \mu(\bm X_t)dt + \sigma d\bm W_t.
\end{equation}
Here, $\bm X_t$ is the $d$-dimensional state vector, $\bm \mu(\bm X_t)$ is the $d$-dimensional drift, $\sigma$ is a constant drift coefficient and $\bm W_t$ is an $d$-dimensional standard 
Wiener process. The Fokker-Planck equation that 
corresponds to \eqref{Ito_SDE} has the form 
\begin{equation}
\label{fp-lorenz-96}
\frac{{\partial} f({\bf x},t) }{\partial t}=
-\sum_{i=1}^{d}
\frac{\partial}{\partial x_i}\left(
\mu_i({\bf x}) f({\bf x},t) \right)
+ \frac{\sigma^2}{2} 
\sum_{i=1}^{d}\frac{\partial^2 f({\bf x},t)}{\partial x_i^2}, \qquad f(\bm x, 0) = f_0(\bm x),
\end{equation}
where $f_0(\bm x)$ is the PDF of the initial state $\bm X_0$.
In our numerical demonstrations, we set $\sigma = 2$,  
\begin{equation}
\mu_i({\bf x}) = 
(\gamma(x_{i+1}) - \gamma(x_{i-2}))\xi(x_{i-1}) - \phi(x_{i}), \qquad i =1,\ldots, d,
\label{82}
\end{equation}
where the functions
$\gamma(x)$, $\xi(x)$,
and $\phi(x)$ are $2 \pi$-periodic. Also, in \eqref{82} 
$x_{i+d}=x_i$. 
We solve \eqref{fp-lorenz-96} on the
flat torus $\Omega=[0,2\pi]^d$ with  dimension 
$d=2$ and $d=4$.

\begin{figure}
\centering
\centerline{\hspace{1.5cm}
\footnotesize 
Adaptive Euler \hspace{2.0cm}
Adaptive AB2    \hspace{2.5cm}
Reference}
\centerline{\line(1,0){420}}
\vspace{0.1cm}
{\rotatebox{90}{\hspace{1.6cm}\rotatebox{-90}{
\hspace{0.1cm}
\footnotesize$t=0$\hspace{0.9cm}}}}
\includegraphics[scale=0.27]{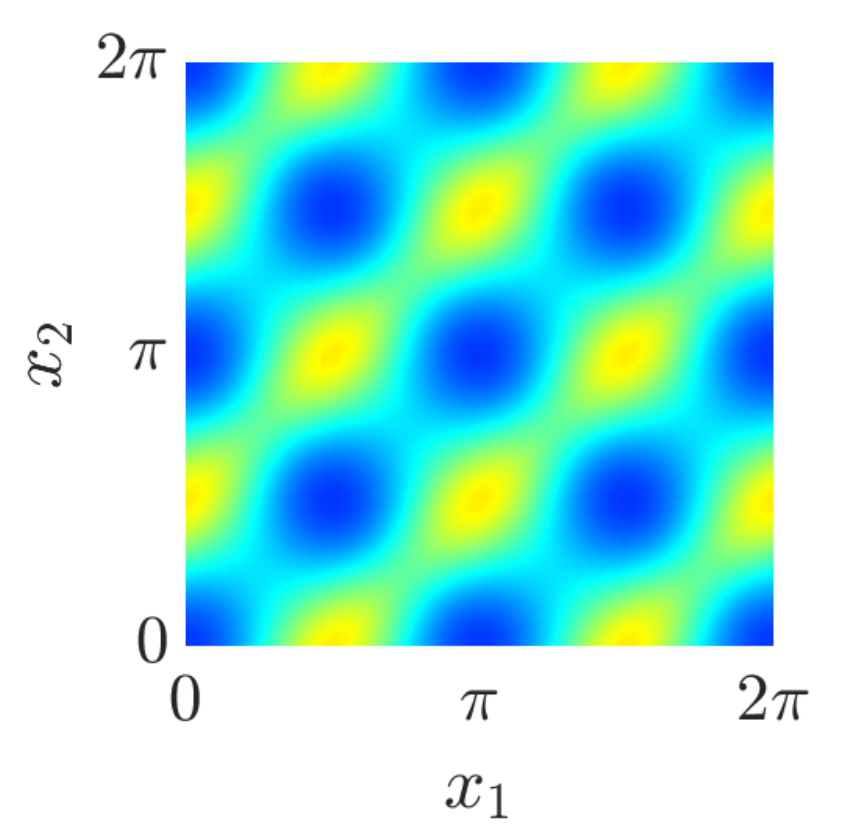}
\includegraphics[scale=0.27]{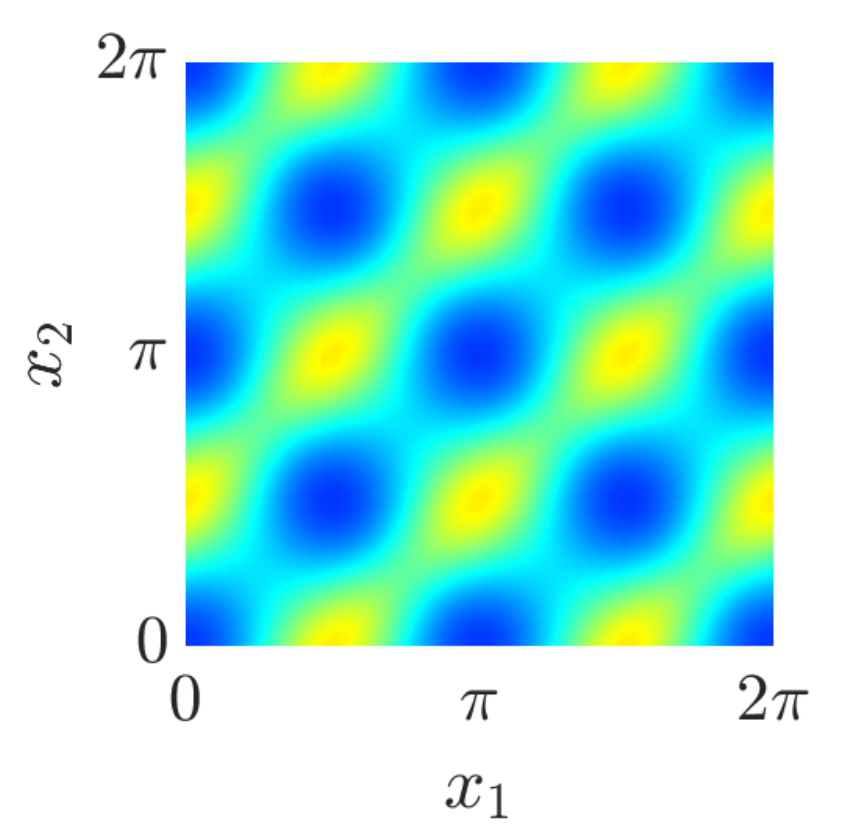}
\includegraphics[scale=0.27]{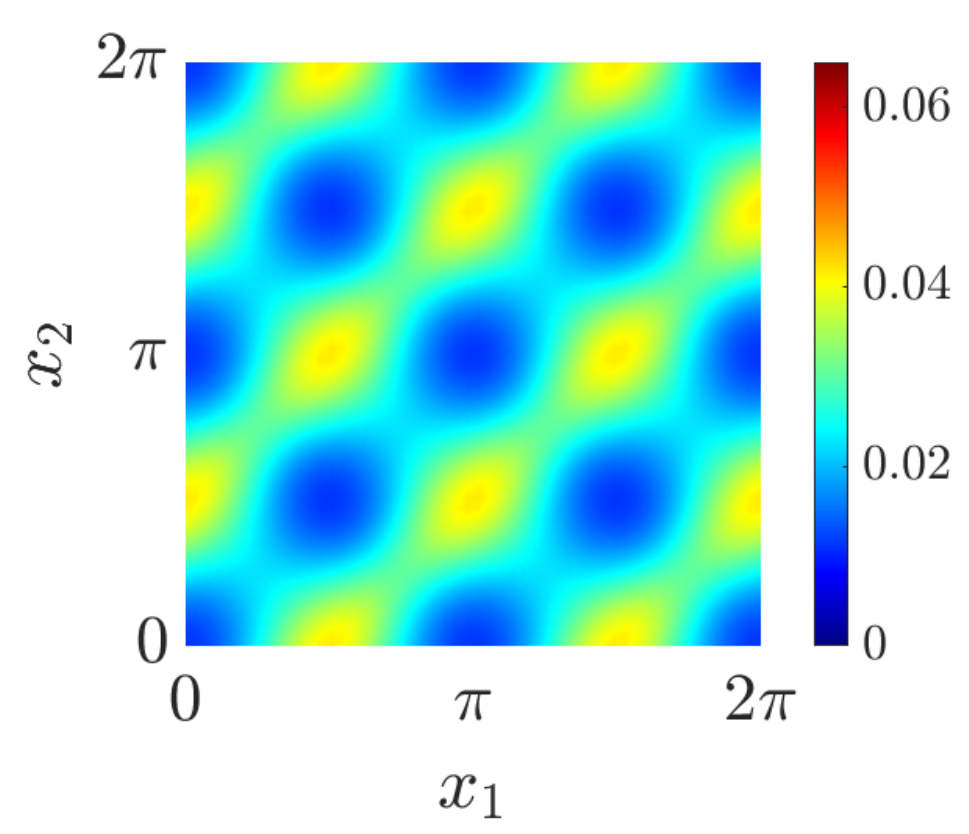}\\
{\rotatebox{90}{\hspace{1.6cm}\rotatebox{-90}{
\footnotesize$t=0.05$\hspace{0.7cm}}}}
\includegraphics[scale=0.27]{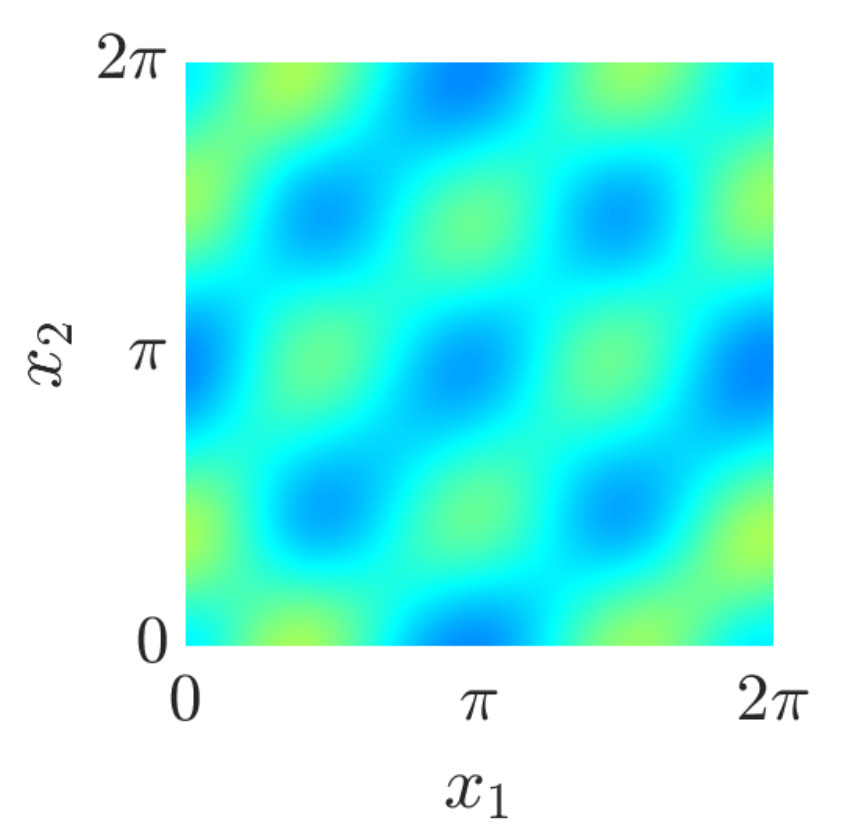}
\includegraphics[scale=0.27]{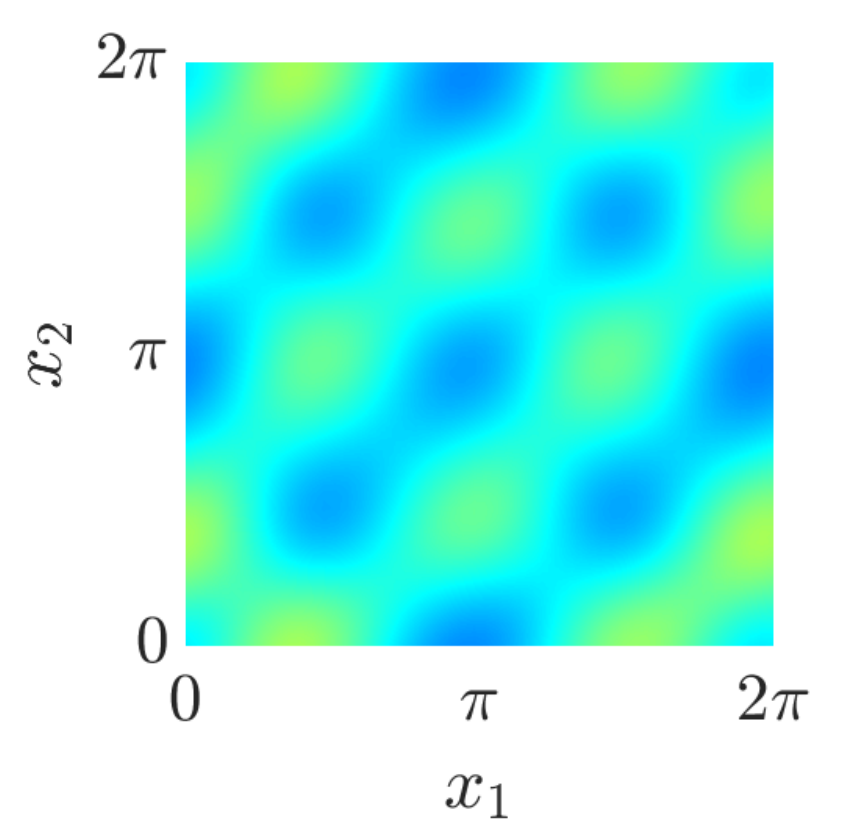}
\includegraphics[scale=0.27]{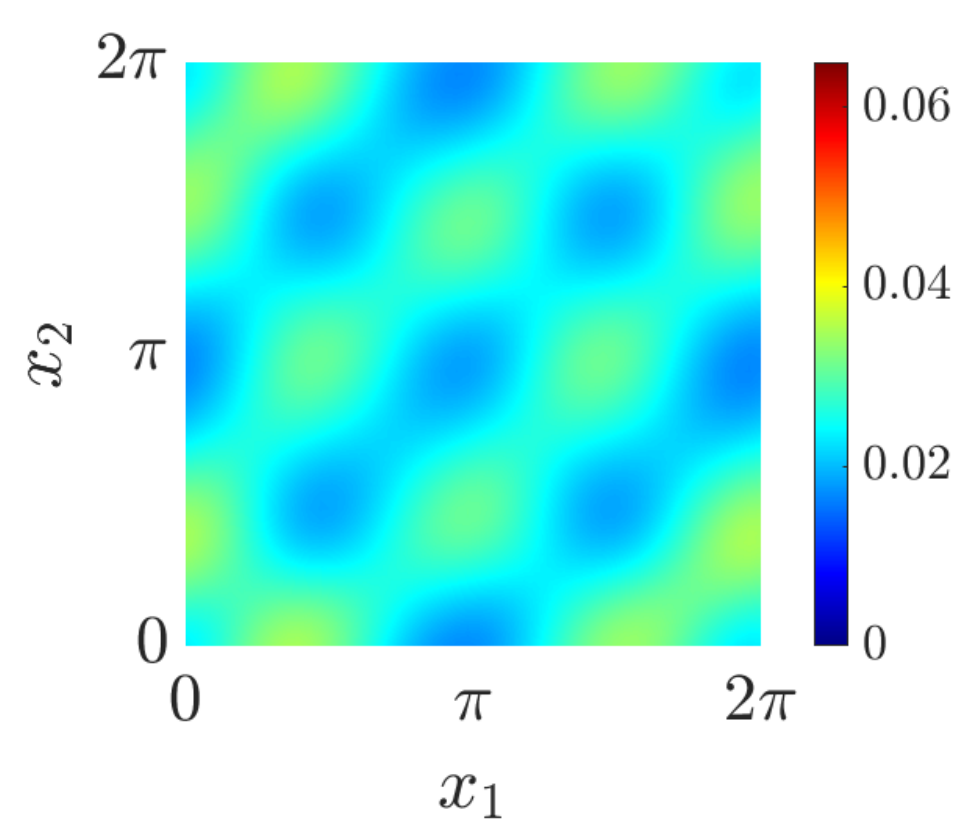}\\
{\rotatebox{90}{\hspace{1.6cm}\rotatebox{-90}{
\footnotesize$t=0.15$\hspace{0.7cm}}}}
\includegraphics[scale=0.27]{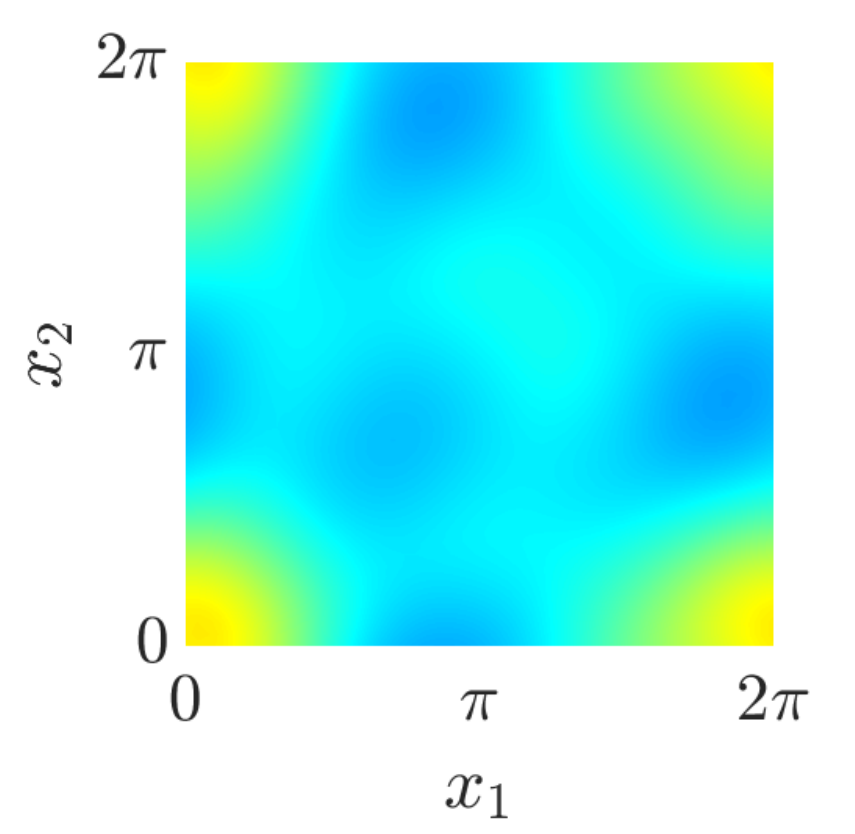}
\includegraphics[scale=0.27]{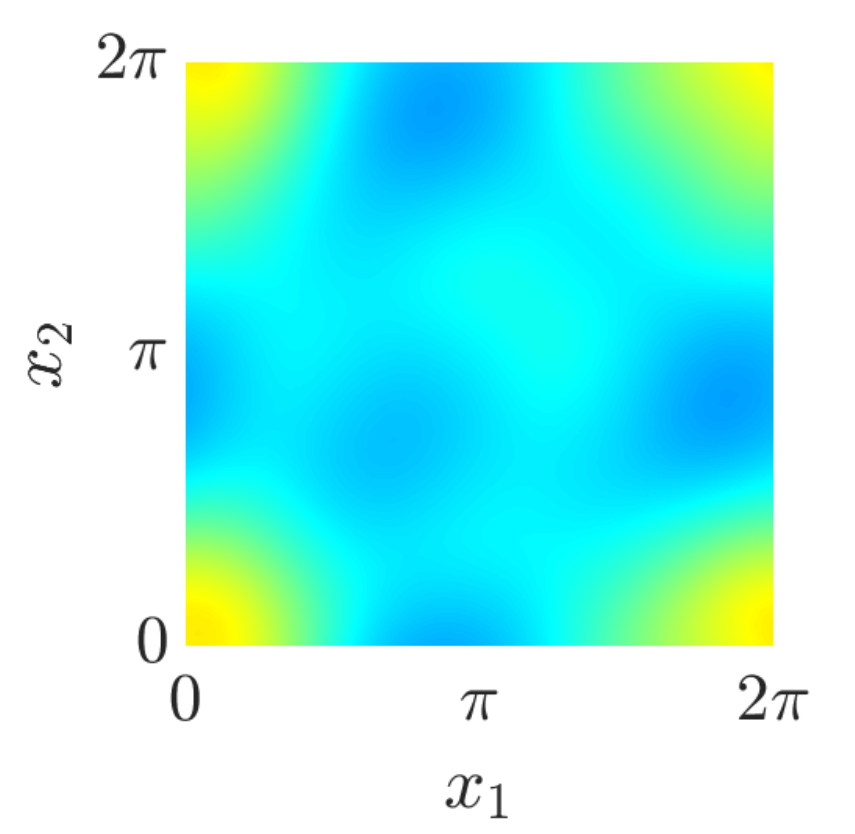}
\includegraphics[scale=0.27]{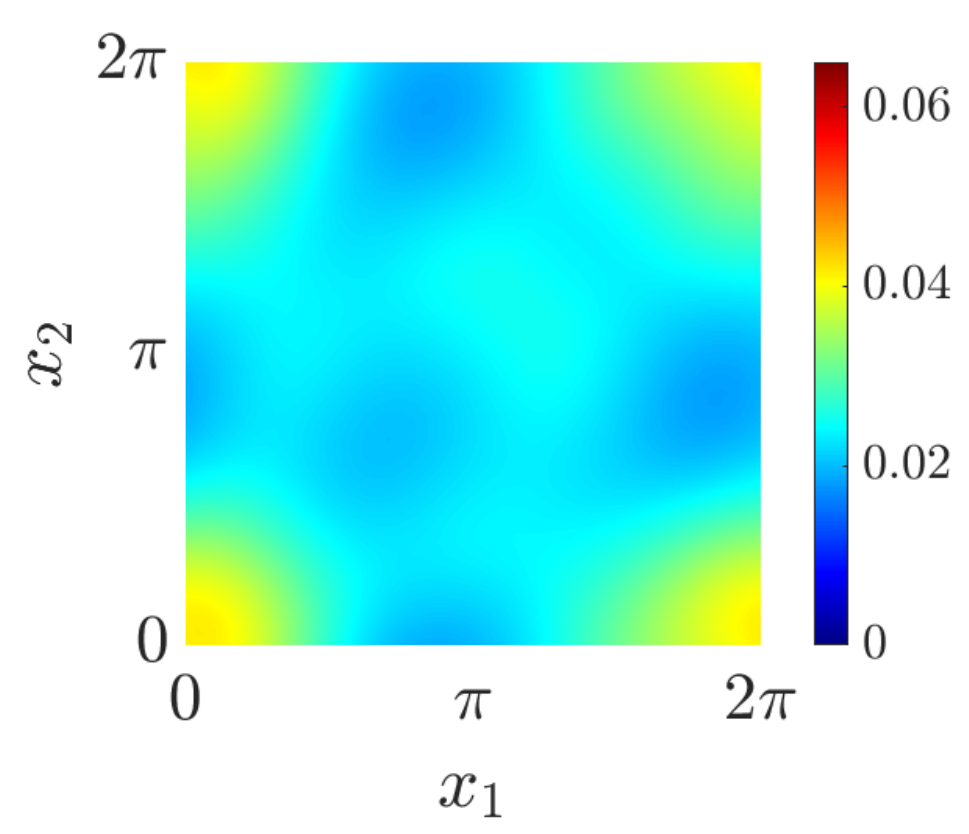}\\
{\rotatebox{90}{\hspace{1.6cm}\rotatebox{-90}{
\footnotesize$t=0.25$\hspace{0.7cm}}}}
\includegraphics[scale=0.27]{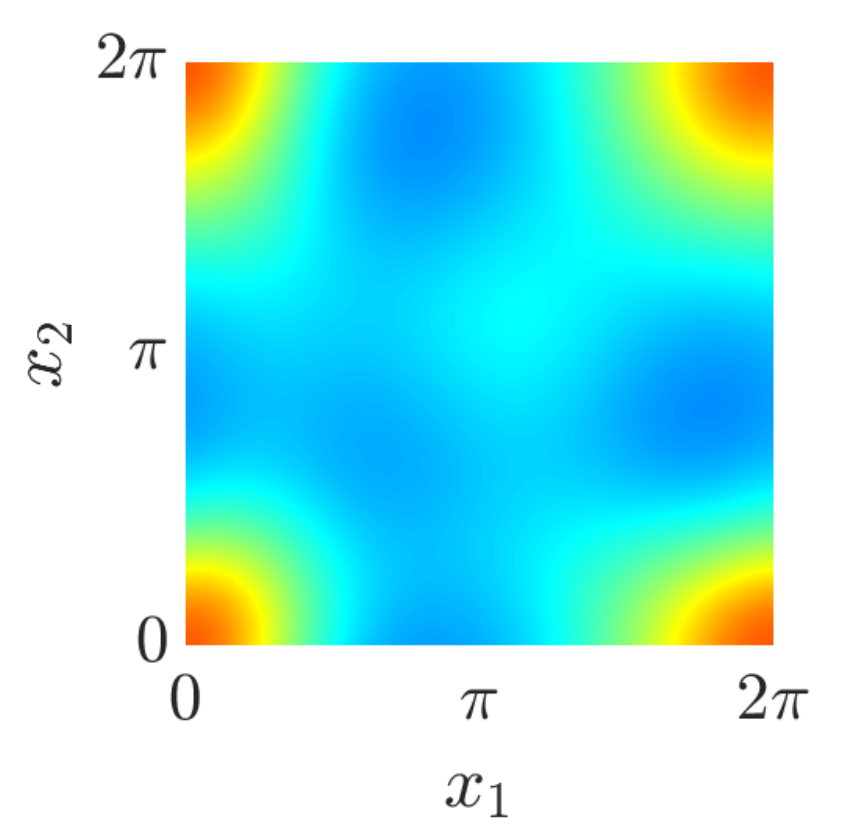}
\includegraphics[scale=0.27]{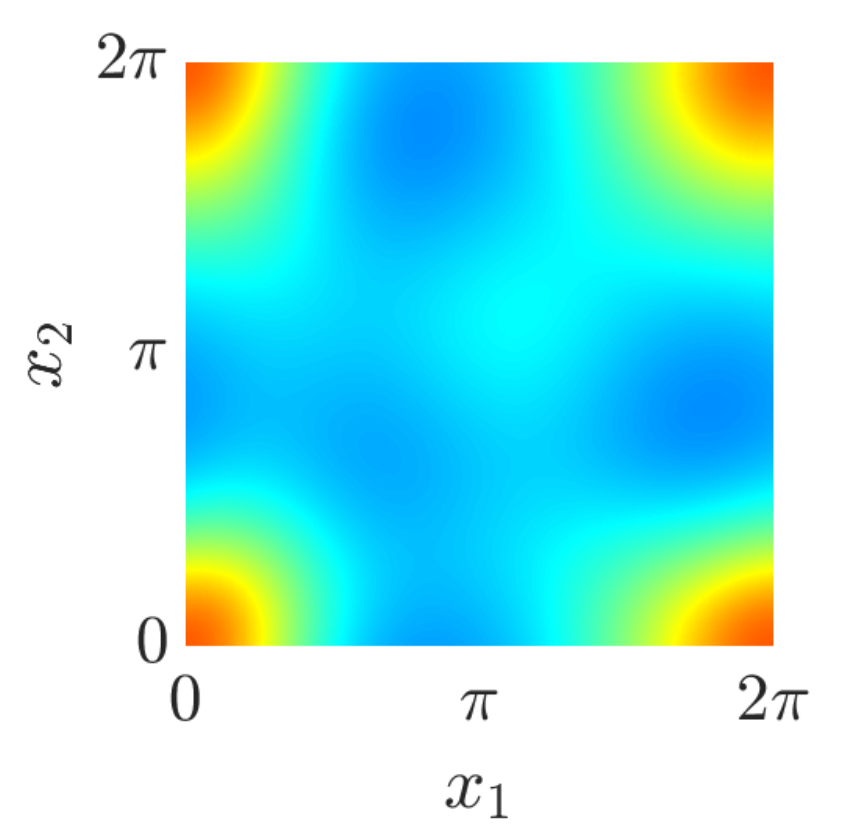}
\includegraphics[scale=0.27]{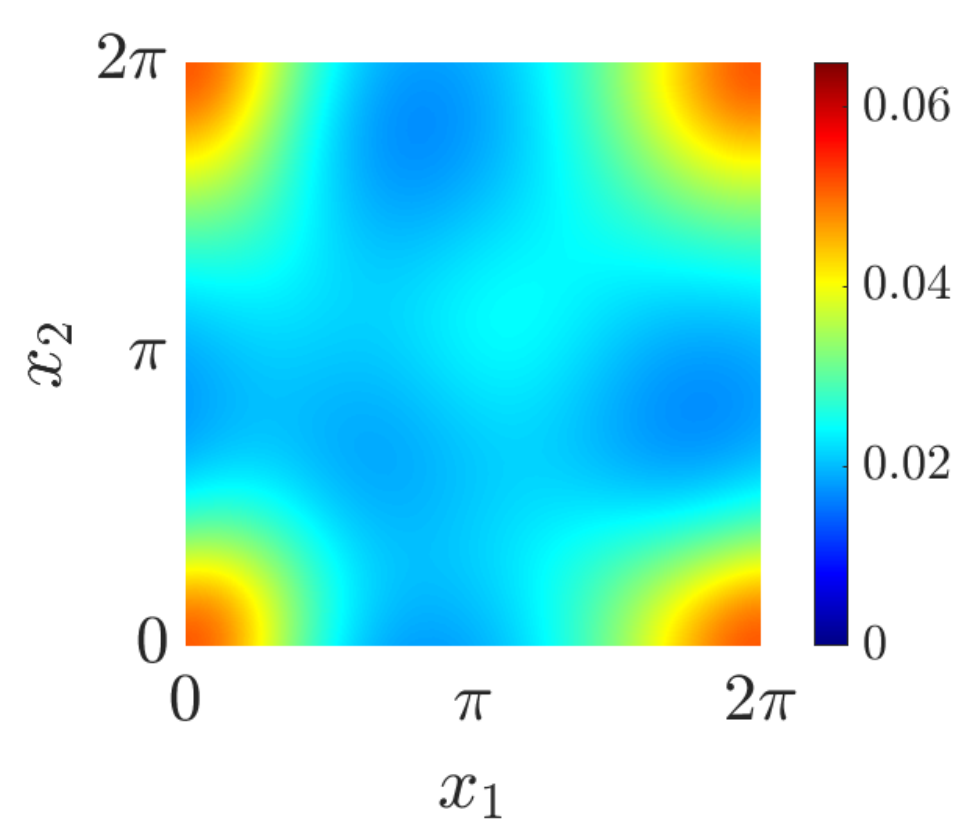}\\
{\rotatebox{90}{\hspace{1.6cm}\rotatebox{-90}{
\footnotesize{Steady State}\hspace{0.1cm}}}}
\includegraphics[scale=0.27]{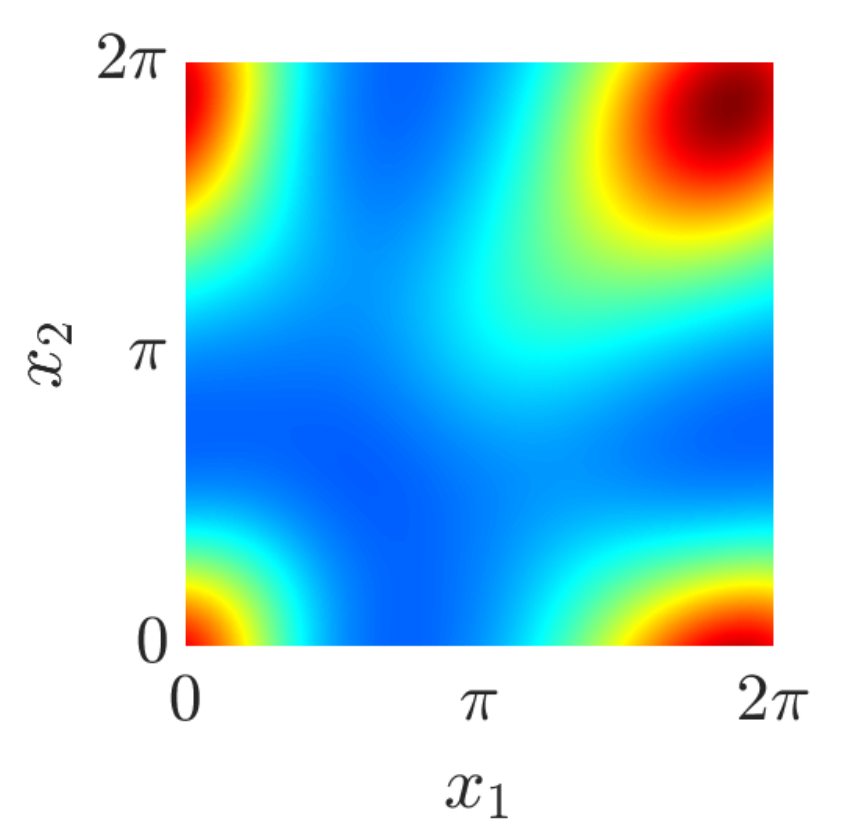}
\includegraphics[scale=0.27]{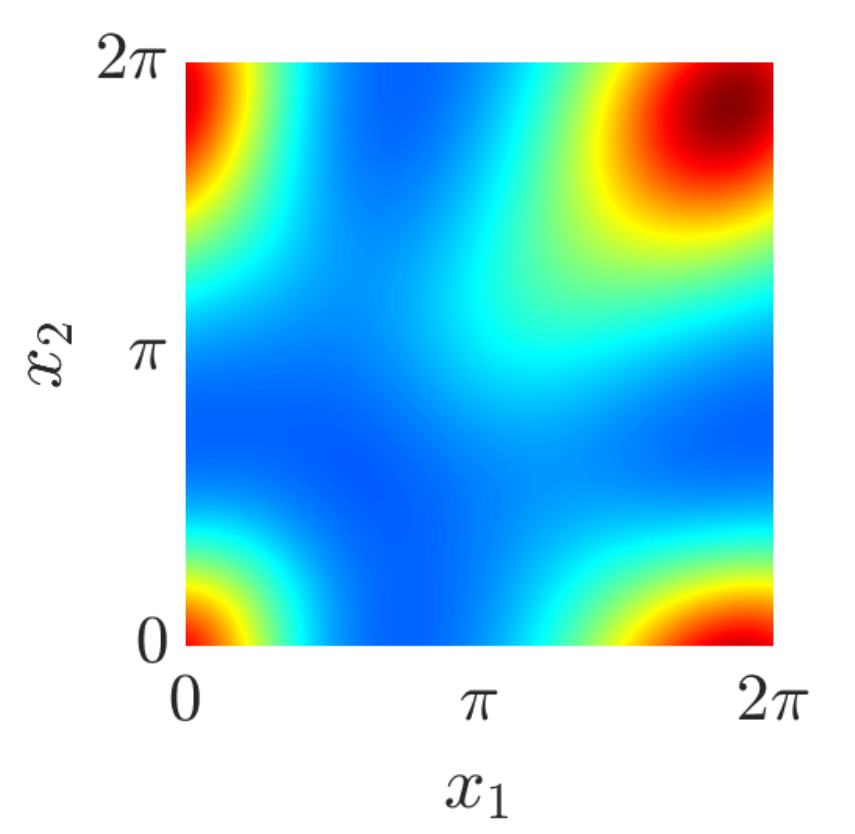}
\includegraphics[scale=0.27]{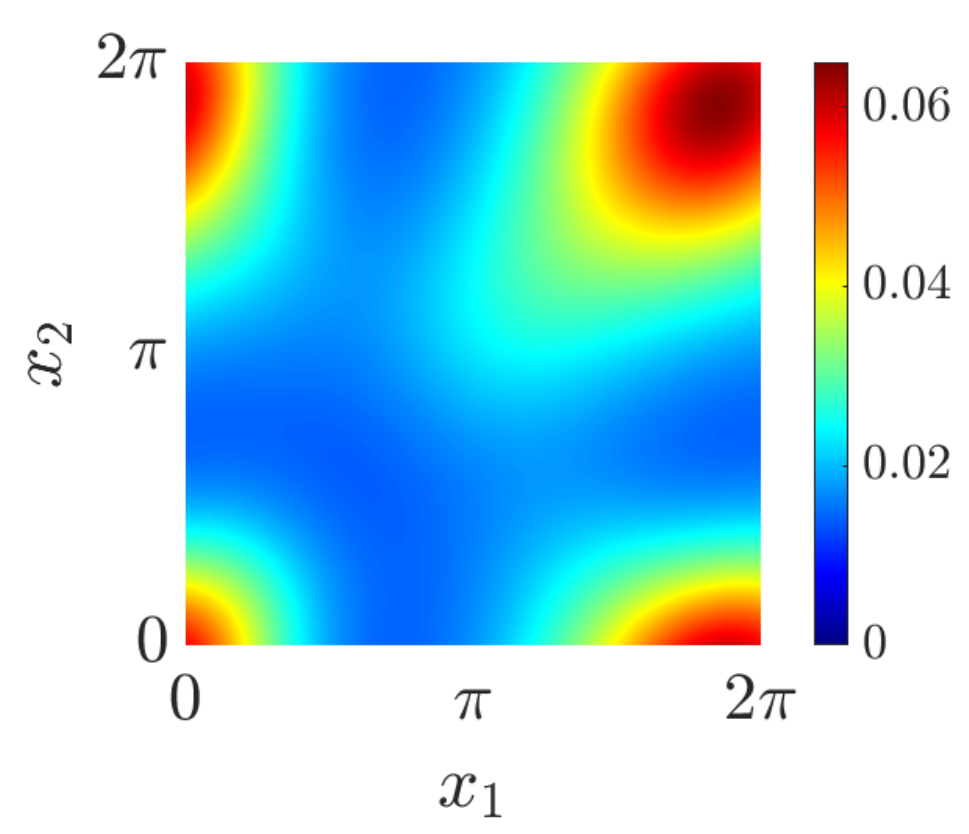}\\
\caption{Numerical solution to the Fokker-Planck equation 
\eqref{fp-lorenz-96} in dimension $d=2$ with 
initial condition \eqref{ic2d} obtained 
using three distinct methods: 
rank-adaptive explicit Euler
\eqref{adaptive-euler-method},
two-step rank-adaptive Adams-Bashforth \eqref{adaptive-ab}, 
and a reliable reference solution 
obtained by solving the ODE \eqref{mol-ode} 
corresponding to \eqref{fp-lorenz-96}.
{%\color{red}
The numerical results
are obtained on a $50\times 50$ spatial grid.}
The parameters for the step-truncation integrators we used in this example are detailed in Table \ref{fig:2d-tables}.
}
\label{fig:time-snapshot-plot}
\end{figure}

\begin{table}
\begin{centering}
\renewcommand{\arraystretch}{1.3}
\begin{tabular}{ |c|c|c|}
\hline
 Integration Method &
 Free Parameters &
 Dependent Parameters\\
\hline
\begin{tabular}{c}
Adaptive Euler
\\
(Sec. \ref{subsec:adaptive-euler})
\end{tabular}
&
\begin{tabular}{c}
$\Delta t=6.25\times 10^{-4},$ \\
$M_1=M_2=10^2$
\end{tabular}
 &
 \begin{tabular}{rl}
 $\varepsilon_{\bf r}$&$=3.90625\times 10^{-5},$\\
$\varepsilon_{\bf s}$&$=6.25\times 10^{-2}$
 \end{tabular}\\
\hline
\begin{tabular}{c}
Adaptive Midpoint\\
(Sec. \ref{subsec:adaptive-midpoint}) 
\end{tabular}&
\begin{tabular}{c}
$\Delta t=6.25\times 10^{-4},$\\
$A=B=10^3,$\\
$G=10^2$
\end{tabular}
&
\begin{tabular}{rl}
$\varepsilon_{{\bm\alpha}}$&$=
2.44140625\times 10^{-7},$\\
$\varepsilon_{{\bm\beta}}$&$=
3.90625\times 10^{-4},$\\
$\varepsilon_{{\bm\gamma}}$&$=
6.25\times 10^{-2}$
\end{tabular}
\\
\hline
\begin{tabular}{c}
Two-step rank-adaptive
Adams-Bashforth\\
(Sec. \ref{subsec:adaptive-adams})
\end{tabular}
&

\begin{tabular}{c}
$\Delta t=6.25\times 10^{-4},$\\
$A=B=10^3,$\\
$G_0=G_1=10^2$
\end{tabular}
&
\begin{tabular}{rl}
$\varepsilon_{{\bm\alpha}}$&$=
2.44140625\times 10^{-7},$\\
$\varepsilon_{{\bm\beta}}$&$=
3.90625\times 10^{-4},$\\
$\varepsilon_{{\bm\gamma}(0)}$&$=
3.90625\times 10^{-5},$\\
$\varepsilon_{{\bm\gamma}(1)}$&$=
3.90625\times 10^{-5},$
\end{tabular}
\\
\hline
\end{tabular}\\
\renewcommand{\arraystretch}{1.0}
\end{centering}
\caption{
Table of parameters for the rank-adaptive 
step-truncation integrators of the 
Fokker-Planck equation
\eqref{fp-lorenz-96} in dimension $d=2$ with 
initial condition \eqref{ic2d}.
The only free parameters are
the local error coefficients.
These were heuristically chosen so that
the truncation
at each step (to rank 
$\bf r$ or $\bm \alpha$)
would be considerably smaller than
the time step.
}
\label{fig:2d-tables}
\end{table}

\begin{figure}
\begin{center}
\includegraphics[scale=0.57]{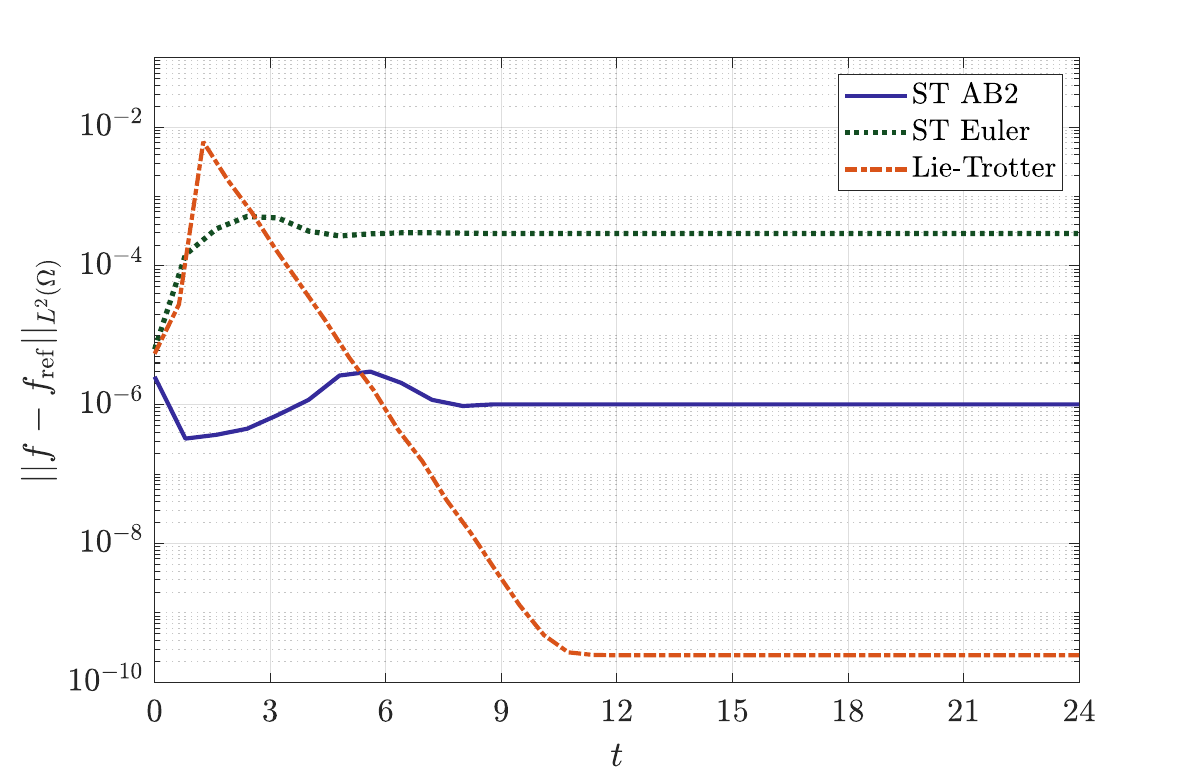}
\end{center}
\caption{Fokker-Planck equation 
\eqref{fp-lorenz-96} in dimension $d=2$ with 
initial condition \eqref{ic2d}. $L^2(\Omega)$ 
error of rank-adaptive Euler forward, rank-adaptive 
AB2, and rank-adaptive Lie-Trotter \cite{dektor2020rankadaptive} 
(with normal vector threshold $10^{-4}$) 
solutions with respect to the reference solution.
{%\color{red}
The numerical results
are obtained on a $50\times 50$ spatial grid.}
}
\label{fig:error-compare-2d}
\end{figure}
\begin{figure}
\centerline{
\includegraphics[scale=0.57]{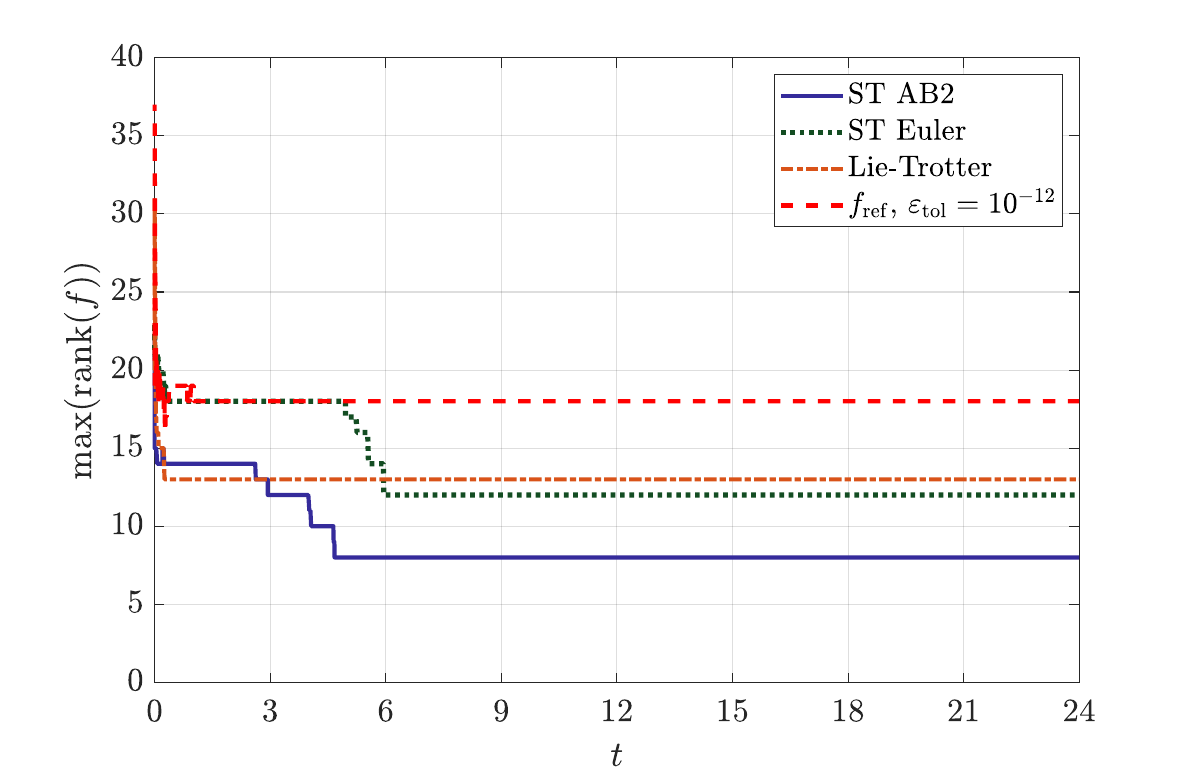}
}
\caption{
Fokker-Planck equation 
\eqref{fp-lorenz-96} in dimension $d=2$ with 
initial condition \eqref{ic2d}.  
Rank versus time for rank-adaptive step-truncation Euler forward, 
AB2, rank-adaptive Lie-Trotter with normal vector threshold $10^{-4}$ \cite{dektor2020rankadaptive}, and reference
numerical solutions.
{%\color{red}
The numerical results
are obtained on a $50\times 50$ spatial grid.
The reference solution 
rank was computed with a singular value tolerance
of $\varepsilon_{\rm tol}^{-12}$.}
}
\label{fig:rank-compare-2d}
\end{figure}

\begin{figure}
\centerline{ 
\includegraphics[scale=0.69]{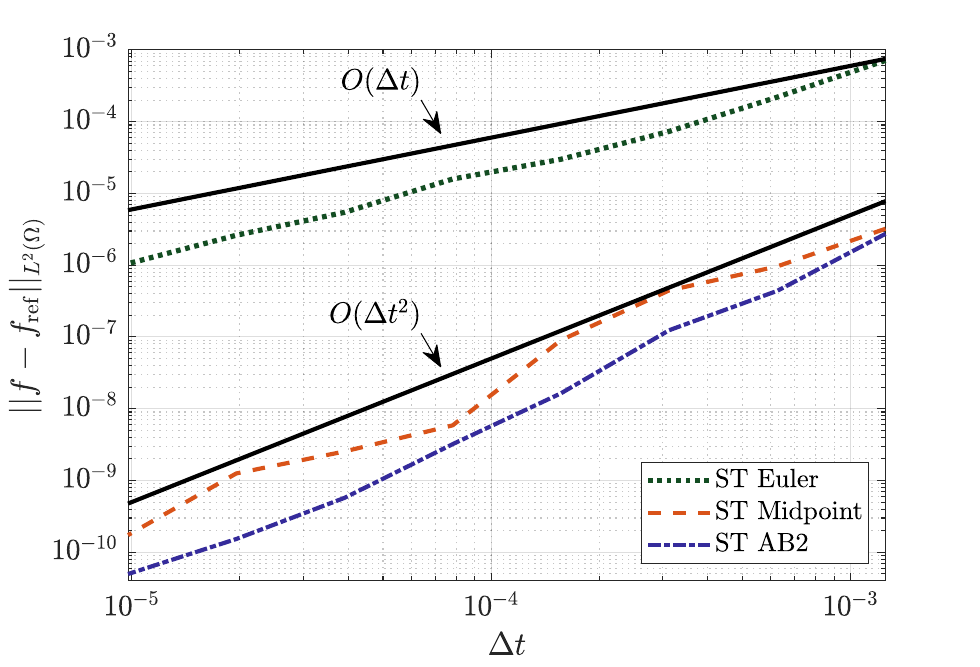}
}
%\centerline{
%\includegraphics[scale=0.57]{./global_err_2d/adaptive_euler_global}
%}
\caption{Fokker-Planck equation 
\eqref{fp-lorenz-96} in dimension $d=2$ with 
initial condition \eqref{ic2d}. $L^2(\Omega)$ error at $T=1$ for 
the rank-adaptive step-truncation methods summarized in Table \ref{fig:2d-tables}.
{%\color{red}
The numerical results
are obtained on a $40\times 40$ spatial grid.}
}
\label{fig:global-err-2d}
\end{figure}

\begin{figure}
\centering
\centerline{\hspace{1.5cm}
\footnotesize 
Adaptive Euler \hspace{2.0cm}
Adaptive AB2    \hspace{2.5cm}
Reference}
\centerline{\line(1,0){420}}
\vspace{0.1cm}
{\rotatebox{90}{\hspace{1.6cm}\rotatebox{-90}{
\hspace{0.1cm}
\footnotesize$t=0$\hspace{0.9cm}}}}
\includegraphics[scale=0.27]{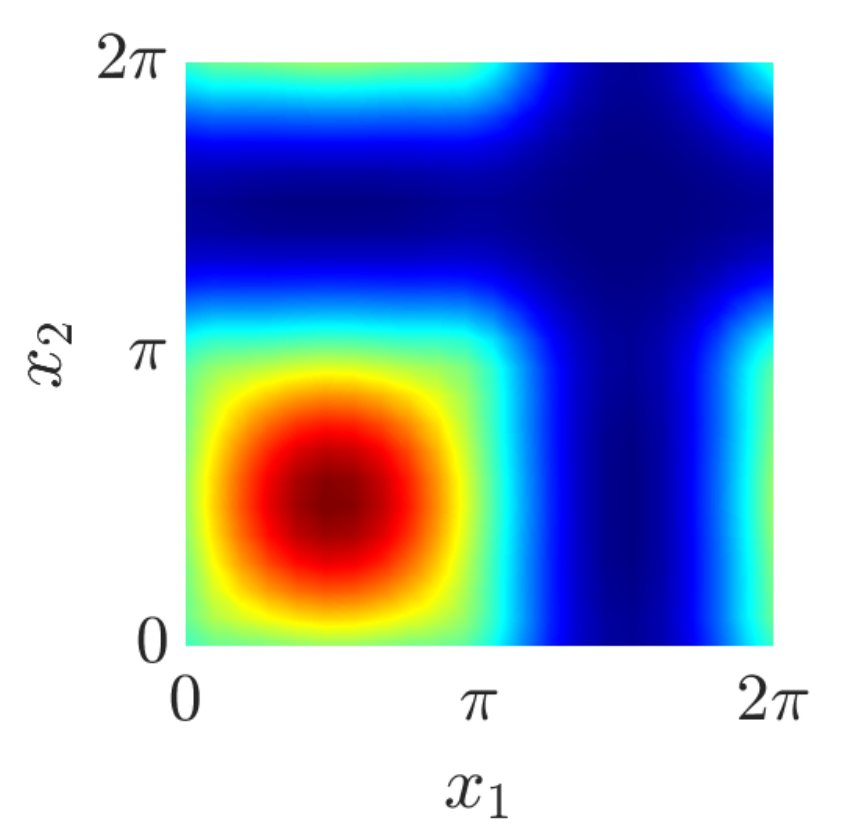}
\includegraphics[scale=0.27]{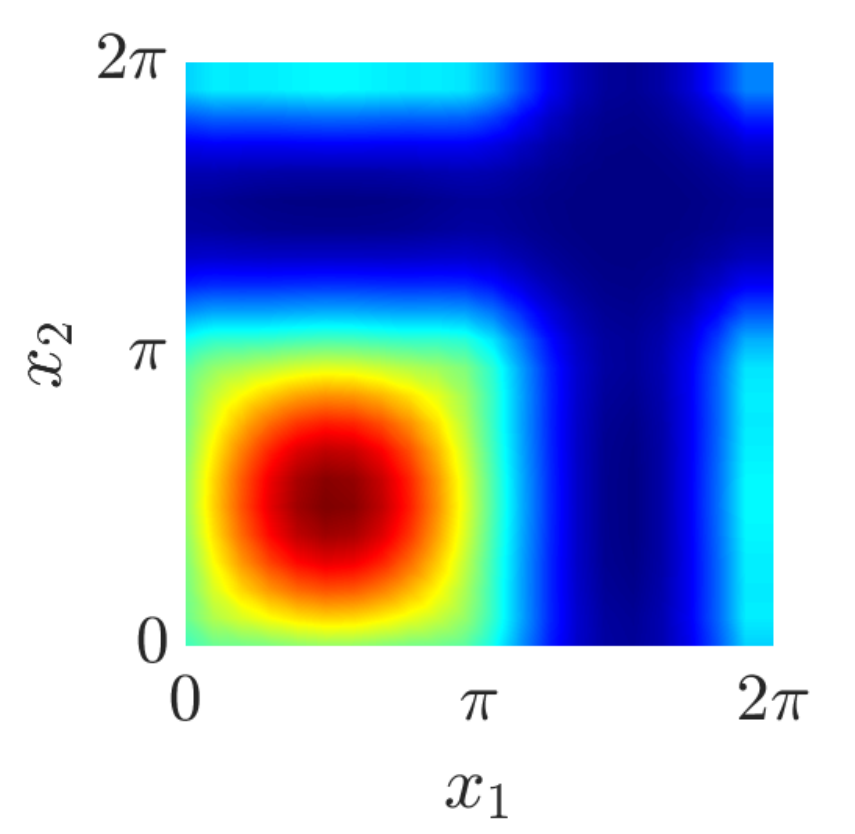}
\includegraphics[scale=0.27]{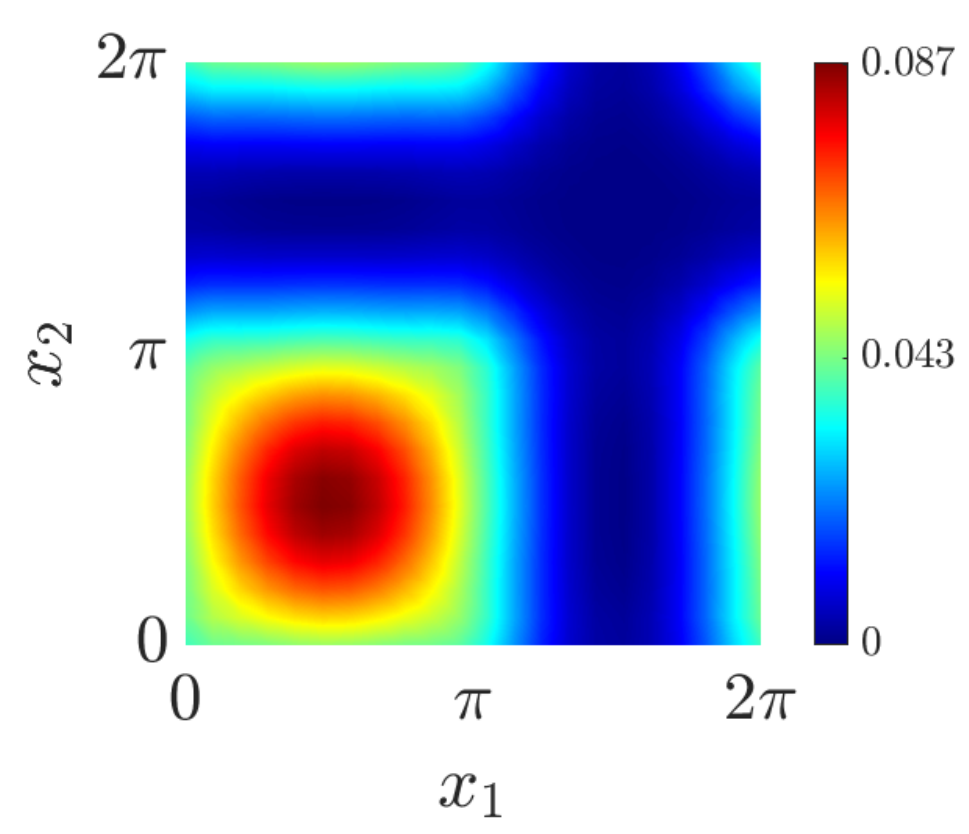}\\
{\rotatebox{90}{\hspace{1.6cm}\rotatebox{-90}{
\footnotesize$t=0.3$\hspace{0.7cm}}}}
\includegraphics[scale=0.27]{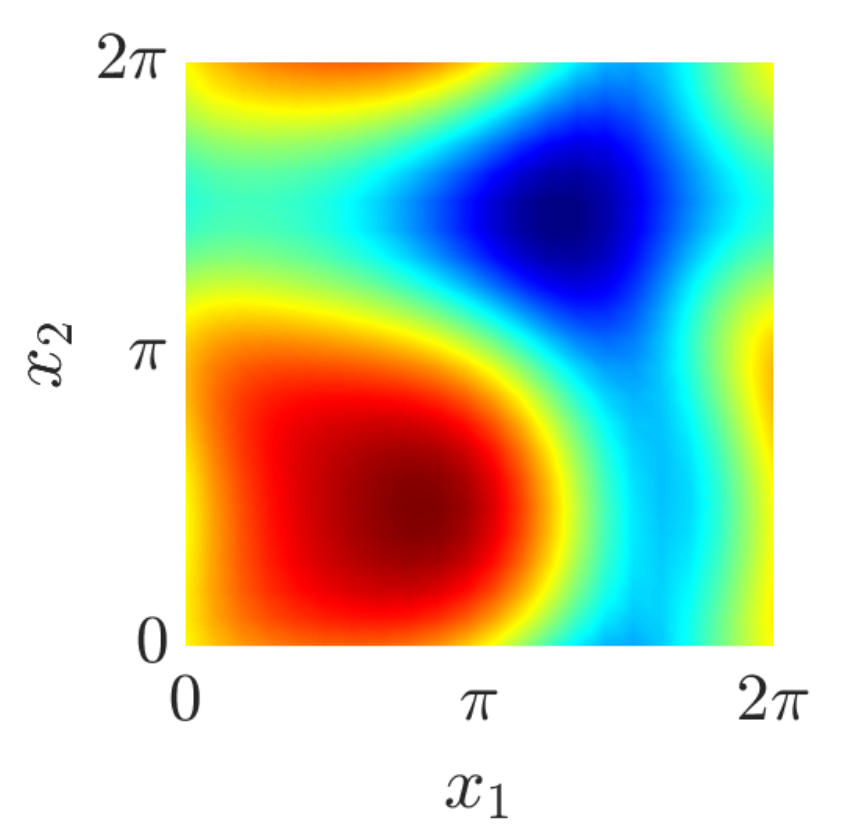}
\includegraphics[scale=0.27]{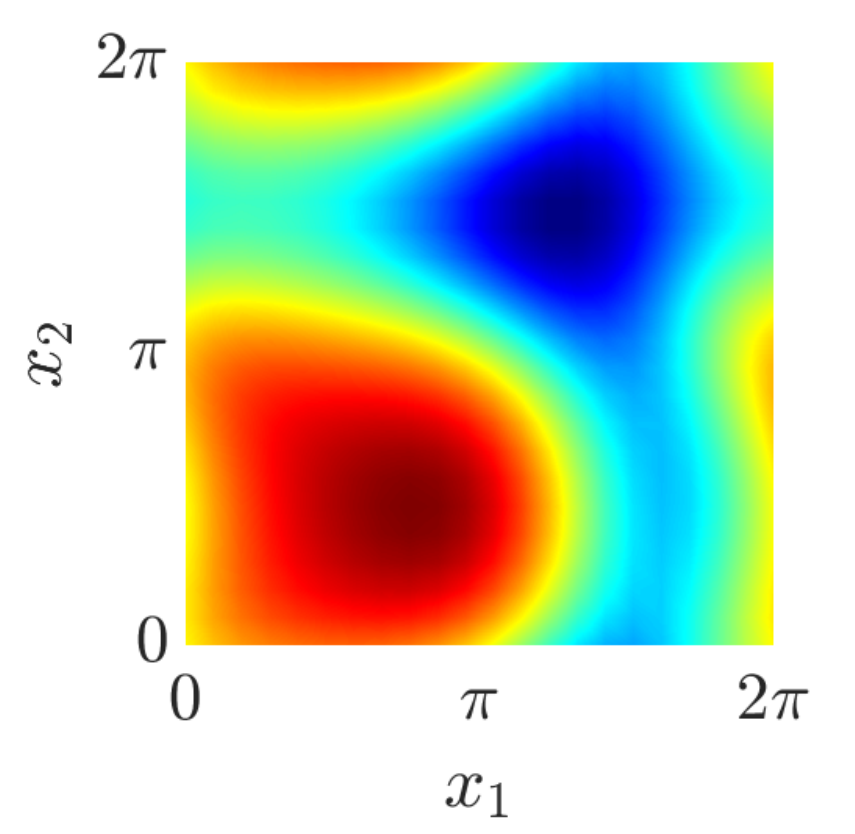}
\includegraphics[scale=0.27]{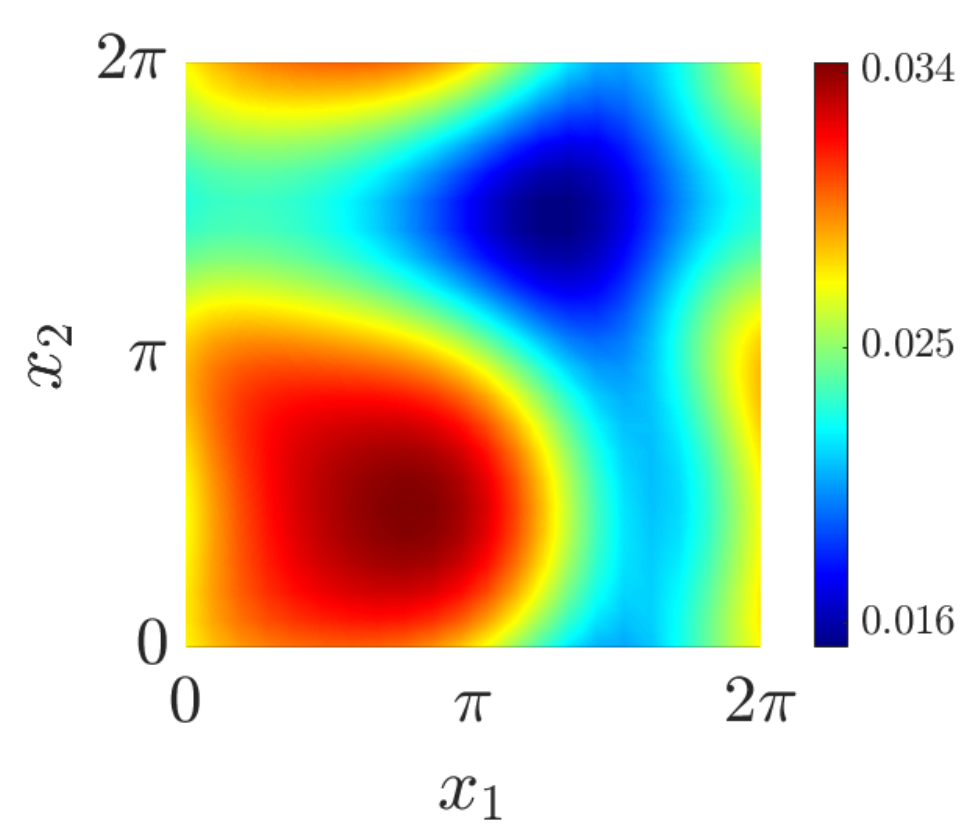}\\
{\rotatebox{90}{\hspace{1.6cm}\rotatebox{-90}{
\footnotesize$t=0.4$\hspace{0.7cm}}}}
\includegraphics[scale=0.27]{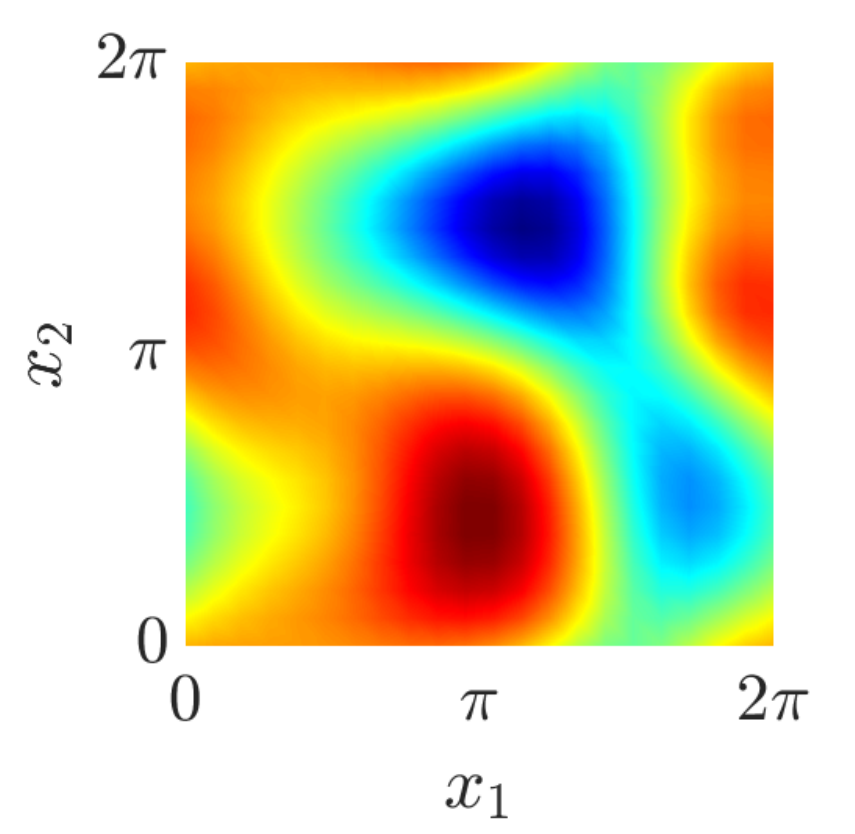}
\includegraphics[scale=0.27]{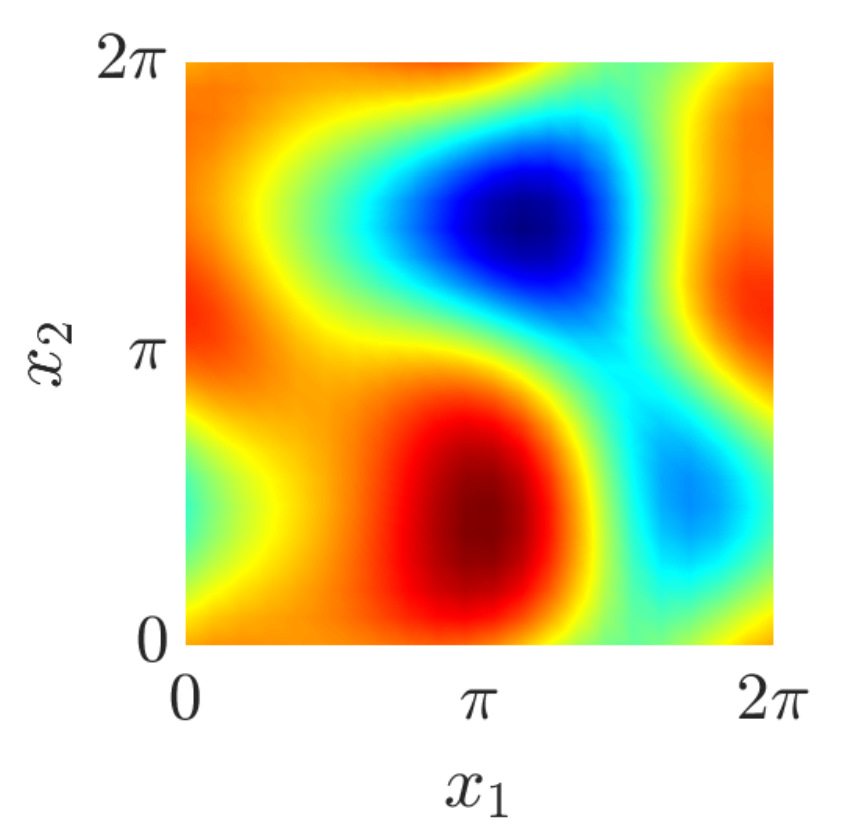}
\includegraphics[scale=0.27]{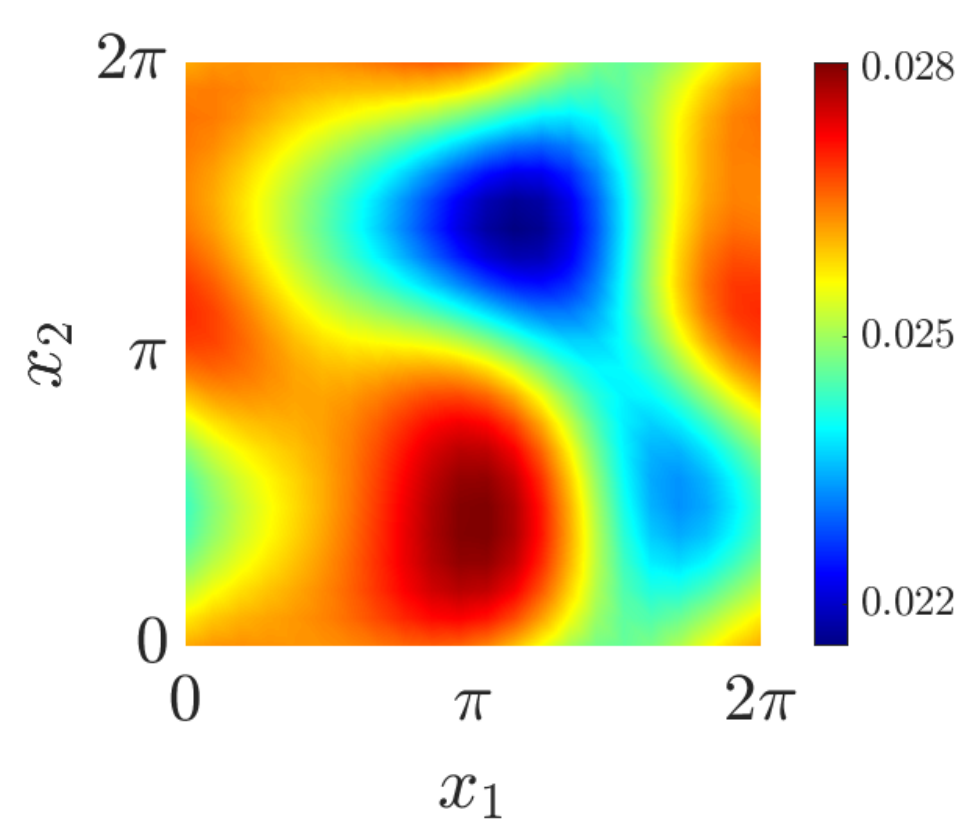}\\
{\rotatebox{90}{\hspace{1.6cm}\rotatebox{-90}{
\footnotesize$t=0.47$\hspace{0.7cm}}}}
\includegraphics[scale=0.27]{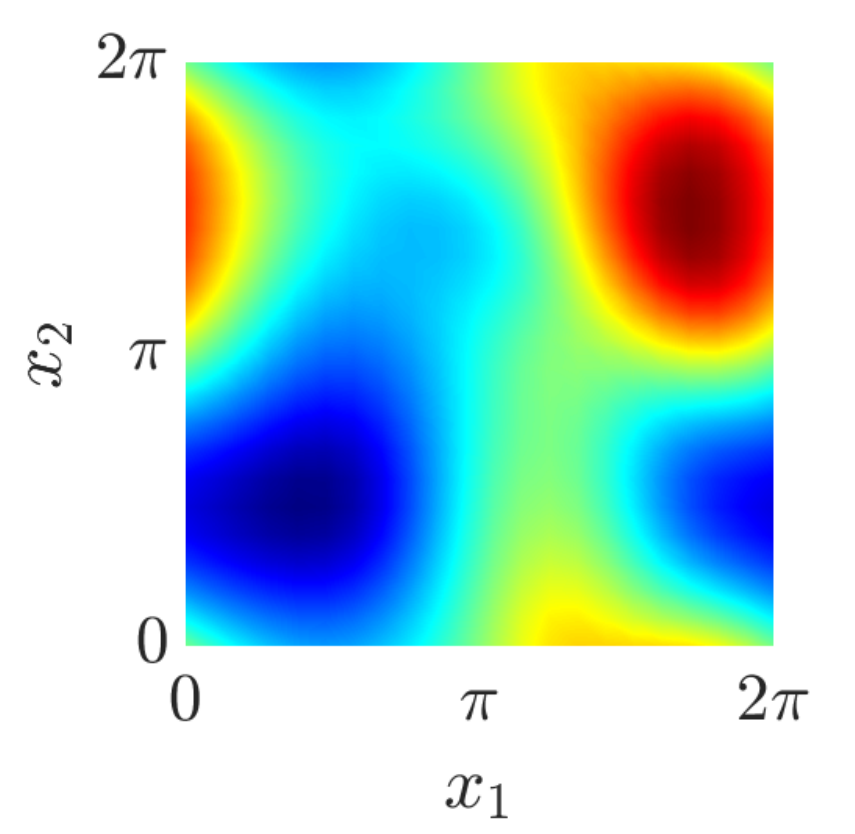}
\includegraphics[scale=0.27]{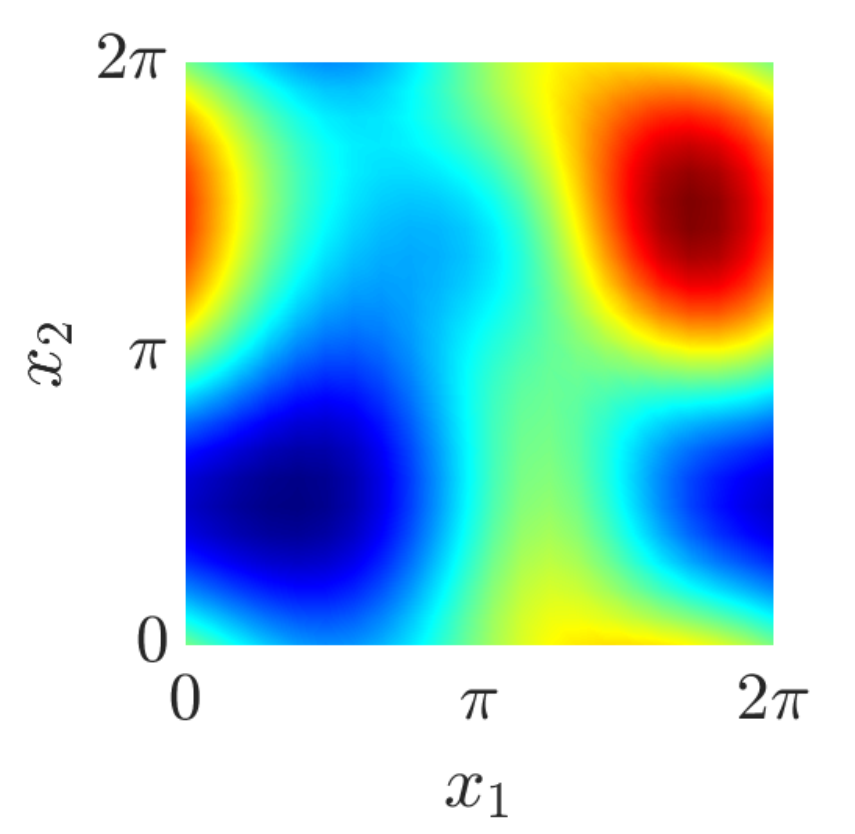}
\includegraphics[scale=0.27]{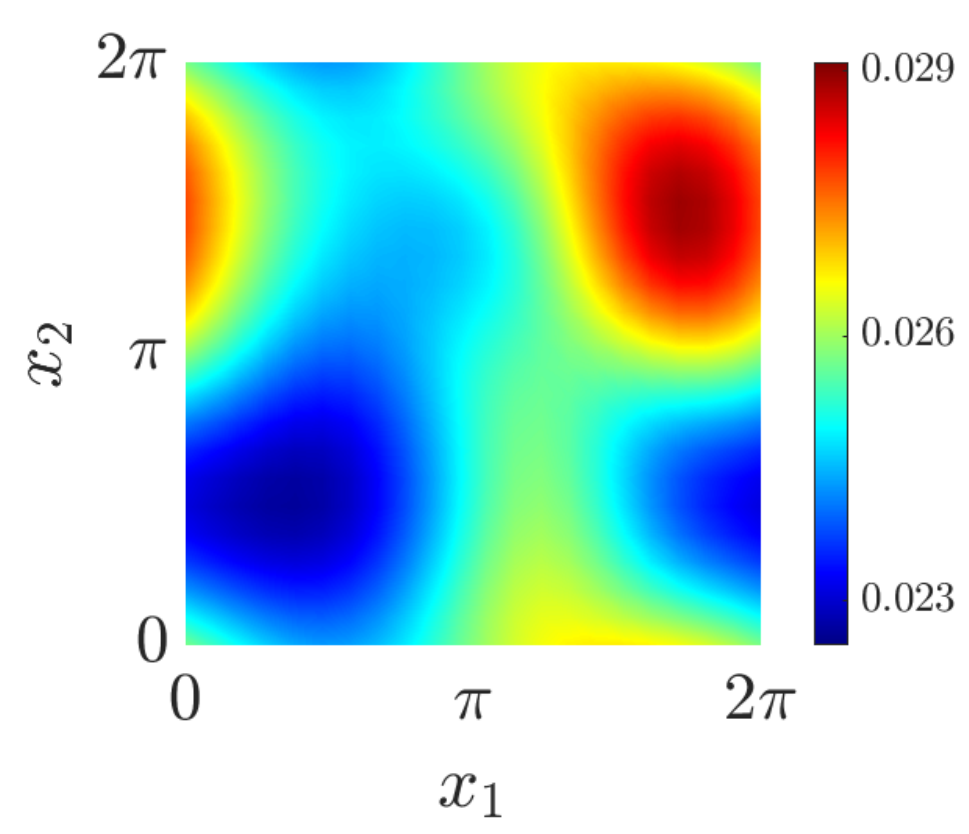}\\
{\rotatebox{90}{\hspace{1.6cm}\rotatebox{-90}{
\footnotesize{Steady State}\hspace{0.1cm}}}}
\includegraphics[scale=0.27]{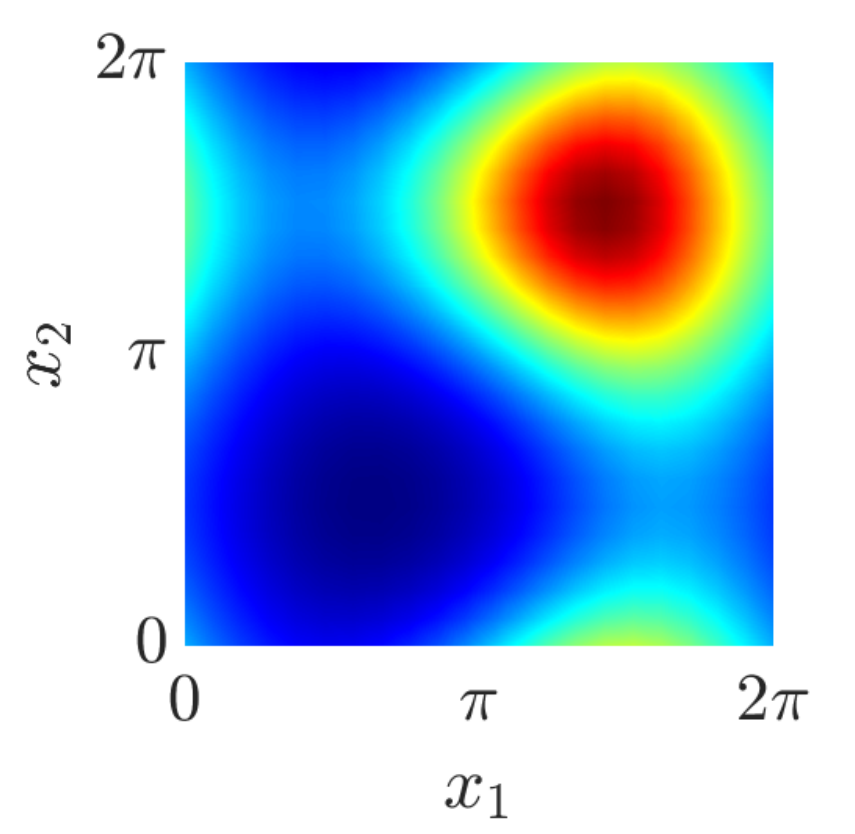}
\includegraphics[scale=0.27]{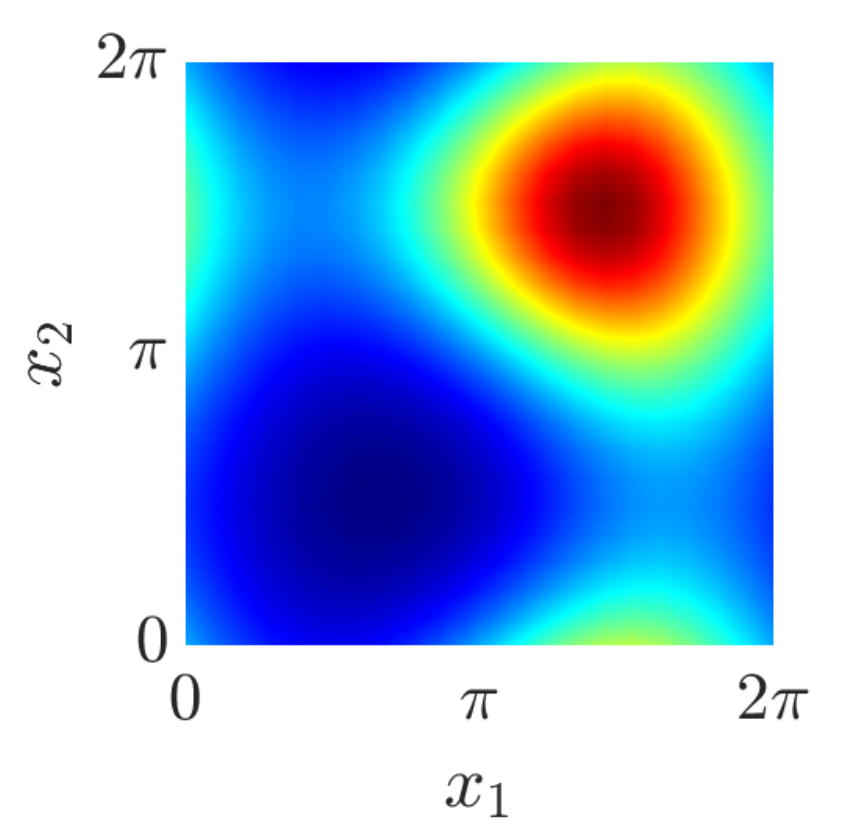}
\includegraphics[scale=0.27]{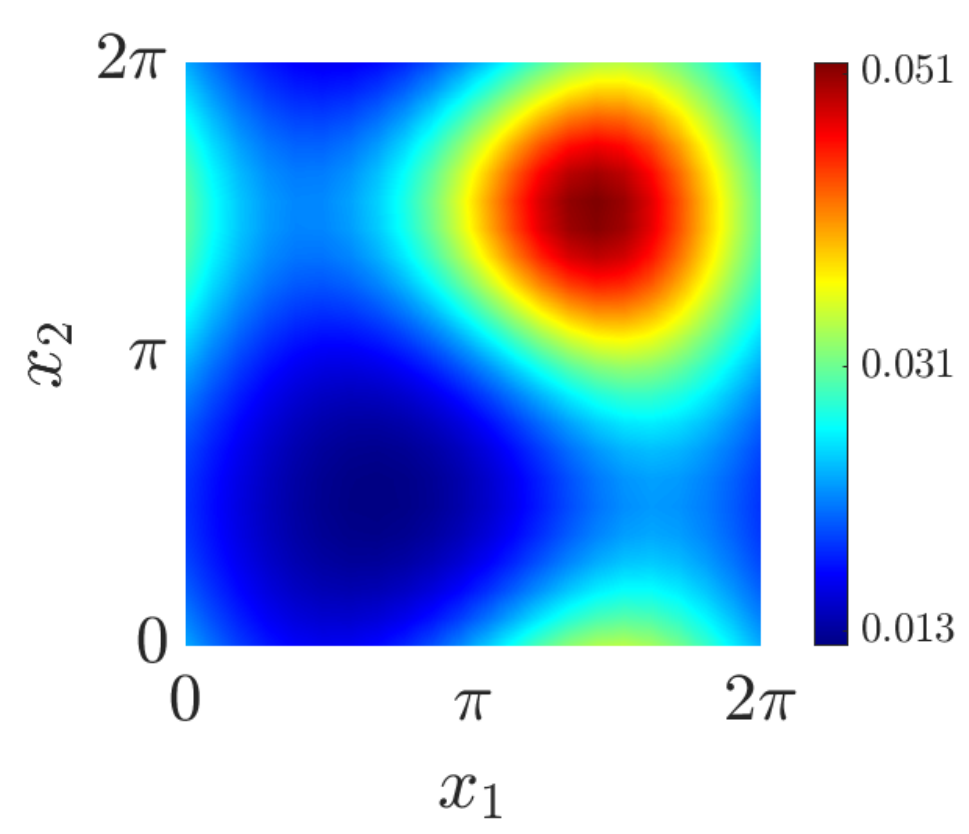}\\
\caption{
Marginal probability density function \eqref{marginal}
obtained by integrating numerically the 
Fokker-Planck equation \eqref{fp-lorenz-96} 
in dimension $d=4$ with
initial condition \eqref{ic4d}
using two methods: 
rank-adaptive Euler forward and rank-adaptive AB2.
The reference solution computed with a variable time step size 
RK4 method with absolute tolerance of $10^{-14}$
%These solutions are
computed  on a grid with 
$20^4=160000$ evenly-spaced points.
%The steady state is determined for 
%this computation by halting execution 
%when $\left\|\partial f_{\text{ref}}/\partial t\right\|_2$ 
%is below the numerical threshold $10^{-8}$. This
%happens at approximately $t \approx 6.25$ for the 
%initial condition \eqref{ic4d}.
}
\label{fig:time-snapshot-plot-4d}
\end{figure}

\begin{table}

\begin{centering}
\renewcommand{\arraystretch}{1.3}
\begin{tabular}{ |c|c|c|}
\hline
 Integration Method &
 Free Parameters &
 Dependent Parameters\\
\hline
\begin{tabular}{c}
Adaptive Euler
(HT \& TT Formats)
\\
(Sec. \ref{subsec:adaptive-euler})
\end{tabular}
&
\begin{tabular}{c}
$\Delta t= 10^{-3},$ \\
$M_1=M_2=10^2$
\end{tabular}
 &
 \begin{tabular}{rl}
 $\varepsilon_{\bf r}$&$=10^{-4},$\\
$\varepsilon_{\bf s}$&$=10^{-1}$
 \end{tabular}\\
\hline
\begin{tabular}{c}
Adaptive Midpoint (HT Format)\\
(Small Threshold)\\
(Sec. \ref{subsec:adaptive-midpoint}) 
\end{tabular}&
\begin{tabular}{c}
$\Delta t=10^{-3},$\\
$A=B=10^3,$\\
$G=10^2$
\end{tabular}
&
\begin{tabular}{rl}
$\varepsilon_{{\bm\alpha}}$&$=
10^{-6},$\\
$\varepsilon_{{\bm\beta}}$&$=
10^{-3},$\\
$\varepsilon_{{\bm\gamma}}$&$=
10^{-1}$
\end{tabular}
\\
\hline
\begin{tabular}{c}
Two-step rank-adaptive
Adams-Bashforth (HT Format)\\
(Small Threshold)\\
(Sec. \ref{subsec:adaptive-adams})
\end{tabular}
&

\begin{tabular}{c}
$\Delta t= 10^{-3},$\\
$A=B=10^3,$\\
$G_0=G_1=10^2$
\end{tabular}
&
\begin{tabular}{rl}
$\varepsilon_{{\bm\alpha}}$&$=
10^{-6},$\\
$\varepsilon_{{\bm\beta}}$&$=
10^{-3},$\\
$\varepsilon_{{\bm\gamma}(0)}$&$=
10^{-4},$\\
$\varepsilon_{{\bm\gamma}(1)}$&$=
10^{-4}$
\end{tabular}
\\
\hline
\begin{tabular}{c}
Two-step rank-adaptive
Adams-Bashforth (HT Format)\\
(Large Threshold)\\
(Sec. \ref{subsec:adaptive-adams})
\end{tabular}
&
\begin{tabular}{c}
$\Delta t= 10^{-3},$\\
$A=B=4\times 10^4,$\\
$G_0=G_1=4\times 10^2$
\end{tabular}
&
\begin{tabular}{rl}
$\varepsilon_{{\bm\alpha}}$&$=
4\times 10^{-5},$\\
$\varepsilon_{{\bm\beta}}$&$=
4\times 10^{-2},$\\
$\varepsilon_{{\bm\gamma}(0)}$&$=
4\times10^{-3},$\\
$\varepsilon_{{\bm\gamma}(1)}$&$=
4\times10^{-3}$
\end{tabular}
\\
\hline
\begin{tabular}{c}
Adaptive Midpoint (HT Format) \\
(Large Threshold)\\
(Sec. \ref{subsec:adaptive-midpoint}) 
\end{tabular}&
\begin{tabular}{c}
$\Delta t=10^{-3},$\\
$A=B=5\times 10^4,$\\
$G=5\times 10^3$
\end{tabular}
&
\begin{tabular}{rl}
$\varepsilon_{{\bm\alpha}}$&$=
5\times 10^{-5},$\\
$\varepsilon_{{\bm\beta}}$&$=
5\times 10^{-2},$\\
$\varepsilon_{{\bm\gamma}}$&$=
5$
\end{tabular}
\\
\hline
\end{tabular}\\
\renewcommand{\arraystretch}{1.0}
\end{centering}
\caption{Table of parameters for the rank-adaptive 
step-truncation integrators of the 
Fokker-Planck equation \eqref{fp-lorenz-96} 
in dimension $d=4$ with
initial condition \eqref{ic4d}.
These were heuristically chosen so that
the truncation
at each step (to rank 
$\bf r$ or $\bm \alpha$)
would be considerably smaller than
the time step. The first step of AB2 uses midpoint
with the coefficients listed above.
}
\label{fig:4d-tables}
\end{table}

\begin{figure}
%\centerline{\hspace{1cm}
%\footnotesize 
%Order 1 Methods
%\hspace{5.5cm}
%Order 2 Methods
%\hspace{0.5cm}
%}
%\centerline{\line(1,0){420}}
\centerline{
\includegraphics[scale=0.57]{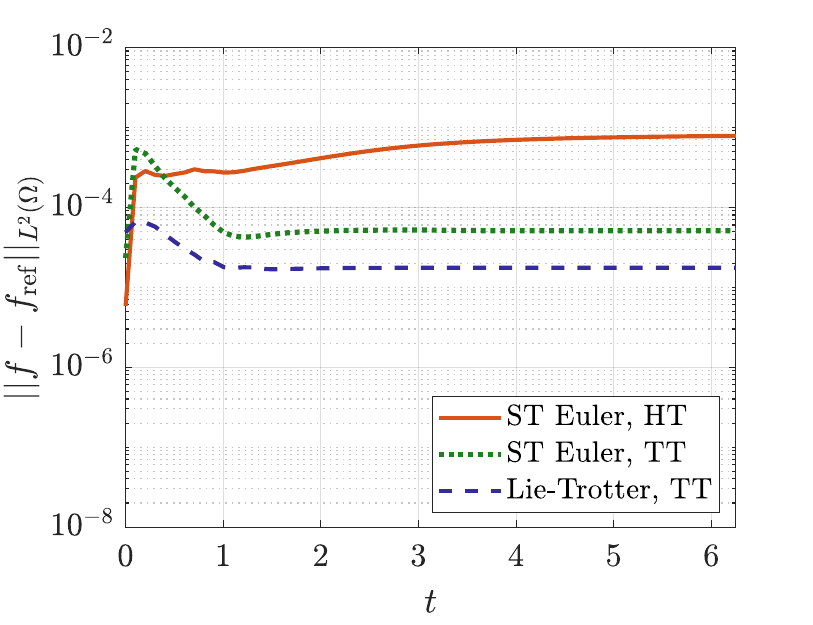}
\includegraphics[scale=0.57]{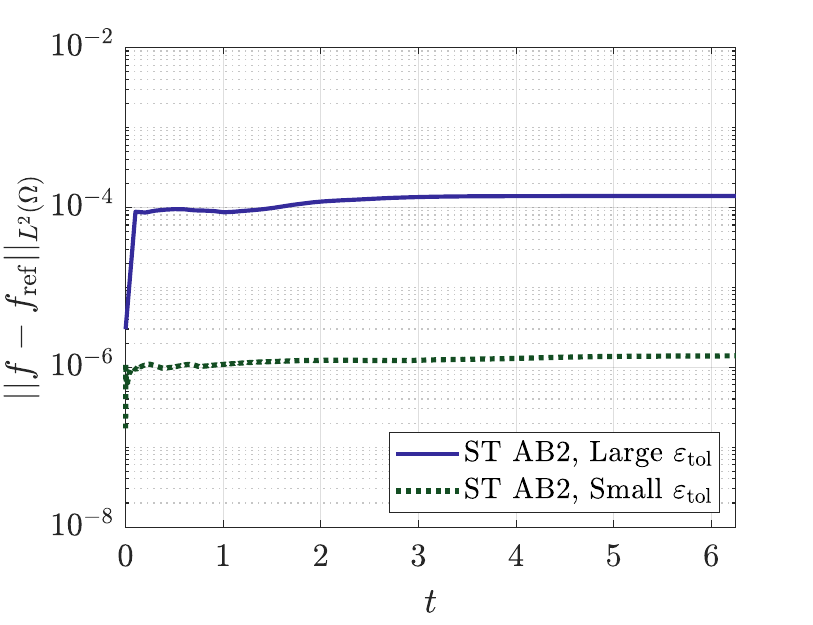}
}
\caption{$L^2(\Omega)$ error of 
numerical solutions to the Fokker-Planck equation \eqref{fp-lorenz-96} 
in dimension $d=4$ with
initial condition \eqref{ic4d}.
The parameters we used for all 
rank-adaptive step-truncation methods are summarized in Table 
\ref{fig:4d-tables}. The rank-adaptive Lie-Trotter method 
uses a threshold of $10^{-2}$ for the PDE component normal to the tensor manifold (see \cite{dektor2020rankadaptive}).
}
\label{fig:error-compare-4d}
\end{figure}
\begin{figure}[]
\centerline{
\includegraphics[scale=0.57]{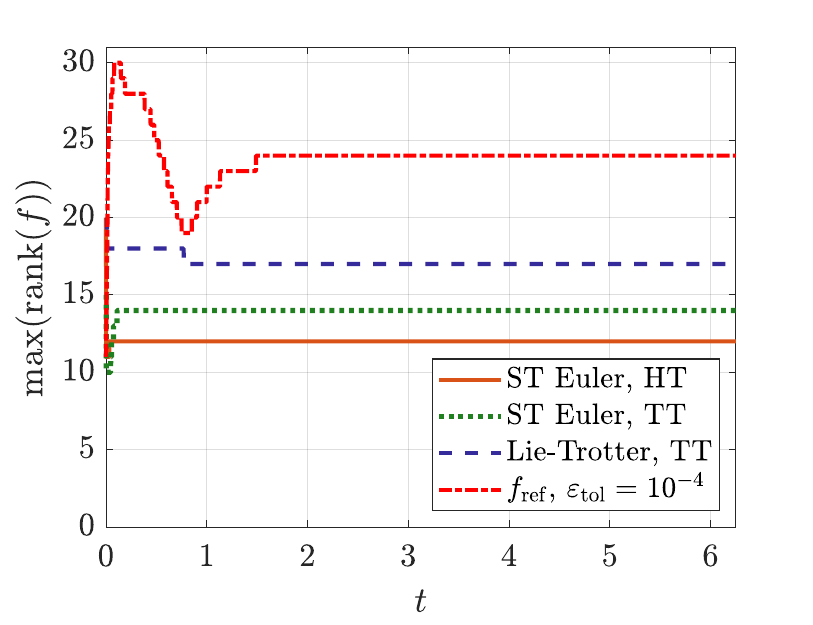}
\includegraphics[scale=0.57]{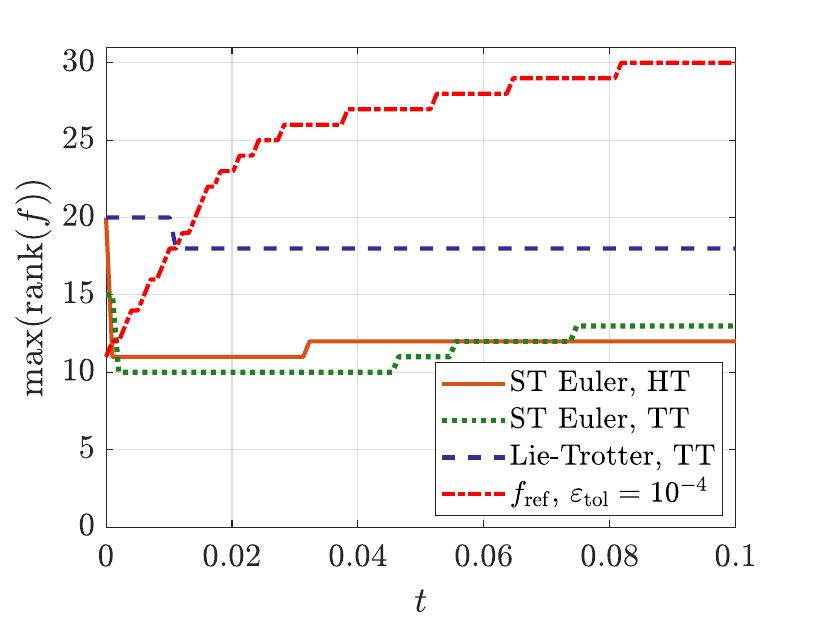}
}
\centerline{
\includegraphics[scale=0.57]{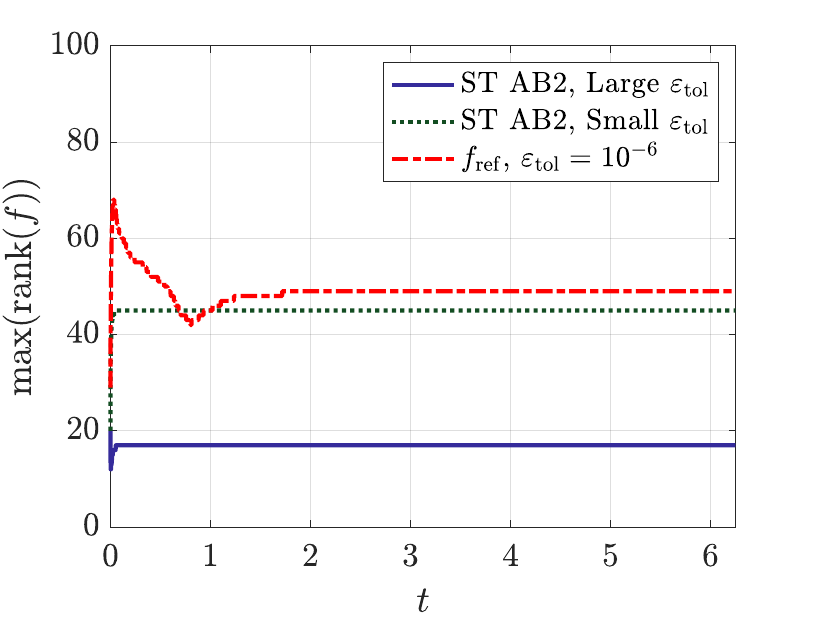}
\includegraphics[scale=0.57]{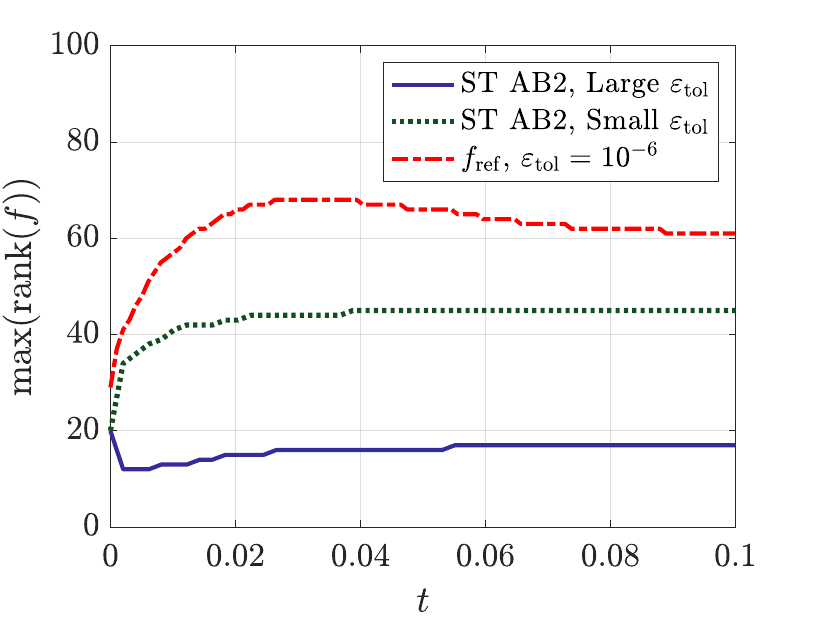}
}
\centerline{
\hspace{0.08cm}
\includegraphics[scale=0.57]{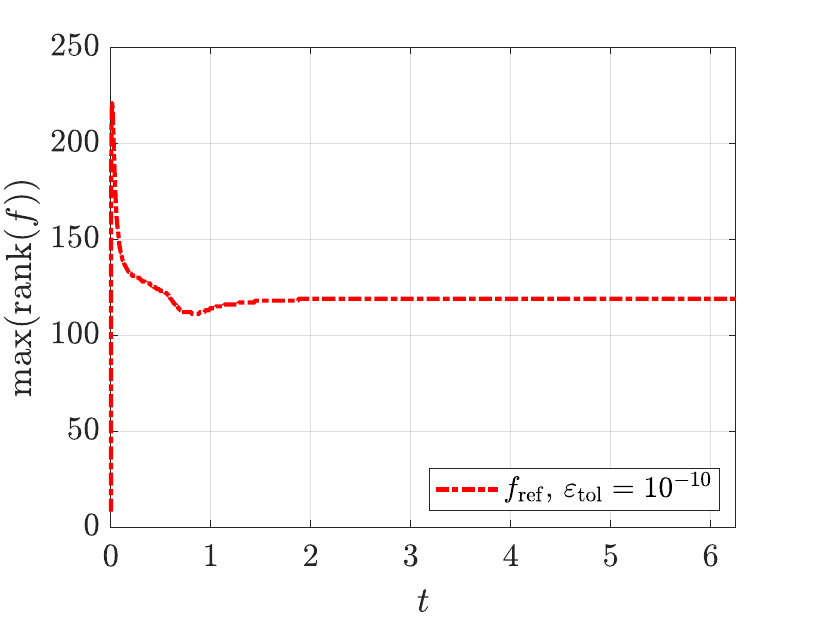}
\includegraphics[scale=0.57]{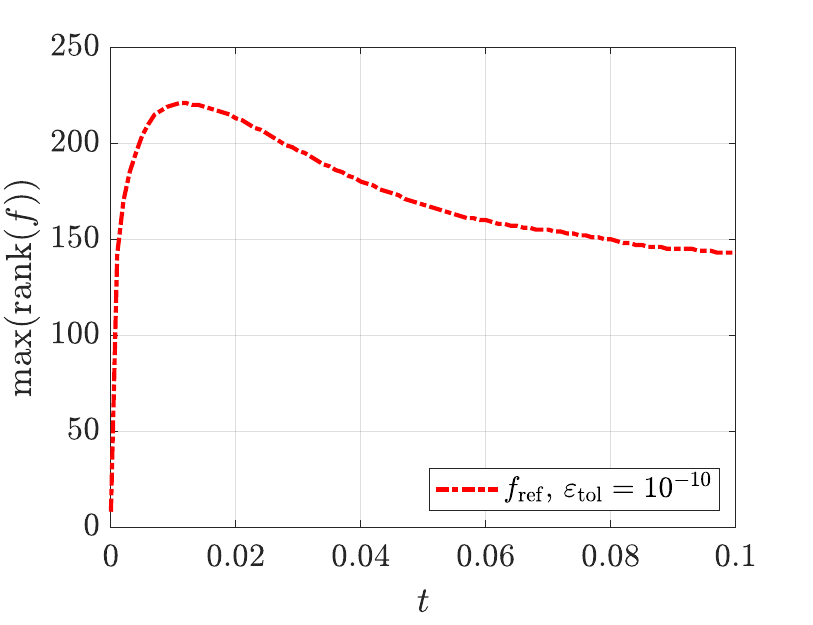}
}
\caption{ 
Rank versus time for the numerical solutions
of Fokker-Planck equation \eqref{fp-lorenz-96} 
in dimension $d=4$ with
initial condition \eqref{ic4d}
{%\color{red}
(left column: $0\leq t\leq 6.25$,
right column: $0\leq t\leq 0.1$).}
%The rank plotted here is the largest rank for all tensors
%being used to represent the solution in HT format
%or TT format, respectively.
{%\color{red}
We truncate the reference
solution to $\varepsilon_{\rm tol}$ in HT format.}
%It should be noted that the
%reference solution is not flowing exactly tangent
%to ${\cal H}_{\bf r}$ at steady state.
%
%By lowering the threshold closer
%to machine accuracy, we can capture over
%400 singular values.
%
%{\color{red}
%Here we show
%various singular value
%tolerances $\varepsilon_{\rm tol}>0$}
%to ensure that
%we are not incidentally capturing numerical
%artifacts due to double precision 
%floating point arithmetic. The parameters we used for all 
%rank-adaptive step-truncation methods are summarized in Table 
%\ref{fig:4d-tables}.
The rank-adaptive Lie-Trotter method 
uses a threshold of $10^{-2}$ for the PDE component normal to the tensor manifold (see \cite{dektor2020rankadaptive}).
}
\label{fig:rank-compare-4d}
\end{figure}

\begin{figure}
\centerline{
\includegraphics[scale=0.6]{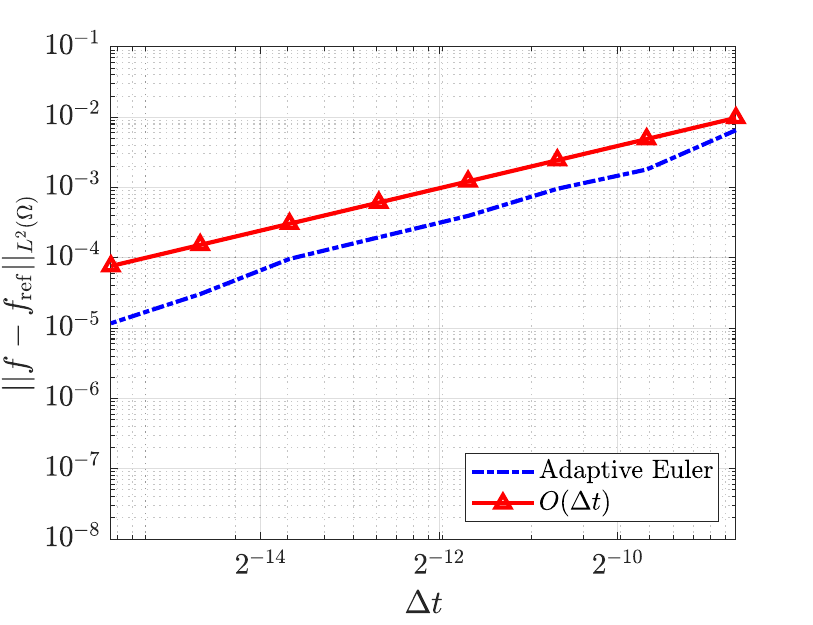}
\includegraphics[scale=0.6]{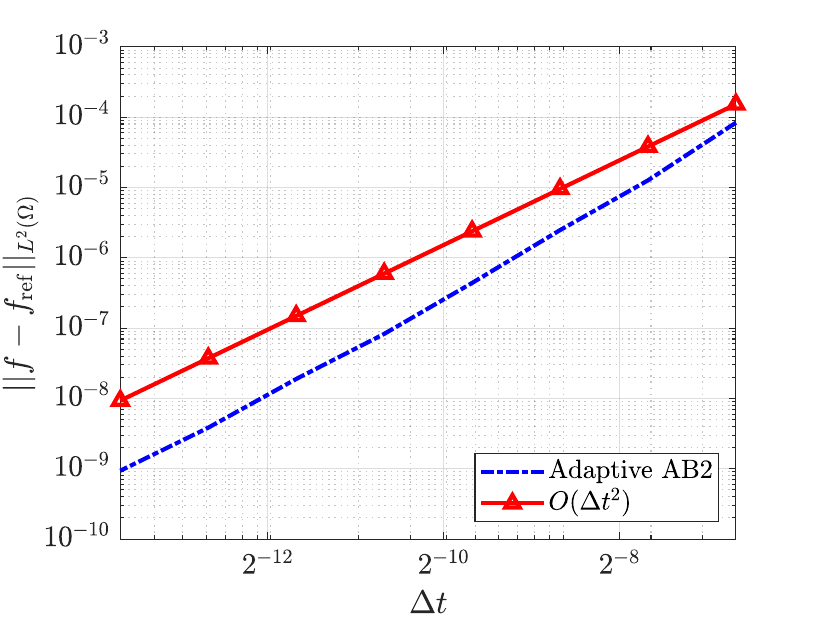}
}

\caption{Fokker-Planck equation \eqref{fp-lorenz-96} 
in dimension $d=4$ with
initial condition \eqref{ic4d}. $L^2(\Omega)$ errors at $T=0.1$ versus $\Delta t$ for different rank-adaptive step-truncation methods.
All tests used the HTucker tensor format.
}
\label{full-dynam-4d-err}
\end{figure}
%
%\begin{table}
%\begin{centering}
%\renewcommand{\arraystretch}{1.3}
%\resizebox{11.5cm}{!}{\begin{tabular}{ |c|c|c|}
%\hline
% Integration Method &
% Free Parameters &
% Dependent Parameters\\
%\hline
%\begin{tabular}{c}
%Rank-adaptive Euler
%\\
%(Sec. \ref{subsec:adaptive-euler})
%\end{tabular}
% &$\Delta t$ varies, $M_1=M_2=10^4$   &
% \begin{tabular}{rl}
% $\varepsilon_{\bf r}$&$=M_1\Delta t^{2},$\\
%$\varepsilon_{\bf s}$&$=M_2\Delta t$
% \end{tabular}\\
%\hline
%%
%\begin{tabular}{c}
%Two-step rank-adaptive
%Adams-Bashforth\\
%(Sec. \ref{subsec:adaptive-adams})
%\end{tabular}
%&
%\begin{tabular}{c}
%$\Delta t$ varies,$A=B=10^5$,
%$G_0=G_1^4$
%\end{tabular}
%&
%\begin{tabular}{rl}
%$\varepsilon_{{\bm\alpha}} $&$= A\Delta t^3,$\\
%$\varepsilon_{{\bm\beta}} $&$= B\Delta t^2,$\\
%$\varepsilon_{{\bm\gamma}(0)} $&$= G_0\Delta t^2$,\\
%$\varepsilon_{{\bm\gamma}(1)} $&$= G_1\Delta t^2$
%\end{tabular}\\
%\hline
%\begin{tabular}{c}
%Rank-adaptive midpoint\\
%(Sec. \ref{subsec:adaptive-midpoint}) 
%\end{tabular}&
%$\Delta t$ varies, $A=B=10^5$, $G=10^3$
%&
%\begin{tabular}{rl}
%$\varepsilon_{{\bm\alpha}}$&$=A\Delta t^3,$\\
%$\varepsilon_{{\bm\beta}}$&$=B\Delta t^2,$\\
%$\varepsilon_{{\bm\gamma}}$&$=G\Delta t$
%\end{tabular}
%\\
%\hline
%\end{tabular}}
%\caption{Table of parameters we used to for different 
%rank-adaptive step-truncation methods we used to 
%compute the numerical solution of the Fokker-Planck 
%equation \eqref{fp-lorenz-96} in dimension $d=4$.}
%
%\renewcommand{\arraystretch}{1.0}
%\end{centering}
%\label{4dFPKtable}
%\end{table}

\subsection{Two-dimensional Fokker-Planck equation}
\label{subsec:numerics-2d}

Set $d=2$ in \eqref{fp-lorenz-96} and consider the initial condition 
\begin{equation}
f_0(x_1,x_2)=
\frac{1}{m_0}
\left[ 
e^{\sin(x_1-x_2)^2} + \sin(x_1+x_2)^2
\right],
\label{ic2d}
\end{equation}
where  $m_0$ is a normalization factor. Discretize 
\eqref{ic2d} on a two-dimensional grid of evenly-spaced 
points
%\footnote{The numerical results shown in Figures 
%\ref{fig:time-snapshot-plot}, 
%\ref{fig:error-compare-2d}, and \ref{fig:rank-compare-2d} 
%are obtained on a $50\times 50$ spatial grid (2500 points total) 
%while  in Figure \ref{fig:global-err-2d} we
%used a $40 \times 40$ spatial grid (1600 points total).}, 
and then truncate the initial tensor 
(matrix) within machine accuracy into HT format.
Also, set $\gamma(x) = \sin(x)$, $\xi(x) = \cos(x)$,
and $\phi(x)=\exp(\sin(x)) + 1$ for the drift 
functions in \eqref{fp-lorenz-96}.
In Figure \ref{fig:time-snapshot-plot}, we plot 
the numerical solution of the Fokker-Planck equation 
\eqref{fp-lorenz-96} in dimension $d=2$ corresponding 
to the initial condition \eqref{ic2d}. We computed our
solutions with {%\color{blue}
four different methods:
\begin{enumerate}
\item Rank-adaptive explicit Euler \eqref{adaptive-euler-method};
\item Two-step rank-adaptive Adams-Bashforth (AB) 
method \eqref{adaptive-ab};
\item Rank-adaptive tensor method with Lie-Trotter 
operator splitting integrator \cite{dektor2020rankadaptive};
\item RK4 method applied to the ODE \eqref{mol-ode} corresponding 
to a full tensor product discretization of \eqref{fp-lorenz-96}. 
We denote this reference solution as $f_{\text{ref}}$. 
\end{enumerate}}
The parameters we used for the rank-adaptive step-truncation 
methods 1. and 2. are summarized in Table \ref{fig:2d-tables}.  
The steady state was determined for this 
computation by halting execution when
$\left\|\partial f_{\text{ref}}/\partial t\right\|_{L^2(\Omega)}$ was 
below the numerical threshold $10^{-13}$. This
occurs at approximately $t \approx 24$ for the 
initial condition \eqref{ic2d}.
The numerical results in Figure \ref{fig:time-snapshot-plot} shows 
that the step-truncation methods listed above match 
all visual behavior of the reference solution.
Observing Figures \ref{fig:error-compare-2d} and 
\ref{fig:rank-compare-2d}, we note that while 
the rank-adaptive AB2 methods nearly doubles
the digits of accuracy (in the $L^2(\Omega)$ norm),
only a modest increase in rank is required to achieve 
this gain in accuracy.  
This is because the rank in each adaptive step-truncation scheme 
is determined by the increment function $\bf \Phi$ 
(which defines the scheme), the nonlinear 
operator $\bf N$, and the truncation error 
threshold (which depends on $\Delta t$).
More precisely, the closer ${\bf \Phi}({\bf N},{\bf f},\Delta t)$
is to the tangent space of the 
manifold ${\cal H}_{\bf r}$ at ${\bf f}_k$, 
the less the rank will increase in next time step. 
In our demonstration, this occurs as the solution 
${\bf f}_k$ approaches steady state, since, 
as the rate at which the probability density evolves 
in time slows down, the quantity 
$\left\|{\bf \Phi}({\bf N},{\bf f},\Delta t)\right\|_2$
tends to zero. 
Consequently, $\left\|({\bf I} - {{\cal P}_{\bf g}}) 
{\bf \Phi}({\bf N},{\bf f},\Delta t)\right\|_2$ 
will also tend towards zero since ${\bf I} - {{\cal P}_{\bf g}}$
is a bounded linear operator. 
For fixed $\Delta t$,
this means that the rank increase conditions
\eqref{eq:68}-\eqref{eq:70} 
will have a smaller likelihood of being triggered.
As we shrink $\Delta t$, the truncation error 
requirements for consistency 
\eqref{eq:68}-\eqref{eq:70} become more demanding, 
and thus a higher solution rank is expected.
{%\color{blue}
In Figures \ref{fig:error-compare-2d} and 
\ref{fig:rank-compare-2d} we also see that the rank-adaptive
tensor method with Lie-Trotter 
integrator proposed in \cite{dektor2020rankadaptive} 
performs better on this problem than 
rank-adaptive step-truncation methods, especially 
when the solution approaches the steady state. 
However, it should be noted that the rank-adaptive 
method with operator splitting and normal 
vector control is considerably more involved to implement
than the step-truncation methods, which
are essentially slight modifications
of a standard single-step or multi-step
method.
}
In Figure \ref{fig:global-err-2d} we demonstrate numerically
the global error bound we proved in Theorem \ref{thm:st-global-error}. 
The error scaling constant $Q$ turns out
to be $Q=2$ for rank-adaptive AB2, $Q=5$ for rank-adaptive midpoint, 
and $Q=0.6$ for rank-adaptive Euler forward.

\subsection{Four-dimensional Fokker-Planck equation}
\label{subsec:numerics-4d}

Next, we present numerical results for the 
Fokker-Planck equation \eqref{fp-lorenz-96} 
in dimension $d=4$. In this case, the best truncation 
operator \eqref{ht-best} is not explicitly known. 
Instead, we use the step-truncation method \eqref{one-step-trunc-svd_adapt_rank},
with truncation operator ${\mathfrak T}_r^{\text{SVD}}$ 
defined in \eqref{ht-SVD} (see \cite{grasedyck2010hierarchical,kressner2014algorithm} for 
more details). We set the initial condition as 
\begin{equation}
 f_0(x_1,x_2,x_3,x_4) = \frac{1}{m_0}\sum_{j=1}^L
 \left ( \ \prod_{i=1}^4
                    \frac{\sin((2j-1)x_i)+1}{2^{2(j-1)}} +
                    \prod_{i=1}^4
                    \frac{\exp(\cos(2jx_i))}{2^{2j-1}}
                    \right ),
                    \label{ic4d}
\end{equation}
where $m_0$ is a normalization constant.
Clearly, \eqref{ic4d} can be represented exactly in 
a hierarchical Tucker tensor format provided we use 
an overall maximal tree rank of $r_{0}=2L$. 
For our numerical simulations we choose $L=10$.
We change the drift functions slightly 
from the two-dimensional example we discussed in the previous 
section. Specifically, here we set 
$\gamma(x) = \sin(x)$, $\xi(x)=\exp(\sin(x)) + 1$,
and  $\phi(x) = \cos(x)$
and repeat all numerical tests 
presented in section \ref{subsec:numerics-2d}, i.e., we 
run three rank-adaptive step-truncation simulations with
different increment functions: one based on Euler
forward \eqref{adaptive-euler-method} and one based 
AB2 \eqref{adaptive-ab}. The parameters we used for these 
methods are summarized in Table \ref{fig:4d-tables}.

For spatial discretization, we use the Fourier pseudo-spectral method
with $20^4 = 160000$ points.
We emphasize that a matrix representing the discretized Fokker-
Planck operator at the right hand side of \eqref{fp-lorenz-96}
would be very sparse and require approximately 205 gigabytes
in double precision floating point format.
The solution vector requires 1.28 megabytes of memory (160000
floating point numbers in double precision).
The HTucker format reduces these memory costs considerably.
The large threshold solution of Figure \ref{fig:rank-compare-4d}
is only 25 kilobytes when stored to disk using
the HTucker Matlab software package \cite{kressner2014algorithm}.
The spatial differential operator for the Fokker-Planck equation
can also be represented in HTucker format,
and costs only 21 kilobytes. The storage savings are massive, 
so long as the rank is kept
low. 
In Figure \ref{fig:time-snapshot-plot-4d}, we plot 
a few time snapshots of the marginal PDF 
\begin{equation}
f_{12}(x_1,x_2,t)=\int_{0}^{2\pi}\int_{0}^{2\pi} 
f(x_1,x_2,x_3,x_4,t)dx_3dx_4
\label{marginal}
\end{equation}
we obtained by integrating \eqref{fp-lorenz-96} in time 
with rank-adaptive Euler forward and rank-adaptive AB2.
In Figure \ref{fig:rank-compare-4d} we plot the solution
rank versus time for all rank-adaptive step-truncation 
integrators summarized in Table \ref{fig:4d-tables}.
The results largely reflect those of the two 
dimensional domain. However, a notable difference is 
the abrupt change in rank. This is because the density 
function in this case relaxes to steady state fairly
quickly.
{%\color{red}
Numerically, the steady state is determined 
by halting execution 
when $\left\|\partial f_{\text{ref}}/\partial t\right\|_2$ 
is below the numerical threshold $10^{-8}$. This
happens at approximately $t \approx 6.25$ for the 
initial condition \eqref{ic4d}.}
%as documented in the caption of Figure
%\ref{fig:time-snapshot-plot-4d}.
As the rate of 
change in the density function becomes very small, 
we see that the rank no longer changes. 
This happens near time $t=0.1$ (see Figure
\ref{fig:rank-compare-4d}).

The proposed rank-adaptive step-truncation methods
can provide solutions with varying accuracy 
depending the threshold, i.e., the parameters 
summarized in Table \ref{fig:4d-tables}. 
To show this, in Figure 
\ref{fig:rank-compare-4d} we 
compare the rank dynamics in the adaptive 
AB2 simulations obtained with small or large 
thresholds.
Note that the solution computed with a 
large error threshold is rather low rank 
(see Figure \ref{fig:rank-compare-4d}).
We also see that the rank can be kept near the 
rank of the initial condition, if desired 
(again see Figure \ref{fig:rank-compare-4d}).
Finally, in Figure \ref{full-dynam-4d-err}
we plot the error $L^2(\Omega)$ error 
at $T=0.1$ versus $\Delta t$ 
for two different rank-adaptive step-truncation 
methods, i.e., Euler and AB2. It is that the order 
of AB2 is slightly larger than $2$.
This can be explained by noting that the 
error due to rank truncation is essentially a sum
of singular values. Such singular values 
can be smaller than the truncation thresholds
${\varepsilon}_{\bm \kappa}$ ($\bm \kappa=\bf r, \bf s, \bm \alpha$, ...), suggesting  the theoretical bounds may not be tight.

\appendix
\section{Proof of Lemma \ref{trunc-smooth-thm}}
\label{sec:appendix-bst-analysis}
 
In this section, we 
present a proof of Lemma \ref{trunc-smooth-thm} which
is specific to ${\cal H}_{\bf r}$. First, we start
by constructing an open set centered about a point
with known rank.
\begin{lemma}
\label{add-open-lemma}
Let ${\bf f}\in {\cal H}_{\bf r}$ be a point on the hierarchical
Tucker manifold of constant rank. Let
${\bf v}\in T_{\bf f}{\cal H}_{\bf r}$ be an arbitrary vector in 
the tangent plane of ${\cal H}_{\bf r}$ at $\bf f$.
Then there exists $\eta >0$ such that for all
$\varepsilon$ satisfying $0\leq\varepsilon
\leq\eta$, we have ${\bf f}+ \varepsilon{\bf v} =
{\bf g} \in{\cal H}_{\bf r}$.
As a consequence, if $U_{\bf f}\subseteq T_{\bf f}{\cal H}_{\bf r}$
is a closed and bounded set containing the origin,
then there exists an open subset $V_{\bf f}\subseteq U_{\bf f}$
such that ${\bf f} + {\bf h}\in {\cal H}_{\bf r}$, for all ${\bf h}\in 
V_{\bf f}$ .
\end{lemma}
\begin{proof}
First, consider a simpler problem, in which we have two 
matrices ${\bf A},{\bf B}\in{\mathbb R}^{n\times m}$,
where $\bf A$ is full column rank. Consider the function
\begin{equation}
p(\eta) = \det\left(({\bf A}+\eta {\bf B})^{T}({\bf A}+\eta {\bf B})\right).
\end{equation}
Clearly, $p(\eta)$ is a polynomial and thus smooth in $\eta$. 
Moreover, $p(0)\neq 0$ since ${\bf A}$ is full column rank. 
Since $p$ is smooth, there exists some $\eta > 0$ such 
that $p(\varepsilon)\neq 0$ for all 
$\varepsilon\in[0,\eta]$. Since the full-rank
hierarchical Tucker manifold is defined via the 
full column rank constraints on an array of matrices 
corresponding to matricizations of the tensor \cite{uschmajew2013geometry}, we can apply 
the principle above to every full column rank 
matrix associated with the tree, using addition of a point
and a tangent as referenced in Proposition 3 of
\cite{da2015optimization}.
We have now proved the part one of the 
lemma where $\eta$ is taken to be the
minimum over the tree nodes. As for existence of an
open set, suppose $U_{\bf f}$ is open and bounded. Now 
we apply the above matrix case to the boundary 
$\partial U_{\bf f}$,
giving us a star shaped set $S_{\bf f}\subseteq U_{\bf f}$.
Letting $V_{\bf f} = S_{\bf f}\setminus \partial S_{\bf f}$
be the interior, completes the proof of the lemma.
\begin{flushright}
\qed
\end{flushright}
\end{proof}
%
%{\color{red}
%At this point we have all elements to prove Proposition
%\ref{trunc-smooth-thm}. Our proof is an adapted version of
%a general theorem for surfaces without boundary by
%Marz and Macdonald \cite{marz2012calculus}. 
%}
%
%{\color{blue}
We use the open set constructed 
above to prove smoothness using the same techniques as
\cite{marz2012calculus}.
%}
\begin{proof}{\em (Lemma \ref{trunc-smooth-thm})}
Let ${\bf f} \in {\cal H}_{\bf r}
\subseteq {\mathbb R}^{n_1\times n_2 \times \cdots \times n_d}$.
By Lemma \ref{add-open-lemma}, there exists 
an open norm-ball $B({\bf f}, \kappa)$ located 
at ${\bf f}$ with radius $\kappa>0$ so that
\begin{equation}
{{\bf f}} + {\bf v}\in  {\cal H}_{\bf r}\qquad \forall {\bf v}\in  {\cal P}_{\bf f}\left ( B({{\bf f}}, \kappa)\right).
\end{equation}
Let ${\cal U}_{\bf f} = {\cal H}_{\bf r} \cap B({\bf f}, \kappa)$
be a set which is open in the topology of ${\cal H}_{\bf r}$.
Also,  let
$({\bf q}_{{\bf f}}, {\bf q}_{{\bf f}}^{-1}({\cal U}_{{\bf f}}))$
be a local parametrization at
${\bf f}$. For the parametrizing coordinates, we take an
open subset
${\bf q}_{{\bf f}}^{-1}({\cal U}_{{\bf f}})=V_{{\bf f}}\subseteq T_{{\bf f}}{\cal H}_{\bf r}$
of the tangent space embedded in
${\mathbb R}^{n_1\times n_2 \times \cdots \times n_d}$.
This means that the parametrization ${\bf q}_{{\bf f}}$ 
takes tangent vectors as inputs and maps 
them into tensors in ${\cal H}_{\bf r}$, i.e.
\begin{equation}
{\bf q}_{{\bf f}} :T_{{\bf f}}{\cal H}_{\bf r}\rightarrow
{\cal H}_{\bf r}.
\end{equation}
Moreover, we assume that the coordinates
are arranged in column major ordering as a vector. 
This allows for the Jacobian
$\partial {{\bf q}}_{{\bf f}} /\partial {\bf v}$
to be a basis for the tangent space $T_{{\bf f}}{\cal H}_{\bf r}$. 
Note that $\partial {\bf q}_{{\bf f}} /\partial {\bf v}$ is 
a $(n_1 n_2 \cdots n_d)\times
\text{dim}(T_{{\bf f}}{\cal H}_{\bf r})$ matrix with 
real coefficients. Now, let ${\bf M}({\bf f})$ be a 
matrix of column vectors spanning the space orthogonal to
$T_{\bf f}{\cal H}_{\bf r}$ in
${\mathbb R}^{n_1\times n_2 \times \cdots \times n_d}$.
Since the two linear spaces are disjoint,
we have a local coordinate map for the ball
$B({\bf f}, \kappa)$, given by
\begin{equation}
{\bf C}({\bf v},{ \bf g}) =
{\bf q}_{{\bf f}}({\bf  v})
+[{\bf M}({\bf q}_{{\bf f}}({\bf v}))]
{\bf g},
\end{equation}
where $\bf v$ is tangent and $\bf g$ is normal (both column vectors). By construction,
\begin{equation}
{\mathfrak T}_{\bf r}^{\text{best}}({\bf C}({\bf v},{ \bf g}) )
= {\bf q}_{{\bf  f}} ({\bf v})
\end{equation}
is smooth in both $\bf v$ and $\bf g$.
Therefore, we can take the total derivative on the embedded
space and apply the chain rule to obtain the Jacobian
of ${\mathfrak T}_{\bf r}^{\text{best}}({\bf f})$. 
Doing so, we have
\begin{align}
\frac{\partial}{\partial(\bf v, \bf g)}
{\mathfrak T}_{\bf r}^{best}({\bf C}({\bf v},{\bf g}))
&=
\frac{\partial{\mathfrak T}_{\bf r}^{\text{best}}}{\partial \bf C}
\frac{\partial \bf C}{\partial(\bf v, \bf g)}
=
\frac{\partial{\mathfrak T}_{\bf r}^{\text{best}}}{\partial \bf C}
\left [
\frac{\partial {\bf q}_{{\bf f}}}{\partial \bf v}
+
\sum_{i=1}^{n^\perp}
\frac{\partial {\bf M}_i({\bf q}_{{\bf f}}({\bf v}))}
{\partial \bf v}{ g}_i
\bigg |
{\bf M}({\bf q}_{{\bf f}}({\bf v}))
\right ],
\end{align}
where the symbol $[\cdot |\cdot ]$ denotes 
column concatenation of matrices, $n^\perp$ is 
the dimension of the normal space
$(T_{\bf f}{\cal H}_{\bf r})^\perp$,  ${\bf M}_i$
is the $i$-{th} column of $\bf M$, and ${g}_i$ 
is the $i$-{th} component of $\bf g$.
We can take ${\bf g} = {\bf 0}$ since the above expression
extends smoothly from the embedding space onto
${\cal H}_{\bf r}$. Hence, the Jacobian 
of $\mathfrak{T}^{\text{best}}$ is the solution to the 
linear equation
\begin{equation}
\frac{\partial{\mathfrak T}_{\bf r}^{\text{best}}}{\partial \bf C}
\left [
\frac{\partial {\bf q}_{{\bf f}}}{\partial \bf v}
\bigg |
 {\bf M}({\bf q}_{{\bf f}}({\bf v}))
\right ]
=
\left [
\frac{\partial {\bf q}_{{\bf f}}}{\partial {\bf v}}
\bigg  |
{\bf 0}
\right ].
\label{A.11}
\end{equation}
Since the right factor of the left hand side has
a pair of orthogonal blocks, we can write the inverse
using the pseudo-inverse of the blocks, i.e., 
\begin{equation}
\left [
\frac{\partial {\bf q}_{{\bf f}}}{\partial \bf v}
\bigg |
 {\bf M}({\bf q}_{{\bf f}}({\bf v}))
\right ]^{-1}
=
\begin{bmatrix}
\left [
\displaystyle \frac{\partial {\bf q}_{{\bf f}}}{\partial \bf v}
\right]^{+} \\ \\
 \left[{\bf M}({\bf q}_{{\bf f}}({\bf v}))\right]^{+}
\end{bmatrix}.
\end{equation}
The right hand side is the block concatenation 
of the rows of each pseudo-inverse. 
Plugging the above expression into \eqref{A.11}, 
we find
\begin{equation}
\frac{\partial{\mathfrak T}_{\bf r}^{\text{best}}}{\partial \bf C}
=
\frac{\partial {\bf q}_{{\bf f}}}{\partial \bf v}
\left [
\frac{\partial {\bf q}_{{\bf f}}}{\partial  \bf v }
\right ]^{+},
\end{equation}
which is exactly the expression for the orthogonal
projection onto the tangent space \cite{lubich2013dynamical}. 
This completes the proof.
%\begin{flushright}
\hfill\(\qed\)
%\end{flushright}
\end{proof}

\section{Step-truncation methods for matrix-valued ODEs on matrix manifolds with fixed rank}
\label{sec:appendix-projected2D} 
To make Lemma \ref{trunc-smooth-thm} concrete,
in this Appendix we write down 
${\mathfrak T}_{\bf r}^{\text{best}}$ and its
Jacobian ${\cal P}_{{\bf f}}$ for problems where
${\bf f}\in\mathbb{R}^{n_1 \times n_2}$ is a matrix.
In this situation, the tree rank $\bf r$ is just a
single integer $r$. One can see from the accuracy 
inequalities for best truncation proven in
\cite{grasedyck2010hierarchical}
that the ${\mathfrak T}_{r}^{\text{best}}$ is obtained
from truncating the smallest $\text{min}(n_1,n_2) - r$
singular values and singular vectors.
For simplicity, we will write down the best truncation 
scheme for \eqref{mol-ode} using the Euler forward method. 
This gives
\begin{equation}
\label{step-trunc-2d}
{\bf f}_{k+1} = {\mathfrak T}_{r}^{\text{best}}({\bf f}_k
+\Delta t {\bf N}({\bf f}_k)).
\end{equation}
Assuming that we are fixing rank to be the same as the
initial condition for all $k$, we have that
${\mathfrak T}_{r}^{\text{best}}({\bf f}_k) = {\bf f}_k$.
Now we can apply SVD perturbation theory
\cite{liu2008first,stewart1998perturbation} to express 
the best truncation operator in terms of a power series 
expansion in $\Delta t$. Representing
our decomposition as the a tuple of matrices
$({\bf \Sigma}_k,{\bf Q}_k,{\bf V}_k)$, where
${\bf f}_k={\bf Q}_k {\bf \Sigma}_k{\bf V}_k^\top$
and ${\bf f}_{k+1}={\bf Q}_{k+1} {\bf \Sigma}_{k+1}{\bf V}_{k+1}^\top$ is the reduced singular value decomposition, we have that
\begin{align}
{\bf \Sigma}_{k+1} &= {\bf \Sigma}_k + 
\Delta t
\text{diag}({\bf Q}_k^\top {\bf N}({\bf f}_k)
 {\bf V}_k) 
+ O(\Delta t^2) ,
\\ 
{\bf Q}_{k+1}&=
{\bf Q}_k +
\Delta t
{\bf Q}_k(
{\bf H}_k\odot(
{\bf Q}_k^\top
{\bf N}({\bf f}_k){\bf V}_k
{\bf \Sigma}_k + {\bf \Sigma}_k{\bf V}_k^\top
{\bf N}({\bf f}_k)^\top{\bf Q}_k)
)\nonumber \\
 &\qquad\ \  +
\Delta t
({\bf I} - {\bf Q}_k{\bf Q}_k^\top)
{\bf N}({\bf f}_k){\bf V}_k{\bf \Sigma}_k^{-1}
+ O(\Delta t^2), \\
{\bf V}_{k+1}&=
{\bf V}_k +
\Delta t
{\bf V}_k(
{\bf H}_k\odot(
{\bf \Sigma}_k{\bf Q}_k^\top
{\bf N}({\bf f}_k){\bf V}_k
 + {\bf V}_k^\top
{\bf N}({\bf f}_k)^\top{\bf Q}_k{\bf \Sigma}_k)
)\nonumber \\
 &\qquad\ \   +
\Delta t
({\bf I} - {\bf V}_k{\bf V}_k^\top)
{\bf N}({\bf f}_k)^\top{\bf Q}_k{\bf \Sigma}_k^{-1}
+ O(\Delta t^2).
\end{align}
Here, $\odot$ denotes is the
element-wise (Hadamard) product of matrices, and the matrix
\begin{equation}
\begin{cases}
{\bf H}_k[i,j]=
1/({\bf \Sigma}_k[j,j]^2 -{\bf \Sigma}_k[i,i]^2), & i\neq j, \\
{\bf H}_k[i,j]=0, & i = j,
\end{cases}
\end{equation}
is skew-symmetric and stores information about the
differences of the singular values.
The $\text{diag}(\cdot )$ operation zeros out all
elements off of the diagonal.
The tangent space projection operator is the coefficient
of the $\Delta t$ terms. From here, we can see that
the  evolution equation corresponding to \eqref{step-trunc-2d} is
\begin{align}
\frac{\text{d}}{\text{d} t}{\bf \Sigma} &= 
\text{diag}({\bf Q}^\top {\bf N}({\bf f} )
 {\bf V} )  ,
\\ 
\frac{\text{d}}{\text{d} t}{\bf Q}&=
{\bf Q} (
{\bf H} \odot(
{\bf Q}^\top
{\bf N}({\bf f} ){\bf V} 
{\bf \Sigma}  + {\bf \Sigma} {\bf V}^\top
{\bf N}({\bf f} )^\top{\bf Q} )
)  +
({\bf I} - {\bf Q} {\bf Q}^\top)
{\bf N}({\bf f} ){\bf V} {\bf \Sigma} ^{-1},
 \\
\frac{\text{d}}{\text{d} t}{\bf V}&=
{\bf V} (
{\bf H} \odot(
{\bf \Sigma} {\bf Q}^\top
{\bf N}({\bf f} ){\bf V} 
 + {\bf V} ^\top
{\bf N}({\bf f} )^\top{\bf Q} {\bf \Sigma} )
) +
({\bf I} - {\bf V} {\bf V} ^\top)
{\bf N}({\bf f} )^\top{\bf Q} {\bf \Sigma} ^{-1}.
\end{align}
By setting ${\bf U} = {\bf Q}{\bf \Sigma}$ 
It can be verified that the pair
$({\bf U},{\bf V})$ satisfy the dynamically bi-orthogonal
equations of \cite{cheng2013dynamically}.
It should be noted that this is not the only 
parametrization of the fixed-rank
solution ${\bf f} = {\bf Q}{\bf \Sigma}{\bf V}^\top$.
Of particular interest is the closely related projection
method given by the DDO approximation 
%(both of which are equivalent to
%H-Tucker tangential dynamics in 2D)
\begin{align}
\frac{\text{d}}{\text{d} t}{\bf A} &= 
{\bf W}^\top {\bf N}({\bf f} )
 {\bf B}   ,
\\ 
\frac{\text{d}}{\text{d} t}{\bf W}&=
({\bf I} - {\bf W} {\bf W}^\top)
{\bf N}({\bf f} ){\bf B} {\bf A} ^{-1},
 \\
\frac{\text{d}}{\text{d} t}{\bf B}&=
({\bf I} - {\bf B} {\bf B} ^\top)
{\bf N}({\bf f} )^\top{\bf W} {{\bf A}^{-\top}}.
\end{align}
Which is equivalent to the SVD equations above
in the sense that
\begin{equation}
{\bf W}(t){\bf A}(t){\bf B}^{\top}(t)
=
{\bf f}(t)
=
{\bf Q}(t){\bf \Sigma}(t){\bf V}^{\top}(t)
\end{equation}
as long as the singular values are distinct and the equation
holds at $t=0$. A comparison of methods for fixed rank 
initial value problems is given in \cite{musharbash2015error}.

\section*{Declarations}
\noindent 
{\em Funding:} 
This research was supported by the U.S. Air 
Force Office of Scientific Research (AFOSR) 
grant FA9550-20-1-0174 and by the U.S. Army 
Research Office (ARO) grant W911NF-18-1-0309.

\vspace{0.3cm}
\noindent 
{\em Data availability statement:} 
The datasets generated during and/or analysed 
during the current study are available from the 
corresponding author on reasonable request.
%\newpage

\vspace{0.3cm}
\noindent 
{\em Conflicts of interest:} 
The authors declare that they have no known competing financial interests or personal relationships that could have appeared to influence the work reported in this paper.

\vspace{0.3cm}
\noindent 
{\em Code availability:} 
The code generated during the current study is 
available from the corresponding author on 
reasonable request.

\bibliography{main}
\bibliographystyle{spmpsci}

\end{document}